\documentclass[reqno]{amsart}

\usepackage[centering]{geometry}

\usepackage{mathtools}
\usepackage{mathrsfs}
\usepackage{amsmath}
\usepackage{bm}
\usepackage[dvipsnames]{xcolor}
\usepackage
[colorlinks=true,linkcolor=Maroon,citecolor=OliveGreen,backref]
{hyperref}

\usepackage[capitalize]{cleveref} %, autonum} \renewcommand{\qedhere}{\GenericError{}{qedhere plays up with autonum}{}{}}
\crefname{equation}{}{}

\usepackage[norefs, nocites]{refcheck}

\makeatletter
\newcommand{\refcheckize}[1]{%
	\expandafter\let\csname @@\string#1\endcsname#1%
	\expandafter\DeclareRobustCommand\csname relax\string#1\endcsname[1]{%
		\csname @@\string#1\endcsname{##1}\wrtusdrf{##1}}%
	\expandafter\let\expandafter#1\csname relax\string#1\endcsname
}
\makeatother

\newtheorem{theorem}{Theorem}[section]
\newtheorem{lemma}[theorem]{Lemma}
\newtheorem{proposition}[theorem]{Proposition}
\newtheorem{corollary}[theorem]{Corollary}
\newtheorem{conjecture}[theorem]{Conjecture}

\theoremstyle{definition}
\newtheorem{remark}[theorem]{Remark}

\newtheorem*{collection_1}{Collection $\ca H_1$}
\newtheorem*{collection_2}{Collection $\ca H_2$}

\Crefname{lemma}{Lemma}{Lemmas}
\Crefname{lemma}{Lemma}{Lemmas}

\crefname{proposition}{proposition}{propositions}
\Crefname{proposition}{Proposition}{Propositions}

\crefname{corollary}{corollary}{corollaries}
\Crefname{corollary}{Corollary}{Corollaries}

\usepackage{amssymb}

\newcommand{\eps}{\epsilon}
\newcommand{\ca}{\mathcal}

\newcommand{\sm}{\smallsetminus}

\def\R{\mathbb R}
\def\Q{\mathbb Q}
\def\End{\mathrm{End}}
\def\F{\mathbb F}

\def\GL{\mathrm{GL}}
\def\PGL{\mathrm{PGL}}
\def\AGL{\mathrm{AGL}}
\def\Gal{\mathrm{Gal}}
\def\SU{\mathrm{SU}}
\def\PSU{\mathrm{PSU}}
\def\GU{\mathrm{GU}}
\def\Sp{\mathrm{Sp}}

\def\PSp{\mathrm{PSp}}
\def\SL{\mathrm{SL}}
\def\PSL{\mathrm{PSL}}
\def\Or{\mathrm{O}}
\def\CO{\mathrm{CO}}
\def\SO{\mathrm{SO}}

\def\Om{\Omega}
\def\POm{\mathrm P \Omega}

\def\det{\operatorname{det}}

\def\min{\mathrm{min}}

\def\Out{\mathrm{Out}}
\def\Aut{\mathrm{Aut}}

\def\End{\mathrm{End}}
\def\fpr{\mathrm{fpr}}
\def\fp{\mathrm{fix}}
\def\l{\lambda}

\newcommand\floor[1]{\left\lfloor{#1}\right\rfloor}

\newcommand{\gen}[1]{\ensuremath{\langle #1\rangle}}

\numberwithin{table}{section}
\numberwithin{equation}{section}

\usepackage{todonotes}

\begin{document}
	
	\author{Daniele Garzoni}
	\address{Daniele Garzoni, Department of Mathematics, University of Southern California, Los Angeles, CA 90089-2532,
		USA}
	\email{garzoni@usc.edu}
	
	\author{Robert M. Guralnick}
	\address{Robert M. Guralnick, Department of Mathematics, University of Southern California, Los Angeles, CA 90089-2532,
		USA}
	\email{guralnic@usc.edu}
	
	\author{Martin W. Liebeck}
	\address{Martin W. Liebeck, Department of Mathematics, Imperial College, London SW7 2AZ, UK}
	\email{m.liebeck@imperial.ac.uk}
	
	\title[On a conjecture of Peter Neumann]{On a conjecture of Peter Neumann on fixed points in permutation groups}
	\maketitle

\begin{center} {\it Dedicated to the memory of Peter Neumann} \end{center}
    
	\begin{abstract}
		We prove a conjecture of Peter Neumann from 1966, predicting that every finite non-regular primitive permutation group of degree $n$ contains an element fixing at least one point and at most $n^{1/2}$ points. In fact, we prove a stronger version, where $n^{1/2}$ is replaced by $n^{1/3}$, and this is best possible. The case where $G$ is affine was proved by Guralnick and Malle;  in this paper we address the case where $G$ is non-affine. 
	\end{abstract}

	\section{Introduction}

The study of fixed points of elements in finite primitive permutation groups has a long history. An early theme was the {\it minimal degree}: for $G$ primitive of degree $n$, this is defined to be the smallest number of points moved by any non-identity element of $G$, denoted by $\mu(G)$. Nineteenth century results of Bochert and Jordan show that for $G \ne A_n,S_n$, the minimal degree $\mu(G)$ tends to infinity as $n\to \infty$, and further work of Jordan, Manning and others provides explicit lower bounds, culminating in Babai's bound $\mu(G) > \frac{1}{2}\sqrt{n}$ (see \cite{Babai} and \cite[Thms. 5.3A,5.4A]{DM}), which is not far from best possible. Babai's proof does not use the classification of finite simple groups (CFSG). Using CFSG, the result has been much extended: for example in \cite{GM} it has been shown that $\mu(G) \ge \frac{1}{2}n$, provided certain explicit families of primitive groups $G$ are excluded. To put it another way, in all primitive groups apart from these families, we have $\fp(g) \le \frac{1}{2}n$ for all $1\ne g \in G$ (where $\fp(g)$ is the number of fixed points of $g$). 

Given these bounds for fixed point numbers covering all non-identity elements, it is natural to ask whether there exist individual elements that fix many fewer points than the overall bound. 
This was the topic of Peter Neumann's 1966 DPhil thesis \cite{neumann1966thesis}. By a well-known lemma of Jordan, any transitive permutation group of degree $n \ge 2$ contains a derangement -- that is, an element $g$ with $\fp(g)=0$. Neumann investigated the existence of elements fixing few points (but at least one) in non-regular transitive groups, and proved that any such group $G$ contains an element $g$ with $1\le \fp(g)\le \frac 12 n$. This is sharp, as can be seen from the group $G = C_m \wr S_2$ in an imprimitive action of degree $n=2m$. For primitive groups, Neumann conjectured a much stronger result:
	
	\begin{conjecture} [P. Neumann, 1966]
		\label{conj:neumann}
		Let $G$ be a finite primitive non-regular permutation group of degree $n$. Then, there exists $g\in G$ with $1\le \fp(g)\le n^{1/2}$.
	\end{conjecture}
	
	Most of the work in the literature on this conjecture has concerned the case of primitive groups of affine type. For these we have $G = V \rtimes G_0 \le \AGL(V)$, where $V$ is a finite vector space, and $G_0$ is an irreducible subgroup of $\GL(V)$; the conjecture says that such a group $G_0$ should possess an element $g$ such that $\dim C_V(g) \le \frac{1}{2}\dim V$. For $G$ solvable, Neumann proved this, with the stronger upper bound $\frac{7}{18}\dim V$ instead of $\frac{1}{2}\dim V$; and he conjectured that the true bound in this case should be $\frac{1}{3}\dim V$, noting that equality is achieved by the group $G_0 = A_4 < \GL_3(p)$ for odd primes $p$. Of course Neumann did not have the CFSG available to use at the time. Using CFSG, Segal and Shalev \cite[Lemma 2.3]{SS} proved the original Neumann conjecture for affine groups; following a further improvement in \cite{IKMM}, the affine conjecture with the bound $\frac{1}{3}\dim V$ was finally established by Guralnick and Malle in \cite{guralnick2012malle}. Generalizing Neumann's example, one can see that the bound is sharp for the groups $G_0 = \SO_3(q) < \GL_3(q)$ for $q$ an odd prime power.
	
	In this paper, we complete the proof of the strong version of Neumann's conjecture, by addressing the case where $G$ is non-affine. Combining with affine case \cite{guralnick2012malle}, we prove the following theorem.
	
	\begin{theorem} 
		\label{t:main}
		Let $G$ be a finite primitive non-regular permutation group of degree $n$. Then there exists $g\in G$ with $1\le \fp(g)\le  n^{1/3}$.
	\end{theorem}

    If $G$ is not affine, we will in fact find $g\in G$ with $1\le \fp(g) <  n^{1/3}$; see Theorem \ref{t:main_simple}, Proposition \ref{twr} and the proofs of Propositions \ref{diag} and \ref{prod}.
        
         As we have remarked already, the $n^{1/3}$ bound is sharp for affine groups. For non-affine groups, the strict inequality is close to best possible. For example, if $S=\PSL_2(q)$ with $q$ even, and if we consider $G=S\times S$  acting on $S$ via $s^{(x_1,x_2)}=x_1^{-1}sx_2$, then the number of fixed points of an element of $G$ is either $0$ or at least $q-1=n^{1/3}(1-o(1))$.  A similar example is given by the action of $\PSL_2(q^2)$ on the cosets of $\PSL_2(q)$ for $q$ even.

	The following  is an immediate consequence of \Cref{t:main} and Frobenius density theorem. For completeness, we give a proof of the deduction in \Cref{sec:polynomials}. In the statement,  a \textit{minimal field extension} refers to an extension having no nontrivial subextensions.
	
	\begin{corollary}
		\label{cor:polynomials}
		Let $f(X)\in \mathbb Z[X]$ be of degree $n$ and irreducible in $\Q[X]$. Letting $\alpha \in \overline{\Q}$ be a root of $f(X)$, assume that $\Q(\alpha)/\Q$ is a minimal field extension and that it is not Galois. Then, there is a set $\ca P$ of primes of positive density such that for $p\in \ca P$, the reduction of $f(X)$ modulo $p$ has at least one root and at most $n^{1/3}$ roots in $\F_p$.
	\end{corollary}

	Let us briefly discuss the proof of the theorem and the layout of the paper. In Section \ref{reduction}, using the O'Nan-Scott theorem, the proof is reduced to the case where $G$ is almost simple. This reduction is not entirely straightforward, and in particular primitive groups of twisted wreath type require some effort (see \Cref{twr}). The rest of the paper covers the case where $G$ is almost simple. Let $H$ be a point-stabilizer, a maximal subgroup of $G$, and write $G/H$ for the set of right cosets of $H$ in $G$, and $\fp(g,G/H)$ for the number of fixed points of $g$ in the action on $G/H$. For $g \in H$ we have 
	\begin{equation}\label{fpteq}
		\fp(g,G/H) = |C_G(g)|.\frac{|g^G \cap H|}{|H|} \le |C_G(g)|
	\end{equation}
	(see Lemma \ref{l:fpr_enough}), and so the main aim is to find an element $g \in H$ for which $|C_G(g)|$ is small (ideally, less than $|G:H|^{1/3}$). This is achieved in Sections \ref{sec:alternating_groups} and \ref{sporadpf} for alternating and sporadic groups, respectively. Exceptional groups of Lie type are dealt with in Section \ref{exceppf}. The classical groups require the most work, and are handled in the last three sections \ref{class1}, \ref{sec:geometric_classes} and \ref{sec:class_S}. 
	
	As far as the methods of proof are concerned, the starting point is the substantial literature on the maximal subgroups of almost simple groups. For sporadic groups the information is complete (the maximal subgroups have been classified), and it is a fairly routine matter to pick an element $g$ such that the right hand side of (\ref{fpteq}) is less than $|G:H|^{1/3}$, as required. For exceptional groups of Lie type the maximal subgroups have been classified except when $G$ is of type $E_7$ or $E_8$, but it is much less routine to find suitable elements $g$. When $G$ is alternating or classical, there are well-known theorems (\cite{AS} and \cite{aschbacher1984maximal_subgroups}) that partition the maximal subgroups into classes $\ca C$ of known subgroups, and $\ca S$ of unknown subgroups; the class $\ca S$ consists of almost simple primitive subgroups in alternating groups, and almost simple irreducible subgroups in classical groups. While for alternating groups, the class $\ca S$ does not cause much difficulty (see Lemma \ref{asprim}), this is not the case for classical groups, and the longest and most difficult section of the paper (Section \ref{sec:class_S}) is devoted to the proof of Theorem \ref{t:main} for $G$ classical and $H$ in the class $\ca S$. We use a full range of methods and results on the structure, generation and modular representation theory of simple groups to achieve the proof in this case (see \Cref{subsec:strategy} for a brief outline of the argument).
	
	Let us offer a final word about our use of computation in the proofs, using the computer software GAP (\cite{GAP_character_table}). If a group $G$ and a maximal subgroup $H$ are available in GAP, and also the character table of $G$, then for $g\in G$ it is a routine matter to compute the number of fixed points $\fp(g,G/H)$, using the equation (\ref{fpteq}). We shall make frequent use of such computations for small simple groups $G$, and we usually suppress details, using phrases such as ``we check that $\fp(g,G/H) < |G:H|^{1/3}$ using GAP".

	\section{Preliminaries}
	
	We prove a few elementary results on fixed points, and conclude with two lemmas concerning simple groups and a summary of some notation used throughout the paper.
	Recall that for a group $G$ with a subgroup $H$, we write $G/H$ for the set of right cosets of $H$ in $G$, and $\fp(g,G/H)$ for the number of fixed points of an element $g$ in its action on $G/H$.
	
	\begin{lemma}
		\label{l:fpr_enough}
		Let $G$ be a finite group, let $H$ be a subgroup of $G$, and let $g \in H$. 
		\begin{itemize}
			\item[{\rm (i)}] Then
			\[
			\fp(g,G/H) = |C_G(g)|.\frac{|g^G \cap H|}{|H|} \le |C_G(g)|.
			\]
			\item[{\rm (ii)}] Let $g_1, \ldots, g_t$ be representatives of the $H$-classes in $g^G\cap H$. Then
			\[
			\fp(g,G/H) = \sum_{i=1}^t \frac{|C_G(g)|}{|C_H(g_i)|}.
			\]
		\end{itemize}
	\end{lemma}
	
	\begin{proof}
		(i)	It is well-known that 
		\[
		\fp(g,G/H)=\frac{|g^G\cap H|}{|g^G|}|G:H|.
		\]
		This can be seen by counting the set $X = \{(Hx,y) : x \in G,\,y\in g^G,\,Hxy=Hx\}$ in two ways; we see that $|X| = |g^G|.\fp(g,G/H) = |G:H|.|g^G\cap H|$. Part (i) follows.
		
		(ii) We have
		\[
		\fp(g,G/H)=\frac{|g^G\cap H|}{|g^G|}|G:H|=\sum_{i=1}^t \frac{|g_i^H||G|}{|g^G||H|} = \sum_{i=1}^t \frac{|C_G(g)|}{|C_H(g_i)|},
		\]
		as required.
	\end{proof}

	The following variant will sometimes be useful.
	
	\begin{lemma}
		\label{l:fixed_points_precise_normal_subgroup}
		Let $G$ be a finite group, let $H$ be a subgroup of $G$, let $H_0\triangleleft H$ and let $g\in H$. 
		\begin{itemize}
			\item[(i)] Let $g_1, \ldots, g_t$ be representatives of the $H_0$-classes in $g^G\cap H$. Then
			\[
			\fp(g,G/H) = \frac{1}{|H:H_0|}\sum_{i=1}^t \frac{|C_G(g)|}{|C_{H_0}(g_i)|}.
			\]
			\item[(ii)] If for each coset $C$ of $H_0$ in $H$, $g^G\cap C$ is either empty or equal to an $H_0$-class, and if $g\in g^G\cap H$ is chosen so that $|C_{H_0}(g)|$ is minimal, then
			\[
			\fp(g,G/H) \le   \frac{|C_G(g)|}{|C_{H_0}(g)|}.
			\]
		\end{itemize}
	\end{lemma}
	
	\begin{proof}
		Part (i) follows by a similar argument to the previous lemma. As for (ii), the assumption that $g^G\cap C$ is either empty or equal to an $H_0$-class implies $t\le |H:H_0|$, and so the bound follows immediately from the minimality of $|C_{H_0}(g)|$.
	\end{proof}
	
	We shall need the following result on centralizer orders in simple groups.
	
	\begin{lemma}
		\label{l:centralizer_simplegroup_13}
		Let $S$ be finite non-abelian simple group. Then $S$ contains an element $g$ such that $|C_S(g)|< |S|^{1/3}$.
	\end{lemma}
	
	\begin{proof}
		For $S = A_n$ with $n\ge 6$, we take $g$ to be either an $n$-cycle or an $(n-1)$-cycle; and for $S=A_5$ we take $g$ to be a $3$-cycle. For $S$ sporadic we use \cite{atlas}.
		
		Now suppose $S \in {\rm Lie}(p)$, and let $g\in S$ be a regular unipotent element. The order of $C_S(g)$ can be read off from results in \cite{liebeck_seitz_2012unipotent}. For $S$ a classical group, \cite[Chapters 3,4]{liebeck_seitz_2012unipotent} gives $|C_S(g)| \le (2,p)q^r$, where $r$ is the untwisted Lie rank of $S$; and for $S$ of exceptional Lie type, \cite[Tables 22.2.1-6]{liebeck_seitz_2012unipotent} gives $|C_S(g)| \le (60,p^2)q^r$. It follows that $|C_S(g)|< |S|^{1/3}$ in all cases except for $S = \PSL_2(q)$. Finally, for $S=\PSL_2(q)$ with $q\ge 7$, we take $g$ of order $(q-1)/(2,q-1)$. 
	\end{proof}
	
	The final lemma of this section concerns the orders of simple groups. In the proof, and throughout the rest of the paper, we make use of Zsigmondy's theorem on primitive prime divisors \cite{Zs}: this states that if $q,n$ are integers with $q \ge 2$, $n\ge 3$ and $(q,n) \ne (2,6)$, then there is a prime number that divides $q^n-1$ and does not divide $q^i-1$ for $1 \le i<n$. Such a prime is called a {\it primitive prime divisor} of $q^n-1$, abbreviated as {\it ppd}, and denoted by $q_n$. Observe that $q_n \equiv 1 \hbox{ mod }n$.
	
	\begin{lemma}
		\label{l:order_divisible_by_7}
		If $G$ is a finite simple $\{2,3,5\}$-group, then $G$ is $A_5$, $A_6$ or $\PSU_4(2)$.
	\end{lemma}
	
	\begin{proof}
		If $G$ is alternating this is clear, and we can verify the result for $G$ sporadic by checking the orders of these groups (for example in \cite{atlas}).  
		Assume now that $G$ is of Lie type over $\F_q$. If $G$ is not isomorphic to one of
		\[
		\PSL_2(q), \,^2\!B_2(q),\, \PSp_4(q)',\,\PSU_4(2),\, \PSp_6(2),\,\Omega^+_8(2)
		\]
		then by considering at the order formula for $|G|$, we see that $|G|$ is divisible by a primitive prime divisor (ppd) of $q^i-1$ for some $i\ge 3$ with $i\ne 4$. Such a ppd is congruent to $1\pmod i$, and hence is at least $7$. 
		
		Let us address the remaining groups in the above list. The case $\PSU_4(2)$ is in the statement, as are $\PSL_2(q), q=4,5,9$ and $\PSp_4(q)', q=2,3$. The groups $\PSp_6(2),\,\Omega^+_8(2)$, have order divisible by 7. For $G = \,^2\!B_2(q)$ we have $q = 2^{2a+1}\ge 8$ and $|G|$ is divisible by a ppd of $2^{2a+1}-1$. 
		
		Assume next $G=\PSp_4(q)$ with $q \ge 4$.  If $q=p^a$ with $a>1$, then $|G|$ is divisible by a ppd of $p^{4a}-1$; if $q=p > 5$, observe that $p$ divides $|G|$; and $|\PSp_4(5)|$ is divisible by 13. 
		
		Assume finally $G=\PSL_2(q)$ with $q=p^a$. We can take $p = 2,3$ or 5 and $a \ge 3,3$ or 2 respectively. Then $|G|$ is divisible by a ppd of $p^{2a}-1$, or by 7 if $(p,a) = (2,3)$, or by 13 if $(p,a) = (5,2)$. This completes the proof.
	\end{proof}

\subsection*{Notation}  For a prime $p$, we denote by ${\rm Lie}(p)$ the set of simple groups of Lie type over a field of characteristic $p$; and ${\rm Lie}(p')$ is the set $\bigcup_{r\ne p}{\rm Lie}(r)$.
	
	For $\eps = \pm$, we use $\PSL_n^\eps(q)$ to denote $\PSL_n(q)$ if $\eps = +$, and $\PSU_n(q)$ if $\eps = -$. Similarly, $E_6^\eps(q)$ is $E_6(q)$ for $\eps = +$, and $^2\!E_6(q)$ for $\eps = -$. 
	
	Finally, we remind the reader of our notation for primitive prime divisors (ppds):  for $q \ge 2$, $n\ge 3$ and $(q,n) \ne (2,6)$, we denote by $q_n$ a ppd of $q^n-1$.

	%%%%%%%%%%%%

	\section{Reduction to simple groups}\label{reduction}
	
	The heart of the proof of Theorem \ref{t:main} is the case of almost simple groups, for which we shall prove the following slightly stronger result. For $G$ a non-abelian finite simple group, denote by 
	\begin{equation}
		\label{eq:ca A}
		\ca A=\ca A(G)
	\end{equation}
	the set of subgroups of $G$ that extend to a maximal subgroup of some almost simple group with socle $G$.
	
		\begin{theorem} \label{t:main_simple}
		Let $G$ be a non-abelian finite simple group and let $M\in \ca A$. Then, there exists $g\in M$ such that $\fp(g,G/M)<|G:M|^{1/3}$.
	\end{theorem}
    
	 In this section we deduce Theorem \ref{t:main} from this result. After this, the rest of the paper is devoted to proving Theorem \ref{t:main_simple}.

Assume then that Theorem \ref{t:main_simple} holds, and let $G$ be a primitive non-regular permutation group of degree $n$ on a set $\Omega$. According to the O'Nan-Scott theorem (see for example \cite{LPS}), $G$ is of one of the following types:
\begin{itemize}
    \item[(1)] affine,
    \item[(2)] almost simple,
    \item[(3)] simple diagonal type,
    \item[(4)] product action,
    \item[(5)] twisted wreath.
\end{itemize}

Theorem \ref{t:main} was proved for affine groups in \cite{guralnick2012malle}, and follows from Theorem \ref{t:main_simple} for almost simple groups. The remaining types (3),(4) and (5) are handled in the next three propositions.

\begin{proposition}\label{diag} Theorem \ref{t:main} holds for $G$ of simple diagonal type.
\end{proposition}

\begin{proof}
		Assume $G$ is of simple diagonal type, and let $H$ be a point-stabilizer. Then $G$ has socle $N = S^{r}$, where $S$ is non-abelian simple, $r\ge 2$, $n=|S|^{r-1}$ and $H\cap N$ is a diagonal subgroup of $S^r$. Then for any $x\in S$, there is an element 
        $g\in N\cap H$ fixing exactly $|C_S(x)|^{r-1}$ points.  
		By \Cref{l:centralizer_simplegroup_13}, we can find $x\in S$ such that $|C_S(x)|<|S|^{1/3}$, and the corresponding element $g$ satisfies $1\le \fp(g)< n^{1/3}$.
	\end{proof}

\begin{proposition}\label{prod} Theorem \ref{t:main} holds for $G$ in product action.
\end{proposition}

\begin{proof}
In this case we have $G\le A\wr S_t$ acting on $\Omega = \Delta^t$, where $A$ is an almost simple or simple diagonal type primitive group on $\Delta$, and ${\rm Soc}(G) = {\rm Soc}(A)^t$. Set $m=|\Delta|$. By the proof of Proposition \ref{diag} for the simple diagonal case, and by Theorem \ref{t:main_simple} for the almost simple case, there exists $x\in {\rm Soc}(A)$ such that $1\le \fp(x,\Delta)<m^{1/3}$. Letting $g=(x, \ldots, x)\in {\rm Soc}(G)$, we then have $1\le \fp(g,\Omega)<m^{t/3} = n^{1/3}$, as required.		
	\end{proof}
        
The twisted wreath case (5) requires more effort, and we need two preliminary lemmas for this.

%%%%%%%%%

\begin{lemma} \label{lemtwrs}  
Let $S$ be a finite non-abelian simple group.  There exists an element $x \in S$ of prime order $p \ge 5$ such that for any $S$-set of size $n$ on which $S$ has no fixed points,
the number of orbits of $\langle x \rangle$ is at most $n/3$.  Moreover, we can choose $x \in S$ of prime power order such that the number of orbits of $\langle x \rangle$ is less than $n/3$ unless either 
$S=A_5$ and every orbit of has size $6$ or $12$, or $S=A_6$ and every orbit has size $6$.
\end{lemma}

\begin{proof}  Note that the result reduces to the case of transitive $S$-sets, and then to the case of primitive $S$-sets. So consider a primitive $S$-set $\Omega$ of size $n$, and let $H$ be a point-stabilizer.
Let $p\ge 5$ be a prime dividing $|S|$. By \cite[Thm. 1]{burness2022guralnick}, for any element $x \in S$ of order $p$, one of the following holds:
\begin{itemize}
\item[(i)] $\fp(x) \le n/(p+1)$;
\item[(ii)] $S=A_m$ and the action $(S,\Om)$ is on $\ell$-element subsets of $[m]$ for some $\ell < m/2$;
\item[(iii)] $S$ is a classical group, the action $(S,\Om)$ is a subspace action, and $(S,H,x)$  are as in \cite[Table 6]{burness2022guralnick}. 
\end{itemize}

In case (i),  $\langle x \rangle$ has at most  $2n/(p+1)$ orbits, and since $p \ge 5$, this is at most $n/3$, giving the first assertion of the lemma.  
Moreover, we get the strict inequality
as we long as we can choose $p > 5$, and so by Lemma \ref{l:order_divisible_by_7} we only need to consider $S = A_5$, $A_6$ or $U_4(2)$.     In the last case, strict inequality holds for an element of order $5$  and in the first two cases, the strict inequality holds for an element of order $5$ apart from the case where $n=6$ and  $S=A_5$ or $A_6$; for $S=A_5$, equality also holds for the transitive action of degree 12.

Consider case (iii). Inspection shows that in all but one case, we can choose a different element $x \in S$ of order $p$ that does not lie in \cite[Table 6]{burness2022guralnick}, hence satisfies (i). The exception is $S=\PSL_2(p+1)$ with $p$ a Mersenne prime. In this case, for all actions $(S,\Om)$ we have $\fp(x,\Om) \le 2n/(p+2)$ (with equality for the action on the cosets of a Borel subgroup); for $p \ge 7$, this is at most $2/9$ and again, $\langle x \rangle$ has at most $n/3$ orbits.  If $p > 7$, strict inequality holds and if $p=7$ we choose an element of order $9$ and strict inequality  holds (for all actions of $S$).  

Finally consider case (ii).  For $m=5$ (resp. 6, 7, 8, 9, 10), we choose an element $x$ of order 5 (resp. 5, 7, 7, 9, 9) and check that strict inequality holds.
So assume that $m \ge 11$.  We induct on $m$. Let $p$ be the largest prime at most $m$; then 
$p > \frac{2}{3}m + 1$ (see for example \cite{el_bach}).  
If $p=m$, then an element of order $p$ has no fixed points on $k$-sets for any $k$ and the result is clear.   If $p < m$, then consider $A_{m-1}$ acting.
If $A_{m-1}$ has no fixed points the result holds by induction.  If $A_{m-1}$ does have a fixed point, then the action is the usual action on $n=m$ points and 
$x$ has $1 + (n-p) < n/3$ orbits and the result follows. 
\end{proof}  

\begin{lemma} \label{lemtwr}   Let $S$ be a nonabelian simple group and let $T=S^k$ for some positive integer $k$.   Let $T$ act on a set $\Om$ of size $n$ without fixed points.  
\begin{itemize}
\item[(i)] If $S \ne A_5$ or $A_6$, there exists an element $g \in T$ such that the number of orbits of $\langle g \rangle$ is less than $n/3$.
\item[(ii)] If $S=A_5$ or $A_6$, then there exists an element $g \in T$ of order $5$ such that 
one of the following holds:
\begin{itemize}
\item[{\rm (a)}] the number of orbits of $\langle g \rangle$ is less than $n/3$;
\item[{\rm (b)}] each $T$-orbit has size $6$ or $12$ for $n=5$, or $6$ for $n=6$; in particular, the kernel of $T$ on each orbit contains all but one component of $T$.
\end{itemize}
\end{itemize}
\end{lemma} 

\begin{proof}  We take the element $g = (x, \ldots, x)$ where $x\in S$ is the element given 
in the previous lemma.   It suffices to consider the case that $T$ acts transitively (and faithfully).   Let $H$ be a point stabilizer.  

Suppose first that the projection of 
$H$ into at least two of the components of $T$ is non-surjective; say these projections are $M_1<S$ and $M_2<S$. We can replace $H$ by its overgroup $M_1\times M_2 \times S^{k-2}$, so that the action is that of $S^2$ on the cosets of $M_1\times M_2$. For $i=1,2$ write $n_i = |S:M_i|$. Let $f_i$ be proportion of fixed points of $x$ on $S/M_i$. If $f_1f_2=0$ then $g$ is a derangement of odd prime power order at least $5$, so the number of orbits of $\gen g$ is less than $n/3$. Assume then $f_1f_2\neq 0$, so $f_1f_2<f_1$. 
If $|x|$ is prime, it follows at once that the number of orbits of $\gen g$ divided by $n$ is strictly less than the number of orbits of $\gen x$ on $S/M_1$ divided by $n_1$, and we are done by \Cref{lemtwrs}. Assume now $|x|$ is not prime; by inspecting the proof  of \Cref{lemtwrs}, we have $|x|=9$ (and $S=A_9,A_{10}, \PSL_2(8)$) and the number of orbits of $\gen x$ on $S/M_1$ is strictly less than $n_1/3$. In particular, we see that the number of orbits of $\gen g$ is strictly less than $(n_1/3)n_2=n/3$, and we are done.

If the projection of $H$ is not surjective for exactly one component of $T$, the result follows from Lemma \ref{lemtwrs} (and this is where equality can  hold). 

Finally, suppose all projections are surjective.  We show strict inequality in this case.  Replacing $H$ by a maximal subgroup containing it, we are reduced to the case $k=2$ and $H \cong S$ is a diagonal subgroup.  In this case, we have $\fp(g,\Om) = |C_S(x)|$. We claim that $|C_S(x)| \le |S|/12 = n/12$. This is clear if $S$ has no subgroup of index at most $12$. Otherwise, $S$ embeds in $S_{12}$ and one can compute the list of such simple groups and check the claim directly. (Note that equality can occur for $x$ of order $5$ in $A_5$).   Exclude the cases where $x$ is an element
of order $9$  (i.e.  $S = A_9, A_{10}$ or $\PSL_2(8)$ -- see the proof of \Cref{lemtwrs}).  
Thus $\fp(g,\Om) \le n/12$, and so the number of orbits of $\langle g \rangle$ is at most 
$n/12 + 11n/12p \le (4/15)n$ and the result follows.   In the three cases remaining, one computes directly that the result holds.  
\end{proof}

\begin{proposition}\label{twr} Let $G$ be a primitive group of twisted wreath type on a set of size $n$.  Then there exists an element $g \in G$ such that $1 \le \fp(g) < n^{1/3}$. In particular, Theorem \ref{t:main} holds for $G$ of twisted wreath type.
\end{proposition}

\begin{proof}
Let $G$ be a primitive group of twisted wreath type on a set $\Om$ of degree $n$. 
We use information on the structure of these groups from \cite[Section 7]{AS} (see also \cite{Baddeley}). The socle $E$ of $G$ is a regular normal subgroup and is the unique minimal normal subgroup of $G$. 
Moreover,  $E \cong D^t$, where $D$ is a non-abelian simple group and $t \ge 6$.   If $H$ is a point-stabilizer, then the following hold \cite[Theorem C, (7.1)]{AS}: 
\begin{itemize}
\item[(i)] $G=EH$, $H \cap E=1$ and $H$ acts transitively and faithfully on the set $\Delta$ of components of $E$;
%%  This is because  the centralizer of $E$ is trivial and by AS, any minimal normal subgroup of H fixing all components 
%%   must induce inner automorphisms on $E$ and so is contained in $EC(E)=E$.  
\item[(ii)]   $F^*(H)$ is a direct product of isomorphic nonabelian simple groups.
\end{itemize}
We have $n=[G:H]=|E|$.  Identifying $G/H$ with $E$, we need to produce
an element $g \in H$ with $|C_E(g)| < n^{1/3}$.

 Let $A$ be a minimal normal subgroup of $H$ with $A=S^e$ for some nonabelian simple group $S$ and positive integer $e$.  
 No orbit of $A$ on $\Delta$ is trivial and so by Lemma \ref{lemtwr}, there
exists $g \in A$ (namely $g=(x, \ldots, x)$ with $x$ as in Lemma \ref{lemtwrs})  such that the number of orbits of $\langle g \rangle$ on $\Delta$ is at most $t/3$. Then $|C_E(g)| \le |E|^{1/3}$.   Moreover, unless $S \cong A_5$ or $A_6$, the inequality is strict.

It remains to consider the cases where $S=A_5$ or $A_6$ and equality holds.  Let $g$ be as above. Every $A$-orbit has size $6$, or possibly $12$ for $S=A_5$.
We see that $\Delta$ is the disjoint union of the nontrivial orbits for each of the components of $A$.   Since $A$ has no fixed points, it follows that $A$ is the unique minimal normal subgroup of $H$, and so $H$ acts faithfully on $A$ by conjugation.  

Let $D$ be a component of $E$. In this paragraph we will show that either the statement of the proposition holds or $N_A(D)=C_A(D)$. By \cite[Theorem C(1)]{AS}, $N_H(D)/C_H(D)$ is isomorphic to a subgroup of $\Aut(D)$ containing all inner automorphisms. Since $N_A(D)/C_A(D)$
is a normal subgroup of $N_H(D)/C_H(D)$, it follows that if $N_A(D)\neq C_A(D)$ then $N_A(D)/C_A(D)$ contains a subgroup inducing the full group of inner automorphisms of $D$. Assume this is the case. 
 By the choice of $g$, we have $|C_E(g)| = |E|^{1/3}$ and $\langle g \rangle$ has precisely $t/3$  orbits on the components of $E$. In particular, $g$ centralizes each component it normalizes. Let $S_1$ be the unique component of $A$ not  normalizing $D$. Then $N_A(D) = M \times C$, where $M=N_{S_1}(D)$ and $C$ is the product of all components of $A$ different from $S_1$. 
Note that if $S=A_6$ then $M\cong A_5$, and if $S=A_5$ then $M\cong D_{10}$ or $C_5$. We claim that the action of $N_A(D)$ on $D$ factors through one component of $N_A(D)$. If this is the case, then a diagonal element $g$ of $A$ of order $5$ acts nontrivially on $D$, against our assumption. If $M\cong A_5$ then $N_A(D)$ is perfect, so $N_A(D)/C_A(D)\cong D$, which implies the claim (since a normal subgroup of $N_A(D)$ is a product of components). If $M\cong D_{10}$ or $C_5$, then the image of $M$ in $\Aut(D)$ is a solvable subgroup normalized by $D$. Such image must be trivial, that is, $M$ centralizes $D$. In particular, also in this case $N_A(D)/C_A(D)\cong D$ and the claim holds.

Assume therefore that $N_A(D)=C_A(D)$.
 It follows that a section of $N_H(D)/C_A(D)\cong N_H(D)A/A$ induces all inner automorphisms on $D$, and in particular 
$H/A$ is  not solvable.  Thus, we can choose an $\ell$-element $w \in H\sm A$, where
 $\ell$ is an odd prime and $\ell \ne 5$.

Since the outer automorphism group of $S$ is a $2$-group, by viewing $w$ as an element of $\Aut(A)\cong \Aut(S)\wr S_e$ we see that there exists $a\in A$ such that $wa$ centralizes a diagonal element $g'$ of $A$ of order $5$, and such that $wa$ is an $\ell$-element. Now replace $w$ by $wa$ and $g$ by $g'$, and look at $z:=wg$, so $|C_E(z)|\le |C_E(g)|\le |E|^{1/3}$. It remains to prove that the inclusion $C_E(z)< C_E(g)$ is strict. The inclusion is strict if $z$ does not have the same orbits as $g$ on components of $E$, so assume that the orbits of $z$ and $g$ are the same. This implies that $w$ normalizes each component of $E$, and so $w$ does not centralize some component of $E$, say $L$. If $g$ normalizes $L$, then by the above argument $g$ centralizes $L$, and so $C_E(z)< C_E(g)$. If $g$ does not normalize $L$, the $g$-orbit of $L$ (in the action on $\Delta$) has size $5$, so $B:=\prod_i L^{g^i}\cong L^5$. We have that $C_B(g)$ is a diagonal subgroup isomorphic to $L$. Then $w$ does not centralize $C_B(g)$, so  $C_E(z)< C_E(g)$ and this concludes the proof. 
\end{proof}

	%%%%%%%%%%%%

	\section{Alternating groups} \label{sec:alternating_groups}
	
	The rest of the paper is devoted to the proof of Theorem \ref{t:main_simple}. Recall that $G$ is a non-abelian simple group and $\ca A = \ca A(G)$ is the class of subgroups of $G$ defined in (\ref{eq:ca A}). In this section we handle the case where $G$ is an alternating group.
	
	The cases where $G = A_5$ or $A_6$ are easily handled by computations using the information on these groups in \cite{atlas}. So assume $G=A_n$ with $n\ge 7$. Denote by $[n]$ the set $\{1,\ldots, n\}$. According to the theorem in \cite[Appendix]{AS}, the subgroups $M\in \ca A$ can be divided into six classes, as follows:
	\begin{itemize}
		\item[(i)] intransitive subgroups: here $M = (S_k \times S_{n-k})\cap G$, where $1 \le k < n/2$;
		\item[(ii)] transitive and imprimitive subgroups: here $M = (S_k \wr S_t) \cap G$, where $n = kt$ and $1<t<n$;
		\item[(iii)] affine subgroups: $M = \AGL_d(p) \cap G$, where $n=p^d$, $p$ prime;
		\item[(iv)] product action subgroups: $M = (S_k \wr S_t)\cap G$, where $n=k^t$ with $k \ge 5$, $t\ge 2$;
		\item[(v)] diagonal action subgroups: $M = (S^t.(\Out(S)\times S_t))\cap G$, where $S$ is non-abelian simple, $t\ge 2$ and $n = |S|^{t-1}$;
		\item[(vi)] almost simple primitive subgroups $M$.
	\end{itemize}
	
	\begin{lemma}
		\Cref{t:main_simple} holds if $M$ is intransitive.
	\end{lemma}
	
	\begin{proof}
		Let $M = (S_k \times S_{n-k})\cap G$, where $1 \le k < n/2$. If $n$ is even, we choose $g\in M$ with cycle type $(k,n-k)$; if $n$ is odd and $k\neq 1$, we choose
		$g\in M$ with cycle type $(1,k-1,n-k)$; and if $n$ is odd and $k= 1$, we choose
		$g\in M$ with cycle type $(1,2,n-3)$. In all cases $\fp(g,G/M)=1$ and the conclusion holds.
	\end{proof}
	
	\begin{lemma}
		\Cref{t:main_simple} holds if $M$ is transitive and imprimitive.
	\end{lemma}
	
	\begin{proof}
		Assume $M = (S_k \wr S_t) \cap G$, where $n = kt$ and $1<t<n$. If $n$ is odd, we choose $g$ an $n$-cycle, and then $\fp(g,G/M)=1$ (since $g$ must permute the blocks transitively and has a unique cycle). If $n$ is even and $k\neq 2$, we choose $g\in M$ with cycle type $(1,1,k-2,n-k)$, and we have $\fp(g,G/M)=1$: indeed, since $k-2\neq n-k$, the two fixed points must necessarily belong to the same block of any $k^t$-partition fixed by $g$, and then the $(k-2)$-cycle must preserve that block, and the other blocks are determined as $g$ acts as a cycle on their points. Finally, if $n$ is even and $k=2$, we choose $g\in M$ with cycle type $(n-2,2)$. Since $t>2$, the $2$-cycle must necessarily stabilize a block, and the remaining blocks are determined since $g$ acts as a cycle on their points, so $\fp(g,G/M)=1$.
	\end{proof}
	
	\begin{lemma}
		\Cref{t:main_simple} holds if $M$ is an affine subgroup.
	\end{lemma}
	
	\begin{proof}
		Let $M = \AGL_d(p) \cap G$, where $n=p^d$, $p$ prime. Suppose $p=2$, and choose 
		$g\in \GL_d(2)<M$ a Singer cycle, which acts $[n]$ as an $(n-1)$-cycle. The number of $M$-classes contained in $g^G\cap M$ is half the number of classes of Singer cycles in $\GL_d(2)$, that is, $\phi(2^d-1)/(2d)$, where $\phi$ is Euler's totient function. (We have half the classes since $M$ contains representatives for both $G$-classes of $(n-1)$-cycles.)
		Moreover, $C_G(g)=C_M(g)=\gen g$, and so by \Cref{l:fpr_enough}
		\[
		\fp(g,G/M) = \frac{\phi(2^d-1)}{2d} < \left(\frac{(2^d)!}{2^d |\GL_d(2)|}\right)^{1/3} = |G:M|^{1/3}.
		\] 
		Assume now that $p>2$, and recall that $n\ge 7$. 
		Let $g\in \GL_d(p)$ be the square of a Singer cycle, so $g$ acts on $n$ points as the product of two disjoint $(n-1)/2$-cycles. Then the number of $M$-classes contained in $g^G\cap M$ is equal to the number of classes of $\GL_d(p)$ containing squares of Singer cycles, which is $\phi((p^d-1)/2)/d$. Moreover, for each $x\in g^G\cap M$ we have $|C_G(x)|=(n-1)^2/4$ and $C_M(x)=\gen x$, so by \Cref{l:fpr_enough}
		\[
		\fp(g,G/M)=\frac{\phi((p^d-1)/2)\cdot(p^d-1)}{2d} < \left(\frac{(p^d)!}{p^d |\GL_d(p)|}\right)^{1/3} = |G:M|^{1/3}.
		\]
		This concludes the proof.
	\end{proof}

	\begin{lemma}
		\Cref{t:main_simple} holds if $M$ is a product action subgroup.
	\end{lemma}
	
	\begin{proof}
		In this case we have $n=k^t$ where $k\ge 5$, $t>1$, and $M=(S_k\wr S_t)\cap G$. 
		Let $\delta=0$ (resp. 1) if $k$ is odd (resp. even), let $x\in A_k$ be a $(k-\delta)$-cycle, and let $g=(x, x,\ldots, x)\in A_k^t<M$. Then $x$ fixes $\delta$ points on $[n]$ and has all other cycles of length $k-\delta$. If $k$ is odd, then $|C_G(g)|= (n/k)!k^{n/k}/2=(k^{t-1})!k^{k^{t-1}}/2$, and one checks that $|M||C_G(g)|^3 < |G|$.  
		The case $k$ even is similar.
	\end{proof}
	
	\begin{lemma}
		\Cref{t:main_simple} holds if $M$ is a diagonal action subgroup.
	\end{lemma}
	
	\begin{proof}
		In this case we have $M=(S^t.(\Out(S)\times S_t))\cap G$, where $S$ is a non-abelian simple group, $t\ge 2$, and $n=|S|^{t-1}$. If $g=(g_1, \ldots, g_t)\in S^t$, then the action of $g$ is given by the formula 
		\[
		(x_1, \ldots, x_{t-1})^g = (g_t^{-1}x_1g_1, \ldots, g_t^{-1}x_{t-1}g_{t-1})
		\]
		for $x_i \in S$.
		In particular, the subgroup of $S^t$ given by $g_t=1$ acts regularly on  $[n]$. For $i<t$, choose $g_i\in S$ of order $r\ge 5$ and put $g=(g_1, \ldots, g_{t-1},1)$ , so $|C_G(g)|= r^{n/r}(n/r)!/2\le n^{n/r}/2\le n^{n/5}/2$. We then check that $|M||C_G(g)|^3 < |G|$.
	\end{proof}

	\begin{lemma}\label{asprim}
		\Cref{t:main_simple} holds if $M$ is an almost simple primitive subgroup.
	\end{lemma}
	
	\begin{proof}
		The almost simple primitive groups of degree $n\le 26$ are all in the GAP library, and we can check these cases computationally (for example, for $n\ge 10$ we can find $g\in M$ with $|C_G(g)|<|G:M|^{1/3}$). 
		
		Assume then $n\ge 27$, and let $S$ be the socle of $M$. 
		
		If $S$ is one of $A_5$, $A_6$, or $\PSU_4(2)$, we again use GAP; note that in these cases, since $n\ge 27$ we have $n\in \{36,40,45\}$, and as above we find $g\in M$ with $|C_G(g)|<|G:M|^{1/3}$. 
		
		Assume that $S$ is not one of $A_5$, $A_6$, or $\PSU_4(2)$, and assume furthermore $S$ is not $A_m$ with $n=\binom{m}{\ell}$ in the action on $\ell$-subsets; we will consider this case at the end of the proof.
		
		By \Cref{l:order_divisible_by_7}, there exists $g\in S$ of prime order $r\ge 7$. In some cases, we make a specific choice, as follows. If $S=\PSL_n(q)$ with $n\ge 3$, take $g$ semisimple with $r$ not dividing $q-1$, and if $S=\PSp_n(q)$ with $n\ge 6$, take $g$ semisimple (see the proof of Lemma \ref{l:order_divisible_by_7}); 
		if $S=\PSp_4(q)$ with $q\ge 7$ prime, take $g$ regular unipotent. We then apply  \cite[Theorem 1]{burness2022guralnick}. The specific choices make sure we are not in one of the exceptions in \cite[Table 6]{burness2022guralnick}, and so we deduce that $g$ fixes at most $n/(r+1)$ points on $[n]$. Letting $c = \fp(g,[n])$, we have
		\[
		|C_{S_n}(g)| = c! ((n-c)/r)!r^{(n-c)/r} \le c^c(n-c)^{(n-c)/r}=:f(c),
		\]
		with the convention $0^0=1$.
		Now looking at the derivative of $\ln(f(c))$
		we see that, for fixed $r$, $f(c)$ is decreasing up to a certain value, and then it is increasing. In particular, the maximum $m$ of $f(c)$ is attained either when $c$ is as small as possible, or when $c$ is as large as possible. Assume first $c$ is as large as possible; since $c\le n/8$ we have
		\[
		m\le f(n/8) = (n/8)^{n/8} (7n/8)^{7n/(8r)}\le n^{n/4}(7/64)^{n/8},
		\]
		since $r\ge 7$. 
		Next, by \cite{robbins1955remark} we have
		\[
		n! \ge \sqrt{2\pi}n^{n+1/2}e^{-n},
		\]
		and for $n\ge 25$ we have $|M|<n^{1+\log_2(n)}$ by \cite[Theorem 1.1]{maroti2002orders}. Therefore
		\[
		|G:M|> \sqrt{\pi/2}n^{n+1/2}e^{-n}n^{-1-\log_2(n)}.
		\]
		Now we have
		\begin{align*}
			&n^{3n/4}(7/64)^{3n/8} < \sqrt{\pi/2}n^{n+1/2}e^{-n}n^{-1-\log_2(n)} \\
			&\quad \iff  \sqrt{\pi/2}\cdot  n^{n/4} > (7e^{8/3}/64)^{3n/8} n^{1/2+\log_2(n)}\\
			&\quad\quad\Longleftarrow 1.2\cdot n^{n/4} > (1.58)^{3n/8} n^{1/2+\log_2(n)},
		\end{align*}
		which holds since we assumed $n\ge 27$. Assume now the maximum $m$ of $f(c)$ is attained when $c$ is as small as possible. Since $c\ge 0$, we have 
		\[
		m\le f(0) = n^{n/r}\le n^{n/7},  
		\]
		since $r\ge 7$. But for $n\ge 14$ we have $n^{n/7}<n^{n/4}(7/64)^{n/8}$, hence the conclusion holds by the work done already.
		
		The only remaining case is $n=\binom{m}{\ell}$ and $S=A_m$ acting on $\ell$-subsets with $1<\ell\le m/2$. Then choose $g\in S$ an $m$-cycle or an $(m-1)$-cycle, so $g$ fixes no point on $[n]$, and we conclude with a similar calculation as above.
	\end{proof}
	
	%%%%%%%%%%%%%%%
	
	\section{Sporadic groups}\label{sporadpf}
	
	In this section we prove Theorem \ref{t:main_simple} for primitive groups with socle a sporadic simple group. Recall the definition of $\ca A(G)$ in (\ref{eq:ca A}).
	
	\begin{theorem}\label{spor} Let $G$ be a sporadic simple group and let $M\in \ca A(G)$. Then there exists $g\in M$ such that $\fp(g,G/M)<|G:M|^{1/3}$.
	\end{theorem}
	
	The proof consists largely of routine inspection of tables of maximal subgroups of sporadic groups, found in \cite{atlas} together with a few other references for some of the larger groups. 
	Let $G$ be as in Theorem \ref{spor}, and let $M \in \ca A(G)$. 
	
	In the proof, for each such subgroup $M$ of $G$, we shall specify an element $g \in M$ by its class name in \cite{atlas}. For some subgroups, the permutation character $1_M^G$ is given in \cite{atlas}, so we can directly compute the value of $\fp(g,G/M)$; for the other cases, we use the following consequence of Lemma \ref{l:fpr_enough}:
	\begin{equation}\label{fpbd}
		\fp(g,G/M) = \frac{|C_G(g)|\cdot |g^G\cap M|}{|M|} \le \frac{|C_G(g)|\cdot i_r(M)}{|M|},
	\end{equation}
	where $r = |g|$ and $i_r(M)$ denotes the number of elements of order $r$ in $M$.
	
	\subsection{$G=M_{11}$ } In this case the permutation character $1_M^G$ is given in \cite{atlas} for all maximal subgroups, and our choices for $g$ are as follows.
	\[
	\begin{array}{l|lllll}
		M & M_{10} & L_2(11) & M_9.2 & S_5 & M_8.S_3 \\
		\hline
		g & 5A & 11A & 3A & 5A & 4A \\
		\hline
		\fp(g) & 1&1&1&1&1
	\end{array}
	\]
	
	\subsection{$G=M_{12}$ } Our choices for $g$ are as follows; the permutation character $1_M^G$ is given in \cite{atlas} for all but the last four maximal subgroups, in which cases (\ref{fpbd}) is used.
	\[
	\begin{array}{l|lllllllllll}
		M & M_{11} & M_{10}.2 & L_2(11)^{(1)} & M_9.S_3 & 2\times S_5 & M_8.S_4  & 4^2.D_{12} & A_4\times S_3 & L_2(11)^{(2)} & A_5 & 3^{1+2}.4  \\
		\hline
		g & 11A & 5A & 11A & 3B & 5A & 4A & 4A & 6A & 11A & 5A & 3B \\
		\hline
		\fp(g) & 1&1&1&4&1&2&3&\le 5 & 1 & 4 & \le 9
	\end{array}
	\]
	
	\subsection{$G=M_{22}$ } The permutation character $1_M^G$ is given in \cite{atlas} for all maximal subgroups:
	\[
	\begin{array}{l|lllllll}
		M & L_3(4) & 2^4.A_6 & A_7 & 2^4.S_5 & 2^3.L_3(2) & M_{10} & L_2(11)  \\
		\hline
		g & 7A & 5A & 5A & 5A & 7A & 5A & 11A \\
		\hline
		\fp(g) & 1&2&1&1&1&1&1
	\end{array}
	\]
	
	\subsection{$G=M_{23}$ } 
	\[
	\begin{array}{l|lllllll}
		M & M_{22} & L_3(4).2 & 2^4.A_7 & A_8 & M_{11} & 2^6.(3\times A_5).2 & 23.11  \\
		\hline
		g & 11A & 7A & 7A & 7A & 11A & 15A & 23A \\
		\hline
		\fp(g) & 1&1&1&2&1&1&1
	\end{array}
	\]
	
	\subsection{$G=M_{24}$ } 
	\[
	\begin{array}{l|lllllllll}
		M & M_{23} & M_{22}.2 & 2^4.A_8 & M_{12}.2 & 2^6.3S_6 & L_3(4).S_3 & 2^6.(L_3(2) \times S_3) & L_2(23) & L_2(7)  \\
		\hline
		g & 23A & 11A & 7A & 11A & 5A & 7A & 7A & 23A & 7A \\
		\hline
		\fp(g) & 1&1&3&1&1&1&1&1&6
	\end{array}
	\]
	
	\subsection{$G=J_1$ } 
	\[
	\begin{array}{l|lllllll}
		M & L_2(11) & 2^3.7.3 & 2\times A_5 & 19.6 & 11.10 & D_6\times D_{10} & 7.6  \\
		\hline
		g & 11A & 7A & 5A & 19A & 11A & 5A & 7A \\
		\hline
		\fp(g) & 2&2&3&1&1&1&1
	\end{array}
	\]
	
	\subsection{$G=J_2$ } 
	\[
	\begin{array}{l|lllllllll}
		M & U_3(3) & 3.A_6.2 & 2^{1+4}.A_5  & [2^6].3.S_3 & A_4\times A_5 & A_5\times D_{10} & L_3(2).2 & 5^2.D_{12} & A_5  \\
		\hline
		g & 7A & 3A & 5C & 6A & 15A & 15A & 7A & 5C & 5C\\
		\hline
		\fp(g) & 2&1&5&6&2&1&1& \le 4 & \le 20
	\end{array}
	\]
	
	\subsection{$G=HS$ } 
	\[
	\begin{array}{l|lllllllllll}
		M  & M_{22} & U_3(5).2 & L_3(4).2  & S_8 & [5^3.2^4] & 2^4.S_6 & 4^3.L_3(2) & M_{11} & [2^6].S_5 & A_6.[2^3] &  5.4\times A_5 \\
		\hline
		g & 11A & 5C & 7A & 7A & 5C & 6B & 7A & 11A & 10A & 10B & 10B\\
		\hline
		\fp(g) & 1&1&1&1 & 1 & \le 10 & 2 & 1 & \le 8 & \le 20 & \le 20 
	\end{array}
	\]
	
	\subsection{$G=J_3$ } 
	\[
	\begin{array}{l|lllllllll}
		M  & L_2(16).2 & L_2(19) & 19.9 & [2^4.3].A_5  & L_2(17) & (3\times A_6).2 & [3^5].8 & 2^{1+4}.A_5 &  [2^6].(3\times S_3) \\
		\hline
		g & 17A & 19A & 19A & 15A & 17A & 15A & 9A & 10A & 4A \\
		\hline
		\fp(g) & 2 & 1 & 1 & \le 6 & 1 & \le 6 & \le 3 & \le 4 & \le 21
	\end{array}
	\]
	
	\subsection{$G=McL$ } 
	\[
	\begin{array}{l|llllllllll}
		M  & U_4(3) & M_{22} & U_3(5)  & 3^{1+4}.2S_5 & 3^4.M_{10} & L_3(4).2 & 2A_8 & 2^4.A_7 & M_{11} & 5^{1+2}.3.8 \\
		\hline
		g & 7A & 11A & 7A & 10A & 9A & 7A & 7A & 7A & 11A & 8A \\
		\hline
		\fp(g) & 1&1&2& \le 12 & \le 3 & 1 & 1 & 1 & 1 & \le 4
	\end{array}
	\]
	
	\vspace{2mm}
	
	From now on, the lists of maximal subgroups become longer than the ones seen already, and rather than give a full table, we just summarise our conclusions. 
	
	\subsection{$G=He$ } Here we need to use the permutation character or the bound (\ref{fpbd}) for the following four maximal subgroups:
	\[
	\begin{array}{l|llll}
		M  & Sp_4(4).2 & 2^2.L_3(4).S_3 & 7^2.SL_2(7) & 7^{1+2}.(S_3 \times 3) \\
		\hline
		g & 17A & 15A & 7D & 7D \\
		\hline
		\fp(g) & 1 & 2 & \le 14 & 1
	\end{array}
	\]
	For the remaining eight classes of maximal subgroups, namely $(S_5\times S_5).2$, $2^6.3S_6$, $[2^8].3^2.2^2$, $2^{1+6}.L_3(2)$, $3S_7$, $S_4\times L_3(2)$, $7.3\times L_3(2)$, $5^2.4A_4$, we can choose elements $g \in M$ of orders 10, 15, 6, 14, 21, 28, 21, 12 respectively, such that $|C_G(g)|^3 < |G:M|$.
	
	\subsection{$G=Ru$ } For the maximal subgroup $M = \,^2\!F_4(2)$, the permutation character is given in \cite{atlas}, and an element $g$ of order 13 has $\fp(g) = 4$. For the remaining maximal subgroups, we can choose an element $g \in M$ of order 5, 7, 10, 13, 15 or 29, such that $|C_G(g)|^3 < |G:M|$.
	
	\subsection{$G=Suz$ } For the maximal subgroup $M = G_2(4)$, the permutation character is given in \cite{atlas}, and an element $g$ of order 13 has $\fp(g) = 1$. For the remaining maximal subgroups, we can choose an element $g \in M$ of order 7, 9, 10, 11, 13, 15 or 21, such that $|C_G(g)|^3 < |S:M|$.
	
	\subsection{$G=ON$ } For each maximal subgroup, we can choose an element $g \in M$ of order 5, 6, 7, 10, 11, 12, 19, 20, 28 or 31, such that $|C_G(g)|^3 < |G:M|$.
	
	\subsection{$G=Co_3,\,Co_2$ } For the maximal subgroups $McL.2$ of $Co_3$, and $U_6(2).2$ of $Co_2$,  elements $g$ of orders 5 or 11 satisfy $\fp(g)=1$. For the remaining maximal subgroups, we can choose an element $g \in M$ of order 5, 8, 9, 11, 14, 15, 20, 21 or 23, such that $|C_G(g)|^3 < |G:M|$.
	
	\subsection{$G=Fi_{22}$ } The full list of maximal subgroups can be found in \cite{KW22}. For the maximal subgroups $2.U_6(2)$, $\Omega_7(3)$ and $\Omega_8^+(2).S_3$, elements $g$ of orders 11, 7, 7 respectively,  satisfy $\fp(g)=1,3,1$. For the remaining maximal subgroups, we can choose an element $g \in M$ of order 7, 9, 10, 11 or 13, such that $|C_G(g)|^3 < |G:M|$.
	
	\subsection{$G=HN$ } For each maximal subgroup, we can choose an element $g\in M$ of order 9, 11, 15, 19, 20 or 21, such that $|C_G(g)|^3 < |G:M|$.
	
	\subsection{$G=Ly$ } For each maximal subgroup, we can choose an element $g \in M$ of order 9, 11, 20, 31, 37 or 67, such that $|C_G(g)|^3 < |G:M|$.
	
	\subsection{$G=Th$ } The full list of maximal subgroups can be found in \cite{Linton}. 
	For each maximal subgroup, we can choose an element $g \in M$ of order 5, 8, 10, 12, 13, 14, 19, 27 or 31, such that $|C_G(g)|^3 < |G:M|$.
	
	\subsection{$G=Fi_{23}$ } The full list of maximal subgroups can be found in \cite{KW23}. For the maximal subgroups $2.Fi_{22}$, $P\Omega_8^+(3).S_3$, an element $g$ of order 13  satisfies $\fp(g)=3$ or 1. For the remaining maximal subgroups, we can choose an element $g \in M\cap S$ of order 11, 13, 15, 17, 21, 23 or 27, such that $|C_G(g)|^3 < |G:M|$.
	
	\subsection{$G=Co_{1}$ } The full list of maximal subgroups can be found in \cite{atlas} (with a correction in \cite{WilCo1}). For all of these, it is immediate from the information given there that there is an element $g \in M$ of order 7, 9, 10, 11, 14, 15, 16, 21, 23 or 39, such that $|C_G(g)|^3 < |G:M|$. 
	
	\subsection{$G=J_4$ } The full list of maximal subgroups can be found in \cite{KWJ4}. 
	For each maximal subgroup, we can choose an element $g \in M$ of order 7, 20, 21, 22, 23, 29, 31, 37 or 43, such that $|C_G(g)|^3 < |G:M|$.
	
	\subsection{$G=Fi_{24}'$ } The full list of maximal subgroups can be found in \cite{LW24}. 
	For each maximal subgroup, we can choose an element $g \in M$ of order 8, 11, 13, 17, 18, 20, 21, 22, 23, 29, 30, 35 or 39, such that $|C_G(g)|^3 < |G:M|$.
	
	\subsection{$G=BM$ } The full list of maximal subgroups can be found in \cite{WilBM}. 
	For each maximal subgroup, we can choose an element $g \in M$ of order 11, 13, 15, 17, 19, 23, 25, 27, 28, 31, 33, 35, 39, 47 or 55, such that $|C_G(g)|^3 < |G:M|$.
	
	\subsection{$S=M$, the Monster }  The full list of maximal subgroups can be found in \cite{DLP}. 
	For each maximal subgroup, we can choose an element $g \in M$ of order 7, 11, 13, 15, 17, 19, 21, 23, 29, 30, 31, 33, 35, 39, 41, 47, 48, 52, 55, 59 or 71, such that $|C_G(g)|^3 < |G:M|$.
	
	\vspace{2mm} This completes the proof of Theorem \ref{spor}. 
	
	%%%%%%%%%%%%%%
	
	\section{Exceptional groups}\label{exceppf}
	
	In this section we prove Theorem \ref{t:main_simple} for primitive groups with socle an exceptional group of Lie type.
	
	\begin{theorem}\label{excep} Let $G$ be a simple group of exceptional Lie type and let $M\in \ca A(G)$. Then there exists $g\in M$ such that $\fp(g,G/M)<|G:M|^{1/3}$.
	\end{theorem}
	
	The proof is divided into several subsections.
	
	\subsection{Maximal subgroups of exceptional groups}
	
	Let $G$ be a simple group of exceptional Lie type over $\F_q$, where $q=p^a$ and $p$ is prime. There is a simple adjoint adjoint algebraic group $\bar G$ over $\bar \F_p$, and a Frobenius endomorphism $F$ of $\bar G$, such that $G = (\bar G^F)'$. 
	
	The next result, taken from \cite{LS3}, together with the remarks following it, summarizes the current state of knowledge of the maximal subgroups of finite exceptional groups. In part (4) of the statement, by a subgroup of the same type we mean the normalizer of a subfield subgroup $G(q_0)$ (where $\F_{q_0} \subset \F_q$), or of a twisted version (explicitly: $^2\!E_6(q^{1/2}) < E_6(q)$, $^2\!F_4(q) < F_4(q)$, $^2\!G_2(q) < G_2(q)$); such subgroups are unique up to conjugacy in $\bar G^F$, by \cite[Thm. 5.1]{LSroot}.
	
	\begin{theorem} (\cite[Thm. 8]{LS3})
		\label{maxexcep}
		Let $G$ be a simple group of exceptional Lie type over $\F_q$, $q=p^a$, and let $M \in \ca A(G)$. Then one of the following holds.
		\begin{itemize}
			\item[(1)] $M$ is a parabolic subgroup.
			\item[(2)] $M=N_G(\bar M^F \cap G)$, where $\bar M< \bar G$ is connected reductive of maximal rank: the possibilities are listed in \cite[Tables 5.1,5.2]{LSS}.
			\item[(3)] $M=N_G(\bar M^F \cap G)$, where $\bar M$ is connected reductive of non-maximal rank: these are listed in \cite[Table 3]{LS3}, together with the following subgroups:
			\[
			\begin{array}{l}
				G_2(q),\,\PGL_3^{\pm}(q) < \,^3\!D_4(q),\\
				(2^2\times \POm_8^+(q).2^2).S_3,\;^3\!D_4(q) < E_7(q)\,(p\ne 2),\\
				\PGL_2(q) \times S_5 < E_8(q)\,(p>5), \\
				F_4(q) < E_8(q) \,(p=3).
			\end{array}
			\]
			\item[(4)] $M$ is of the same type as $G$.
			\item[(5)] $M$ is an `exotic local' subgroup, or the `Borovik' subgroup $(Alt_5\times Alt_6).2^2 < E_8(q)$; the exotic locals are as follows:
			\[
			\begin{array}{l}
				2^3.\SL_3(2) < G_2(p)\;(p>2),\\
				3^3.\SL_3(3) < F_4(p) \;(p>5),\\
				3^{3+3}.\SL_3(3) < E_6^\epsilon(p)\;(p\equiv \epsilon \hbox{ mod }3),\\
				5^3.\SL_3(5) < E_8(q)\; (q = p \hbox{ or }p^2),\\
				2^{5+10}.\SL_5(2) < E_8(p)\; (p>2).
			\end{array}
			\]
			\item[(6)] $M$ is in a class ${\mathcal U}$ of almost simple subgroups, not occurring in items (1)-(5); we divide these into subclasses, according to the socle $M_0$ of $M$:
			\begin{itemize}
				\item[(a)] ${\mathcal U}_{p,1}$: $M_0 \in Lie(p),\,M_0 \ne \PSL_2(p^a)$
				\item[(b)] ${\mathcal U}_{p,2}$: $M_0 \in Lie(p),\,M_0 = \PSL_2(p^a)$
				\item[(c)] ${\mathcal U}_{ASp'}$: $M_0$ alternating, sporadic or in $Lie(p')$.
				
			\end{itemize}
		\end{itemize}
	\end{theorem}
	
	Note that in part (3), the last subgroup $F_4(q) < E_8(q)$ with $q=3^a$ was found in \cite{CST} (having been omitted in error from \cite[Table 3]{LS3} and the references cited for it).
	
	For $G \ne E_7(q), E_8(q)$, the class ${\mathcal U}$ is known (see \cite{craven2023maximal, kleidman1988maximal_3D4, kleidman1988maximalG_2, malle1991maximal2F4, suzuki1962class}); and for $E_7(q)$, $E_8(q)$, while the class is not known completely, considerable restrictions on the groups in ${\mathcal U}_{p,1}$,  ${\mathcal U}_{p,2}$ are obtained in \cite{Crmem1, Crmem2}, and on the groups in ${\mathcal U}_{ASp'}$ in \cite{Cralt, Litterick}.
	
	\subsection{Parabolic subgroups}
	
	Here we prove Theorem \ref{excep} in the case where a point-stabilizer $M$ is a parabolic subgroup. This follows from the next lemma, using the existence of regular unipotent elements.

	\begin{lemma}
		\label{l:parabolic}
		Let $G$ be a simple group of Lie type in characteristic $p$, and let $P$ be a parabolic subgroup of $G$. Let $g\in G$ be a regular unipotent element. Then $\fp(g,G/P)=1$.
	\end{lemma}
	
	\begin{proof}
		Let $B$ be a Borel subgroup of $P$ and $U = O_p(B)$, so that $U$ is a Sylow $p$-subgroup of $G$. We may take $g \in U$. If $g \in P^x$ with $x \in G$, then there exists $y \in P$ such that $g \in U^{yx}$. Since $g$ is regular unipotent, it lies in precisely one conjugate of $U$, so $U=U^{yx}$. Therefore $yx\in N_S(U)= B \le P$, so $P^x = P$. Thus $g$ lies in precisely one conjugate of $P$, and the conclusion follows.
	\end{proof}
	
	\subsection{Subgroups of maximal rank}
	
	In this subsection we prove Theorem \ref{excep} in the case where a point-stabilizer $M$ is a subgroup of maximal rank -- that is, a maximal subgroup as in part (2) of Theorem \ref{maxexcep}. Recall from Lemma \ref{l:fpr_enough} that for $g \in G$, we have 
	\begin{equation}\label{fixform}
		\fp(g,G/M) = \frac{|C_G(g)|\cdot |g^G \cap M|}{|M|}.
	\end{equation}
	In almost all cases, we shall find an element $g \in M$ such that 
	\begin{equation}\label{csgbd}
		|C_G(g)|  < |G: M|^{1/3},
	\end{equation}
	which of course suffices to give the conclusion of Theorem \ref{excep}.
	
	We first handle the families of rank at most 4.
	
	\begin{lemma}\label{maxrksmall} Theorem \ref{excep} holds in the cases where $M$ is a subgroup of maximal rank and $G$ is of type $^2\!B_2$, $^2\!G_2$, $^3\!D_4$, $^2\!F_4$, $G_2$ or $F_4$.
	\end{lemma}
	
	\begin{proof}
		Recall that $M$ is as in \cite[Tables 5.1,5.2]{LSS}. 
		
		Assume first that $G=\,^2\!B_2(q)$, with $q=2^{2a+1}$ and $a\ge 1$. Then $M$ is a torus normalizer $(q-1).2$ or $(q\pm\sqrt{2q}+1).4$. Taking $g \in M$ of order $q-1$ or $q\pm\sqrt{2q}+1$ respectively, we have $C_G(g) \le M$ and $g^G \cap M = g^M$, so $\fp(g,G/M)=1$ by (\ref{fixform}).
		
		Next assume that $G=\,^2\!G_2(q)$ with $q=3^{2a+1}$ and $a\ge 1$. Then $M$ is either $2\times L_2(q)$ or a torus normalizer $(q+1).6$ or $(q\pm \sqrt{3q}+1).6$. In the first case, an element $g\in M$ of order $(q-1)/2$ is regular semisimple, and $|C_G(g)| = q-1 < |G:M|^{1/3}$, giving the conclusion. In the other cases, we choose $g \in M$ of order $q+1$ or $q\pm \sqrt{3q}+1$,  and then $\fp(g,G/M)=1$ as in the previous paragraph.
		
		Now assume that $G=G_2(q)$ with $q>2$. Then $M$ is $(\SL_2(q)\circ \SL_2(q)).(2,q-1)$, $\SL_3^\epsilon(q).2$ (with $\eps = \pm$) or a torus normalizer. In the first case, we choose $g\in M$ of order $q^2-1$, and then 
		$|C_G(g)| = q^2-1 < |G:M|^{1/3}|$. In the second case, choose $g \in M$ of order $q^2+\epsilon q+1$; 
		then $C_G(g) \le M$ and $g^G \cap M = g^M$, so $\fp(g,G/M)=1$.
		Finally, suppose $M$ is a torus normalizer. Then $q = 3^a \ge 9$ and $M$ is $(q\pm 1)^2.D_{12}$ or $(q^2 \pm q+1).6$, and we can choose $g \in M$ of order $q\pm 1$ or $q^2\pm q+1$ such that $g$ is regular semisimple; then $C_G(g)| = (q \pm 1)^2$ or $q^2\pm q+1$, which is less than $|G: M|^{1/3}$.
		
		Next let $G=\,^3\!D_4(q)$.  Inspecting the semisimple classes in \cite[Table II]{kleidman1988maximal_3D4}, we see that $M$ contains a regular semisimple element unless $q=2$, and 
		$M =\SU_3(2).3.2$ or $3^2.\SL_2(3)$. When $M$ contains a regular semisimple element $g$, we have $|C_G(g)| \le (q^2+q+1)^2$ (the size of the largest maximal torus), and this is less than 
		$|G: M|^{1/3}$ in all cases. 
		In the exceptional cases with $q=2$ and $M =\SU_3(2).3.2$ or $3^2.\SL_2(3)$, we pick $g \in M$ of order 9 or 6 respectively, and then $|C_G(g)| \le 54$ or 72 (see \cite[p.90]{atlas}), and again 
		this is less than $|G: M|^{1/3}$.
		
		Next, consider the case $G=\,^2\!F_4(q)'$, where $q=2^{2a+1}$. For $a\ge 1$, the semisimple classes of $G$ are listed in \cite[Table IV]{shinoda1974conjugacy_ree}, and we see that each choice of $M$ contains a regular semisimple element $g$. (Indeed, for $q>2$ the only maximal tori not containing a regular semisimple element are $(q-1)^2$ and $(q-\sqrt{2q}+1)^2$, only for $q=8$.) Hence $|C_G(g)| \le (q + \sqrt{2q}+1)^2$ 
		(the size of the largest maximal torus), and again this is less than $|G: M|^{1/3}$ in all cases. 
		To conclude, suppose $a=0$ (so $G = \,^2\!F_4(2)'$). The maximal subgroups are listed in \cite[p.74]{atlas}, and those of maximal rank are just $A_6.2^2$ and $5^2.4A_4$. Both of these have an element $g$ of order 6, and $|C_G(g)| = 12$, which is less than $|G: M|^{1/3}$.
		
		Finally, assume that $G = F_4(q)$. 
		Suppose first that $q>2$. Inspecting the semisimple classes of $G$ in \cite{shinoda1974conjugacy,shoji1974conjugacy}, we see that $M$ contains a regular semisimple element $g$, unless possibly we are in one of the following cases: 
		\begin{itemize}
			\item[(i)] $M =(q-1)^4.W(F_4)$ with $q=8$; 
			\item[(ii)] $M =(q+1)^4.W(F_4)$ with $q=4,8$.  
		\end{itemize}
		If $M$ contains a regular semisimple element $g$, then $|C_G(g)| \le (q+1)^4$ (the size of the largest maximal torus), and this is less than $|G: M|^{1/3}$ for all possibilities for $M$ in \cite[Tables 5.1,5.2]{LSS}. 
		In case (i) above, there is an element $g$ in the maximal torus $7^4$ (namely, the element $h_7$ in \cite[Table II]{shinoda1974conjugacy}), such that $|C_G(g)| = |\SL_2(q)|^2(q-1)^2$, and this is less than $|G: M|^{1/3}$. Likewise, in case (ii) there is an element $g = h_{17}$ in the maximal torus $(q+1)^4$ with centralizer order $|\SL_2(q)\SU_3(q)|(q+1)$, and this is less than $|G: M|^{1/3}$.
		
		Now suppose $q=2$, so $G = F_4(2)$. The conjugacy classes and centralizer orders in $G$ can be found in \cite[p.167-8]{atlas}. The maximal rank subgroups $B_4(2)$, $B_2(4)$, $^3\!D_4(2)$ possess elements of orders 17,17,21 respectively, with the same centralizer orders; the subgroups $D_4(2)$, $B_2(2)^2$ have an element of order 15 with centralizer order 90; the subgroup $3.(^2\!A_2(2))^2$ has an element of order 9 with centralizer order 54; and the subgroup $7^2.(3\times \SL_2(3)$ has an element of order 7 with centralizer order 1176. In all cases the centralizer order is less than $|G: M|^{1/3}$, completing the proof.
	\end{proof}
	
	In the proof of the next lemma, we use standard notation $\l_i$ for fundamental dominant weights, and denote by $V(\l)$ (or just $\l$) the irreducible module over $\bar \F_p$ of highest weight $\l$ for the simply connected cover of $G$.
	
	\begin{lemma}\label{maxrke6} Theorem \ref{excep} holds in the cases where $M$ is a subgroup of maximal rank and $G= E_6^\eps(q)$.
	\end{lemma}
	
	\begin{proof}
		The semisimple element centralizers in $G$ are given in \cite{mizuno, DL}, and the unipotent element centralizers in \cite[Table 22.2.3]{liebeck_seitz_2012unipotent}. 
		
		Suppose first that $M$ is a torus normalizer in \cite[Table 5.2]{LSS}, 
		so that $M = ((q-\eps)^6/d).W(E_6)$, where $d = (3,q-\eps)$ and also $q\ge 5$ if $\eps = +$. Choose an element $x \in M$ projecting to an element of order 5 in $W(E_6)$. In the root space decomposition 
		\[
		L(\bar G) = H \oplus \bigoplus_{\alpha \in \Phi} L_{\alpha}
		\]
		of the Lie algebra of $\bar G$, the element $x$ permutes the 72 root spaces $L_\alpha$ as a permutation of cycle-shape $(5^{14},1^2)$, and fixes at most a 2-dimensional subspace of the Cartan subspace $H$. Hence 
		$\dim C_{L(\bar G)}(x) \le 18$, and it follows that $\dim C_{\bar G}(x) \le 18$ also. Inspecting the list of possible centralizers of such dimension in \cite{mizuno, DL, liebeck_seitz_2012unipotent}, we check that $|C_G(x)| < |G:M|^{1/3}$ except possibly when $(q,\eps) = (2,-)$. However in this case we see from \cite[p.192]{atlas} that $|C_G(x)| = 100800$, which again is less than $|G:M|^{1/3}$.
		
		The other subgroups of maximal rank are in \cite[Table 5.1]{LSS}, and are the normalizers of subgroups 
		$M_0 = \bar M^F$, where the possibilities for $\bar M, M_0$ are as follows:
		\[
		\begin{array}{ll}
			\bar M & M_0 \\
			\hline
			A_1A_5 & A_1(q)A_5^\eps(q) \\
			A_2^3 & A_2^\eps(q)^3,\,A_2(q^2)A_2^{-\eps}(q),\,A_2^\eps(q^3) \\
			D_4T_2 & D_4(q)(q-\eps)^2/d,\,^3\!D_4(q) (q^2+\eps q+1)/d \\
			D_5T_1 & D_5^\eps(q) (q-\eps)/d \\
			\hline
		\end{array}
		\]
		For $\bar M = A_1A_5$, the subgroup $M_0$ has a cyclic torus $\langle g \rangle$ of order $(q+\eps)(q^5-\eps)/d$. From the restriction $L(E_6) \downarrow A_1A_5 = L(A_1A_5) \oplus 1 \otimes \lambda_3$ (see \cite[11.10]{liebeck_seitz_2012unipotent}), we see that $\dim C_{L(E_6)}(g) = 6$ or 8 (the latter only when $q-\eps \le 2$), and hence $|C_G(g)| < |G:M|^{1/3}$.
		
		Next let $\bar M = A_2^3$. Here 
		\begin{equation}\label{rest}
			L(E_6) \downarrow A_2^3 = L(A_2^3) \oplus (\l_1\otimes \l_1\otimes \l_1) \oplus (\l_2\otimes \l_2\otimes \l_2)
		\end{equation}
		(see \cite[2.1]{LSmem96}). The subgroups $M_0 = A_2(q^2)A_2^{-\eps}(q)$ and $A_2^\eps(q^3)$ contain regular semisimple elements -- for example, the first contains an element $g = g_1g_2$ with $g_1 \in A_2(q^2)$ of order
		$q^4+q^2+1$ and $g_2 \in A_2^{-\eps}(q)$ of order $q^2-1$, and (\ref{rest}) gives $\dim C_{L(E_6)}(g) = 6$. Similarly, $M_0 = A_2^\eps(q)^3$ contains a regular semisimple element, except when $(q,\eps) = (2,-)$, in which case we see from \cite[p.191]{atlas} that it contains an element $g$ of order 9 with centralizer order 162. In all cases we have $|C_G(g)| < |G:M|^{1/3}$.
		
		Now suppose $\bar M = D_4T_2$. From \cite[Chap.11]{liebeck_seitz_2012unipotent} we have $L(E_6)\downarrow D_4 = L(D_4) \oplus L(T_2) \oplus (\l_1\oplus \l_3\oplus \l_4)^2$. From this we see that the subgroup $M_0 =\, ^3\!D_4(q) (q^2+\eps q+1)/d$ has a regular semisimple element. Now consider $M_0 = D_4(q)(q-\eps)^2/d$. For $q=2$ we have $\eps = -$ (see \cite[Table 5.1]{LSS}), and $M$ has an element $g$ of order 21 with centralizer order at most 63 (see \cite[p.191]{atlas}), which is less than $|G:M|^{1/3}$. Suppose $q\ge 3$. The subgroup $M$ has an $S_3$ inducing graph automorphisms on $D_4$, so we may pick $\tau \in M$ inducing a triality and centralizing $G_2(q)<D_4(q)$. Let $g = x\tau$, where $x$ is an element of order $q^2+q+1$ in this $G_2(q)$. From the restriction $L(E_6)\downarrow D_4$, we see that $\dim C_{L(E_6)}(\tau) = 30$. If $p\ne 3$, then $C_{\bar S}(\tau) = D_4T_2$ and $x$ is regular in the $D_4$ factor, hence $g$ is regular in $E_6$. And if $p=3$, then $C_{E_6}(x) = T_2A_2A_2$ with $\tau$ regular unipotent in each $A_2$ factor, so $|C_S(g)| \le (q^2+q+1)q^4$. In all cases $|C_G(g)| < |G:M|^{1/3}$.
		
		Finally consider $\bar M = D_5T_1$. Choose an element $g$ of order $(q^5-\eps)/(q-\eps)$ in a subgroup $^2\!A_4(q)$ of $M_0$. From the restriction
		\[
		L(E_6) \downarrow A_4 = L(A_4) \oplus \l_1 \oplus \l_4\oplus (\l_2\oplus \l_3)^2 \oplus 0^4
		\]
		(see \cite[Chap.11]{liebeck_seitz_2012unipotent}), we see that $\dim C_{L(E_6)}(g) = 8$, and so $|C_G(g)| < |G:M|^{1/3}$ as usual.
	\end{proof}
	
	\begin{lemma}\label{maxrke78} Theorem \ref{excep} holds in the cases where $M$ is a subgroup of maximal rank and $S= E_7(q)$ or $E_8(q)$.
	\end{lemma}
	
	\begin{proof}
		First we handle the maximal rank subgroups in \cite[Table 5.1]{LSS}. These are normalizers of subgroups $\bar M^F$, where $\bar M$ is a connected subgroup of $\bar G$ listed in Table \ref{maxe78}. In the table we also list the 
		$\bar G$-class of a unipotent element $u \in \bar M^F$, as given by \cite{Lawunip}. The unipotent class is labelled as in \cite[Tables 22.2.1,22.2.2]{liebeck_seitz_2012unipotent}, where the centralizer order is also given, and an upper bound for this order is included in Table \ref{maxe78}. In all cases except $A_2^4\,(p=3),\,A_1^8 < E_8$ and $A_1^3D_4\,(q=2),\,A_1^7< E_7$, we have $|C_Gu)| < |G:M|^{1/3}$, so it remains to consider these cases.
		
		\begin{table}[h!]
			\caption{Maximal rank subgroups in $E_7,E_8$}\label{maxe78}
			\[
			\begin{array}{|l|l|l|l|}
				\hline
				\bar G & \bar M & \hbox{unip. elt. }u \in \bar M^F & |C_G(u)| \le \\
				\hline
				E_8 & D_8 & E_8(a_4)\,(p\ne 2),\,E_8(b_4)\,(p=2) & 2q^{18} \\
				& A_1E_7 & E_8(a_3)\,(p\ne 2),\,E_7\,(p=2) & 4q^{16} \\
				& A_8 & E_8(a_6)\,(p\ne 3),\,E_8(b_6)\,(p=3) & 6q^{28} \\
				& A_2E_6 & E_8(b_5)\,(p\ne 3),\,E_6A_1\,(p=3) & 3q^{26} \\
				& A_4^2 & E_8(a_7)\,(p\ne 5),\,A_4A_3\,(p=5) & q^{48} \\
				& D_4^2 & A_6 & 2q^{38} \\
				& A_2^4 & A_2D_4(a_1) \,(p\ne 3),\,A_2^2A_1^2\,(p=3) &  2q^{64}\,(p\ne 3),q^{80}\,(p=3) \\
				& A_1^8 & - & - \\
				\hline
				E_7 & A_1D_6 & E_7(a_3) & 2q^{13} \\
				& A_7 & E_6(a_1)\,(p\ne 2), \,E_7(a_4)\,(p=2) & 2q^{17} \\
				& A_2A_5 & E_7(a_5)\,(p\ne 3), \,A_5A_1\,(p=3) & q^{25} \\
				& E_6T_1 & E_6 & 3q^{13} \\
				& A_1^3D_4 & A_1D_5(a_1)\,(p\ne 2), \,A_1D_4\,(p=2) & q^{25}\,(p\ne 2),2q^{31}\,(p=2) \\
				& A_1^7 & - & - \\
				\hline
			\end{array}
			\]
		\end{table}
		
		Consider the case where $\bar M = A_2^4 < E_8$ with $p=3$. Here $M$ is the normalizer of one of the following subgroups $M_0$:
		\begin{itemize}
			\item[(i)] $A_2^\eps(q)^4\, (\eps=\pm)$;
			\item[(ii)] $^2\!A_2(q^2)^2$;
			\item[(iii)] $^2\!A_2(q^4)$.    
		\end{itemize}
		The restriction $L(E_8) \downarrow A_2^4$ is given by \cite[Prop. 2.2]{LSmem96}. From this we see that in cases (ii) and (iii), $M$ contains a maximal torus of order $(q^4-q^2+1)^2$ or $q^8-q^4+1$ respectively, each of which has a regular semisimple element $g$, so $|C_G(g)| < |G:M|^{1/3}$ in these cases. In case (i), we choose an element $g$ projecting to each of the $A_2^\eps(q)$ factors as a regular semisimple element of order $q^2+\eps q+1$. Then we find from the restriction $L(E_8) \downarrow A_2^4$ that $\dim C_{L(E_8)}(g) \le 32$, giving (\ref{csgbd}) as usual.
		
		Now consider $\bar M = A_1^8 < E_8$, in which case $M$ is the normalizer of a subgroup $A_1(q)^8$ and $q>2$. Here $M$ contains a maximal torus $T$ of order $(q+1)^8$, and $T$ lies in a maximal rank subgroup normalizing $^2\!A_4(q)^2$. Hence there is an element $g \in T$ of order $q+1$ with centralizer containing $^2\!A_3(q)^2\!A_4(q)$, and in fact $|C_G(x)|$ is either $|^2\!A_3(q)^2\!A_4(q)|(q+1)$ or $|^2\!A_3(q)^2\!D_5(q)|$ (the latter only if $q=3$). Hence (\ref{csgbd}) holds as usual.
		
		Next suppose $\bar M = A_1^3D_4 < E_7$ with $q=2$. Here $M$ is the normalizer of $M_0 = A_1(2)^3.D_4(2)$ or
		$A_1(2^3).\,^3\!D_4(2)$. The restriction 
		\[
		L(E_7) \downarrow A_1^3D_4 = L(A_1^3D_4) \oplus (1\otimes 1\otimes 0 \otimes \l_1) 
		\oplus (1\otimes 0\otimes 1 \otimes \l_3) \oplus (0\otimes 1\otimes 1 \otimes \l_4) 
		\]
		(see \cite[Chap.11]{liebeck_seitz_2012unipotent}). If $M_0 = A_1(2)^3.D_4(2)$, choose an element $g = g_1g_2g_3g_4 \in M_0$ with $g_1,g_2,g_3 \in A_1(2)$ of order 3 and $g_4 \in D_4(2)$ of order 7. Then from the above restriction we find that $\dim C_{L(E_7)}(g) = 19$, and hence (\ref{csgbd}) holds. In the other case $M_0 = A_1(2^3).\,^3\!D_4(2)$, choose $g = g_1g_2$ with $g_1 \in A_1(8)$ of order 7 and $g_2 \in \,^3\!D_4(2)$ of order 13. Then $g$ is regular in $E_7$ and again (\ref{csgbd}) holds.
		
		Finally, consider $\bar M = A_1^7 < E_7$, in which case $M$ normalizes $M_0 = A_1(q)^7$ or $A_1(q^7)$. In the latter case $M_0$ has an element of order $q^7+1$ which is regular in $\bar S$. In the former case we have $q>2$, and $M_0$ has a maximal torus $T$ of order $(q+1)^7/d$ ($d=(2,q-1)$). For $q>3$, note that $T$ lies in a maximal rank subgroup 
		$A_2^-(q)A_5^-(q)$, so has an element $g$ of order $q+1$ with centralizer containing $A_2^-(q)A_4^-(q)(q+1)/d$; this must be the full centralizer in $S$, and (\ref{csgbd}) holds. And if $q=3$, $T$ has an element $g$ of order 4 with centralizer $A_1(q)A_3^-(q)A_3^-(q)$, and (\ref{csgbd}) holds again.
		
		Now we handle the cases where $M$ is a torus normalizer in \cite[Table 5.2]{LSS}. For $\bar G = E_7$, we have $M = T.W(E_7)$, where $|T| = (q-\eps)^7/d$ ($\eps = \pm 1$), and $q\ge 5$ if $\eps =+1$. Choose an element $g \in M$ mapping to an element of order 7 in $W(E_7)$. Then $g$ permutes the root spaces $L_\alpha$ in the root space decomposition $L(E_7) = H \oplus \sum_{\alpha \in \Phi}L_\alpha$ in 18 cycles of length 7, and also $\dim C_H(g)=1$. Hence $\dim C_{L(E_7)}(g) \le 19$ and (\ref{csgbd}) holds.
		
		For $\bar G = E_8$, $M$ is the normalizer of a torus $T$ of one of the following orders:
		\begin{itemize}
			\item[(i)] $(q-1)^8$ \,($q\ge 5$),
			\item[(ii)] $(q+1)^8$,
			\item[(iii)] $(q^2+q+1)^4$,
			\item[(iv)] $(q^2-q+1)^4$\,($q>2$),
			\item[(v)] $(q^2+1)^4$,
			\item[(vi)] $(q^4+q^3+q^2+q+1)^2$,
			\item[(vii)] $(q^4-q^3+q^2-q+1)^2$,
			\item[(viii)] $(q^4-q^2+1)^2$,
			\item[(ix)] $q^8+q^7-q^5-q^4-q^3+q+1$,
			\item[(x)] $q^8-q^7+q^5-q^4+q^2-q+1$.
		\end{itemize}
		For tori (i) and (ii), we have $M = T.W(E_8)$, and we argue as above using an element of order 7 in the Weyl group $W(E_8)$.
		
		Tori (vi) and (vii) are contained in subsytem subgroups $A_4^\eps(q) A_4^\eps(q)$, and $T$ contains an element $g = g_1g_2$ with each $g_i$ in $A_4^\eps(q)$ of order a primitive prime divisor of $q^5-\eps$. Using the restriction $L(E_8) \downarrow A_4A_4$ given in \cite[Chap.11]{liebeck_seitz_2012unipotent}, we see that $\dim C_{E_8}(g) \le 28$, and (\ref{csgbd}) holds.
		
		Tori (iii), (iv) and (viii) are contained in subsystem subgroups normalizing $A_2^\eps(q)^4$ or $A_2^-(q^2)^2$, and are handled as in the previous paragraph.
		
		Torus (v) lies in a subsystem $D_8(q)$ (as a subgroup of type $O_2^-(q)^4$), and from the restriction $L(E_8) \downarrow D_8$ we find that an element $g = g_1g_2g_3g_4$, with each $g \in O_2^-(q)$ of order $q^2+1$, satisfies $\dim C_{E_8}(g) \le 64$.
		
		Finally, tori (ix) and (x) are cyclic, and their generators are regular semisimple elements $g$, so (\ref{csgbd}) holds once again.
	\end{proof}
	
	\subsection{Subgroups in (3) of Theorem \ref{maxexcep}}
	
	\begin{lemma}\label{max3} Theorem \ref{excep} holds in the case where $M$ is as in part (3) of Theorem \ref{maxexcep}.
	\end{lemma}
	
	\begin{proof}
		In this case we have $M=N_G(\bar M^F \cap G)$, where $\bar M$ is a connected reductive subgroup of non-maximal rank that is either listed in \cite[Table 3]{LS3}, or is one of the four extra possibilities in (3) of Theorem \ref{maxexcep}. 
		
		Suppose first that $M$ is as in \cite[Table 3]{LS3}. We pick a regular unipotent element $u$ of $\bar M^F \cap G$, and in Table \ref{unip3} we list the $\bar G$-class of $u$, as given by \cite{Lawunip}, together with an upper bound for the centralizer order $|C_G(u)|$ given by \cite[Chap.22]{liebeck_seitz_2012unipotent}. In all cases, (\ref{csgbd}) is satisfied.
		
		\begin{table}[h!]
			\caption{Subgroups in \cite[Table 3]{LS3}}\label{unip3}
			\[
			\begin{array}{|l|l|l|l|}
				\hline
				G & \bar M & \bar G\hbox{-class of }u & |C_G(u)| \le \\
				\hline
				E_8(q) & G_2F_4 & E_8(a_5)\,(p\ne 3),\,E_8(b_5)\,(p=3) & 6q^{22} \\
				& A_1G_2G_2 & D_5A_2\,(p\ne 7),\,A_6A_1\,(p=7) & 2q^{36} \\
				& A_1A_2 & E_8(a_4)\,(p\ne 5),\,A_1A_2A_4\,(p=5) & q^{52} \\    
				& B_2 & E_8(a_6)\,(p\ne 5,7),A_6A_1\,(p=7),A_4A_3\,(p=5) & q^{48} \\
				& A_1\,(3 \hbox{ classes}) & E_8, E_8(a_1), E_8(a_2) & 2q^{12} \\
				\hline
				E_7(q) & G_2C_3 & E_7(a_4)\,(p\ne 2,7),\,A_6\,(p=7),\,A_1D_5\,(p=2) & 2q^{20} \\
				& A_1F_4 & E_7(a_2)\,(p\ne 3),\,E_6\,(p=3) & 3q^{13} \\
				& A_1G_2 & E_7(a_4)\,(p\ne 7),\,A_6\,(p=7) & 2q^{20} \\
				& A_1A_1 & E_7(a_5) & 6q^{21} \\
				& A_2 & E_6(a_1)\,(p\ne 5,7),\,A_6\,(p=7),\,A_2A_4\,(p=5) & q^{27} \\
				& A_1\,(2 \hbox{ classes}) & E_7, \,E_7(a_1) & 2q^{9} \\
				\hline
				E_6^\eps(q) & F_4 & E_6 & 3q^{6} \\
				& C_4 & E_6(a_1) & q^{8} \\
				& A_2G_2 & E_6(a_3)\,(p\ne 2),\,D_5(a_1)\,(p=2) & 2q^{14} \\
				& G_2 & E_6(a_1)\,(p\ne 2),\,D_5\,(p=2) & 2q^{11} \\
				& A_2 & E_6(a_3)\,(p\ne 5),\,A_4A_1\,(p=5) & 2q^{16} \\
				\hline
				F_4(q) & A_1G_2 & F_4(a_2) & 8q^{8} \\
				& G_2\,(p=7) & F_4(a_2) & 8q^{8} \\
				& A_1 & F_4 & 4q^{4} \\
				\hline
				G_2(q) & A_1 & G_2 & 3q^{2} \\
				\hline
			\end{array}
			\]
		\end{table}
		
		Now consider the four extra cases in part (3) of Theorem \ref{maxexcep}. The subgroups of $E_7(q)$ and $E_8(q)$ can be dealt using unipotent elements again, as follows. The $F_4 <E_8$ with $p=3$ contains a unipotent element $u$ in the class $E_8(b_4)$ (see \cite{CST}), and $|C_S(u)| < 2q^{18}$. The $D_4<E_7$ (resp. $A_1<E_8$) lies in a subsystem $A_7$ (resp. $A_4A_4$), hence a regular unipotent element  $u$ in the $D_4$ (resp. $A_1$) lies in the class $A_6$ (resp. $E_8(a_7)$ -- see \cite{Lawunip}), and $|C_S(u)| < q^{19}$ (resp. $120q^{40}$). In all cases, (\ref{csgbd}) holds.
		
		It remains to deal with the case where $G = \,^3\!D_4(q)$ and $M$ is one of the subgroups $G_2(q)$ or $\PGL_3^\eps(q)$ (with $q \equiv \eps \hbox{ mod }3$ and $q>2$ for the latter). We use semisimple elements here; centralizers of such elements in $G$ are given in \cite[Table II]{kleidman1988maximal_3D4}. An element $g \in G_2(q)$ of order $q^2+q+1$ (lying in a subsystem $\SL_3(q)$) is regular in $D_4$, and so is an element of order $q^2-1$ in $\PGL_3^\eps(q)$. Hence (\ref{csgbd}) holds in these cases, completing the proof.
	\end{proof}
	
	\subsection{Subfield and twisted subgroups}
	
	\begin{lemma}\label{max4} Theorem \ref{excep} holds in the case where $M$ is as in part (4) of Theorem \ref{maxexcep}.
	\end{lemma}
	
	\begin{proof}
		Recall that the maximal subgroups $M$ in (4) of Theorem \ref{maxexcep} are the normalizers of  subfield subgroups $G(q_0)$ (where $\F_{q_0} \subset \F_q$), or of twisted subgroups $^2\!E_6(q^{1/2}) < E_6(q)$, $^2\!F_4(q) < F_4(q)$ and $^2\!G_2(q) < G_2(q)$. In all cases, $M$ contains a regular unipotent element $u$ of $G$, and the values of $|C_G(u)|$ are given in \cite[Chap.22]{liebeck_seitz_2012unipotent} (also \cite{Spalt, suzuki1962class} for $G$ of type $^3\!D_4$, $^2\!B_2$). We list upper bounds for these values in Table \ref{regu}. In all but three cases, (\ref{csgbd}) holds; the exceptional case are $G = G_2(4)$, $G_2(9)$ and $^2\!B_2(8)$. For these cases we use the bound $\hbox{fix}(u) \le |C_G(u)|\cdot u(M)/|M|$ (where $u(M)$ is the number of elements of $M$ of the same order as $u$) to obtain the conclusion. 
	\end{proof}

	\begin{table}[h!]
		\caption{Regular unipotent elements}\label{regu}
		\[
		\begin{array}{l|lllllllll}
			\hline
			G & E_8(q) & E_7(q) & E_6^\eps(q) & F_4(q) & G_2(q) & ^2\!F_4(q) & ^2\!G_2(q) & ^3\!D_4(q) & ^2\!B_2(q) \\  
			\hline
			|C_G(u)|\le & 4q^8 & 4q^7 & 3q^6 & 4q^4 & 3q^2 & 4q^2 & 3q & 2q^4 & 2q \\
		\end{array}
		\]
	\end{table}
	
	\subsection{Exotic locals}
	
	\begin{lemma}\label{max5} Theorem \ref{excep} holds in the case where $M$ is as in part (5) of Theorem \ref{maxexcep}.
	\end{lemma}
	
	\begin{proof}
		In this case $M$ is one of the `exotic local' subgroups listed in part (5) of Theorem \ref{maxexcep}, or the Borovik subgroup $(Alt_5 \times Alt_6).2^2 < E_8(q)$. For the exotic locals, it is known (see \cite[Sect. 4.3]{KostTiep}) that each of them possesses an element $g$ of prime order $h+1 = 7$, 13, 13, 31 (for $\bar G = G_2$, $F_4$, $E_6$, $E_8$ resp.) such that $g$ is a regular element of $G$, and hence (\ref{csgbd}) holds. Finally, if $M$ is the Borovik subgroup, the proof of \cite[Lemma 3.5]{LSGeomDed} shows that $M$ contains an element $g$ of order 5 with $E_8$-centralizer $A_4A_4$, and so again (\ref{csgbd}) holds.
	\end{proof}
	
	\subsection{Almost simple subgroups in class ${\mathcal U}$}
	
	To complete the proof of Theorem \ref{excep}, it remains to handle the case where $M$ is as in (6) of Theorem \ref{maxexcep} -- that is, $M$ is in the class ${\mathcal U}$ of almost simple maximal subgroups, not occurring in items (1)-(5). These were subdivided into the classes 
	${\mathcal U}_{p,1}$, ${\mathcal U}_{p,2}$ and ${\mathcal U}_{ASp}$, and we shall deal with these in turn.
	
	\begin{lemma}\label{max6p1} Theorem \ref{excep} holds in the case where $M$ is the class ${\mathcal U}_{p,1}$ of Theorem \ref{maxexcep}(6).
	\end{lemma}
	
	\begin{proof}
		Let $M_0$ be the socle of $M$, so that $M_0 \in Lie(p)$ and $M_0 \not \cong \PSL_2(p^a)$. By \cite[Thms. 1.1,1.2]{Crmem1}, we have $G = E_8(q)$ and $M_0$ is one of the following groups:
		\[
		\PSL_3(3),\, \PSL_3(4),\, \PSU_3(3),\, \PSU_3(4),\, \PSU_3(8),\,\PSU_4(2).
		\]
		The possible restrictions $L(E_8) \downarrow M_0$ are given by \cite[Props. 7.2,8.1,8.2]{Crmem1}. Using these, we choose an element $g \in M_0$ of order as in the following table, and compute $\dim C_{L(E_8)}(g)$:
		\[
		\begin{array}{l|llllll}
			M_0 & \PSL_3(3) & \PSL_3(4) & \PSU_3(3) & \PSU_3(4) & \PSU_3(8) & \PSU_4(2) \\
			\hline
			\hbox{order of }g & 13 & 7 & 7 & 13 & 19 & 5 \\
			\hline
			\dim C_{L(E_8)}(g) & 20 & 38 & 38 & 20 & 14 & 48 
		\end{array}
		\]
		Hence in all cases, (\ref{csgbd}) holds.
	\end{proof}
	
	\begin{lemma}\label{max6p2} Theorem \ref{excep} holds in the case where $M$ is the class ${\mathcal U}_{p,2}$ of Theorem \ref{maxexcep}(6).
	\end{lemma}
	
	\begin{proof}
		In this case the socle of $M$ is $M_0 \cong \PSL_2(p^a)$. By \cite[Thm. 1.1]{Crmem2} and \cite[Thm. 1.1]{craven2022E7} (for $G \ne E_8(q)$), and \cite[Thm. 6]{LS1} (for $G = E_8(q)$), one of the following holds:
		\begin{itemize}
			\item[(i)] $G = E_7(q)$ and $M_0 = \PSL_2(p^a)$ with $p^a = 7$ or 8;
			\item[(ii)] $G = E_8(q)$ and $M_0 = \PSL_2(p^a)$ with $p^a \le (2,p-1)\cdot t(E_8)$, where $t(E_8)$ is a constant defined in terms of the root system of $E_8$; and in fact $t(E_8) = 1312$ by \cite{Law}.
		\end{itemize}
		
		In case (i), the possibilities for the restriction $L(E_7) \downarrow M_0$ are given in \cite[Sect. 6]{craven2022E7}. From this, we see that for $p^a=7$, an element $g \in M_0$ of order 7 has Jordan block sizes $7^{19}$ or $7^{17},5,3^3$, hence by \cite{Lawunip} is in one of the unipotent classes labelled $A_6$ or $E_7(a_5)$, and satisfies $|C_G(g)| \le 6q^{21}$; so (\ref{csgbd}) holds. And for $p^a=8$, an element $g \in M_0$ of order 7 satisfies $\dim C_{L(E_7)}(g) \le 18$, and again (\ref{csgbd}) holds.
		
		Now consider case (ii). Write $q_0 = p^a$, and assume that $q_0 \ne 4,5,9$ (we shall deal with these as alternating groups in the next lemma). Let $g \in M_0$ be an element of order $r:= (q_0+1)/(2,p-1)$, and note that $r \ge 4$. If (\ref{csgbd}) holds then we are done, so suppose that (\ref{csgbd}) fails, so that 
		\begin{equation}\label{cg}
			|C_G(g)| \ge |G:M|^{1/3} \ge \left(|E_8(q)|/|{\mathrm{P\Gamma L}}_2(q_0)|\right)^{1/3}.
		\end{equation}
		Now $r$ must divide the order of a maximal torus of $G$, which implies that $q_0 \le q^8$. From (\ref{cg}) and inspection of semisimple element centralizers in $E_8(q)$ (see \cite{Deriz}), it follows that $C_{\bar G}(g) = E_7T_1$, $D_7T_1$, $E_6A_1T_1$ or $E_6T_2$, where $T_i$ denotes a torus of rank $i$. In the $E_6T_2$ case we have $g \in T_2$ and so $q_0 \le q^2$; but then (\ref{cg}) does not hold. Hence 
		\begin{equation}\label{csgbar}
			C_{\bar G}(g) = E_7T_1, \,D_7T_1 \hbox{ or }E_6A_1T_1,
		\end{equation}
		and also $q_0 \le q$ (as $g \in T_1$). Writing $T_1 = \{T(c) : c \in K^*\}$ (where $K = \bar \F_q$), we see from the restrictions of $L(E_8)$ to the subsystem subgroups $E_7A_1$, $E_6A_1$ and $D_7T_1$ given in \cite[11.2,11.3]{liebeck_seitz_2012unipotent} that the eigenvalues of $T(c)$ on $L(E_8)$ are as follows:
		\[
		\begin{array}{l|l}
			C_{\bar G}(g) & \hbox{eigenvalues of }T(c) \\
			\hline
			E_7T_1 & 1^{134},(c^{\pm 1})^{56},(c^{\pm 2})^1 \\
			D_7T_1 & 1^{92},(c^{\pm 1})^{64},(c^{\pm 2})^{14} \\
			E_6A_1T_1 & 1^{82},(c^{\pm 1})^{56},(c^{\pm 2})^7 \\
			\hline
		\end{array}
		\]
		We have $g = T(c)$ for some $c$, and also $g$ has order $r = (q_0+1)/(2,p-1) \ge 4$. If $r \ge 5$, then the eigenvalues $c,c^{-1},c^2,c^{-2}$ are all distinct, so $g$ stabilizes precisely the same subspaces of $L(E_8)$ as the torus $T_1$. Hence $M_0$ stabilizes the same subspaces as the positive-dimensional subgroup $\langle M_0,T_1\rangle$ of $\bar S$. At this point the proof of \cite[Thm. 6]{LS1} shows that $M$ is as in (2) or (3) of Theorem \ref{maxexcep}, contrary to the fact that $M$ is in the class ${\mathcal U}$.
		
		It remains to consider the case where $r = 4$. Here $q_0 = p = 7$, and we give a different argument. 
		Let $W = C_{L(E_8)}(M_0)$, the fixed point space of $M_0$ on $L(E_8)$. If $\dim W \ge 2$, then $M = \PSL_2(7)$ or $\PGL_2(7)$ fixes a nonzero vector in $W$, and so by \cite[Prop. 4.5]{Crmem2}, $M \le N(\bar M^F)$ 
		for some positive dimensional proper connected subgroup $\bar M$ of $\bar G$, a contradiction (since $M$ is in the class ${\mathcal U}$). The same contradiction applies if $\dim W = 1$, since then the stabilizer in $\bar G$ of $W$ is a positive dimensional subgroup containing $M_0$ which is invariant under all automorphisms normalizing $M_0$. 
		Hence we suppose that 
		\begin{equation}\label{dimw}
			W = C_{L(E_8)}(M_0) =0,
		\end{equation}
		and aim for a contradiction.
		
		We can label the irreducible modules in characteristic 7 for $M_0 = \PSL_2(7)$ as $V_1,V_3,V_5,V_7$, where $V_i$ has dimension $i$ and is the $i-1^{th}$ symmetric power $S^i(V_2)$ of the natural module for $\SL_2(7)$. Let 
		\[
		L(E_8) \downarrow M_0 = V_1^a/V_3^b/V_5^c/V_7^d,
		\]
		meaning that the restriction has the composition factors $V_1,V_3,V_5,V_7$ with multiplicities $a,b,c,d$ respectively. From \cite{AJL}, we have $\dim H^1(M_0,V_i)$ equal to 0 for $i=1,3$ and equal to 1 for $i=5,7$. From elementary considerations (see for example \cite[Prop. 3.6]{Litterick}), it follows that the fixed point space $W$ has dimension at least $a-c-d$, and hence by our assumption (\ref{dimw}), we have
		\begin{equation}\label{amcmd}
			a-c-d \le 0.
		\end{equation}
		We also have 
		\begin{equation}\label{abcd1}
			a+3b+5c+7d = 248,
		\end{equation}
		and (\ref{csgbar}) gives
		\begin{equation}\label{abcd2}
			a+b+c+d = 134,\,92 \hbox{ or }82.
		\end{equation}
		From (\ref{abcd1}) we have 
		\[
		3(a+b+c+d)  = 248 +2(a-c-d)-2d,
		\]
		and so it follows from (\ref{amcmd}) that $a+b+c+d = 82$, and also $(a-c-d,\,d) = (0,1)$ or $(-1,0)$. 
		
		Now consider an element $h \in M_0$ of order 3. Then $C_{\bar G}(h)$ is one of the subsystem subgroups $A_8$, $A_2E_6$, $D_7T_1$, $E_7T_1$, of dimensions 80, 86, 92, 134. Since the dimensions of $C_{V_i}(h)$ for $i=1,3,5,7$ are 1,1,1,3 respectively, it follows that 
		\[
		a+b+c+3d = 80, 86, 92 \hbox{ or } 134.
		\]
		This is not compatible with the equation $a+b+c+d=82$ and the fact that $d = 0$ or 1. This final contradiction completes the proof.
	\end{proof}
	
	\begin{lemma}\label{max6asp} Theorem \ref{excep} holds in the case where $M$ is the class ${\mathcal U}_{ASp'}$ of Theorem \ref{maxexcep}(6).
	\end{lemma}
	
	\begin{proof}
		By \cite{kleidman1988maximal_3D4, kleidman1988maximalG_2, suzuki1962class}, the class ${\mathcal U}_{ASp'}$ is empty when $S$ is of type $^2\!B_2$, $^2\!G_2$ or $^3\!D_4$. And from \cite{malle1991maximal2F4} together with \cite[4.11]{craven2023maximal}, for type $^2\!F_4$ the class ${\mathcal U}_{ASp'}$ is also empty unless $q=8$, in which case it contains a subgroup $\PGL_2(13)$; this subgroup has an element $g$ of order 13 which has centralizer of order 65, so (\ref{csgbd}) holds.
		
		For $G = G_2(q)$, $F_4(q)$ or $E_6^\eps(q)$, the subgroups in class ${\mathcal U}_{ASp'}$ are known (see \cite{kleidman1988maximalG_2, craven2023maximal}), as are their actions on the Lie algebra $L(\bar G)$. In Tables \ref{xg2} - \ref{xe6} we give the socles $M_0$ of these subgroups, together with the order of an element $g \in M_0$ for which $\dim C_{\bar G}(g)$ is as in the last column of the table. In all cases except for $(G,M_0) = (G_2(4), J_2)$, and recalling that semisimple and unipotent centralizer orders can be found in \cite{mizuno, DL, liebeck_seitz_2012unipotent}, it follows that (\ref{csgbd}) holds; in the exceptional case, the permutation character of the action is given in \cite[p.97]{atlas}, from which we see that $\fp(g) = 3$, giving the conclusion.
		
		\begin{table}[h!]
			\caption{Subgroups in class ${\mathcal U}_{ASp'}$ for $G = G_2(q)$}\label{xg2}
			\[
			\begin{array}{|l|l|l|l|}
				\hline
				M_0 & q & o(g) & \dim C_{\bar G}(g) \\
				\hline
				\PSL_2(8) & \ge 5 & 7 & 2 \\
				\PSL_2(13) & \ge 4 & 13 & 2 \\
				\PSL_3(3) & \ge 5 & 7 & 2 \\
				J_1 & 11 & 7 & 2 \\
				J_2 & 4 & 7 & 2 \\
				\hline
			\end{array}
			\]
		\end{table}
		
		\begin{table}[h!]
			\caption{Subgroups in class ${\mathcal U}_{ASp'}$ for $G = F_4(q)$}\label{xf4}
			\[
			\begin{array}{|l|l|l|l|}
				\hline
				M_0 & q & o(g) & \dim C_{\bar G}(g) \\
				\hline
				\PSL_2(8) & \ge 7 & 7 & 8+2\delta_{p,7} \\
				\PSL_2(13) & \ge 7 & 13 & 4 \\
				\PSL_2(17) & \ge 13 & 17 & 4 \\
				\PSL_2(25) & \ge 3 & 13 & 4 \\
				\PSL_2(27) & \ge 13 & 13 & 4 \\
				\PSL_4(3) & 2 & 13 & 4 \\
				^3\!D_4(2) & \ge 3 & 13 & 4 \\
				\hline
			\end{array}
			\]
		\end{table}
		
		\begin{table}[h!]
			\caption{Subgroups in class ${\mathcal U}_{ASp'}$ for $G = E_6^\eps(q)$}\label{xe6}
			\[
			\begin{array}{|l|l|l|l|}
				\hline
				M_0 & q & o(g) & \dim C_{\bar G}(g) \\
				\hline
				\PSL_2(8) & \ge 5 & 7 & 12 \\
				\PSL_2(11) & \ge 9 & 11 & 8 \\
				\PSL_2(13) & \ge 5 & 13 & 6 \\
				\PSL_2(19) & \ge 5 & 19 & 6 \\
				\Omega_7(3) & 2 & 13 & 6 \\
				^2\!F_4(2)' & \ge 3 & 13 & 6 \\
				M_{12} & 5 & 11 & 6 \\
				J_3 & 4 & 19 & 6 \\
				Fi_{22} & 2 & 13 & 6 \\
				\hline
			\end{array}
			\]
		\end{table}
		
		\begin{table}[h!]
			\caption{Subgroups in class ${\mathcal U}_{ASp'}$ for $G = E_7(q)$}\label{xe7}
			\[
			\begin{array}{|l|l|l|l|}
				\hline
				M_0 & q & o(g) & \dim C_{\bar G}(g) \\
				\hline
				Alt_6 & \ge 5 & 5 & 27 \\
				\PSL_2(7) & \ge 5 & 7 & 19 \\
				\PSL_2(13) & \ge 3 & 13 & 13 \\
				\PSL_2(19) & \ge 4 & 19 & 7 \\
				\PSL_2(27) & \ge 13 & 13 & \le 13 \\
				\PSL_2(29) & \ge 4 & 29 & 7 \\
				\PSL_2(37) & \ge 3 & 37 & 7 \\
				\PSU_3(3) & \ge 7 & 7 & \le 31 \\
				\PSU_3(8) & \ge 3 & 19\,(p\ne 19) & 7 \\
				& & 7\,(p=19) & 19 \\
				M_{12},M_{22},HS & 5 & 11 & 13 \\
				Ru & 5 & 13 & 13 \\
				J_3 & 4 & 19 & 6 \\
				Fi_{22} & 2 & 13 & 6 \\
				\hline
			\end{array}
			\]
		\end{table}

		\begin{table}[h!]
			\caption{Candidates for class ${\mathcal U}_{ASp'}$ for $G = E_8(q)$}\label{xe8}
			\[
			\begin{array}{|l|l|l|l|}
				\hline
				M_0 & p & o(g) & \dim C_{\bar G}(g) \\
				\hline
				\PSL_2(7) & \hbox{any} & 7 & 38 \\
				\PSL_2(8) & p\ne 7 & 7 & \le 40 \\
				&  p=7 & 9 & 28 \\
				\PSL_2(11) & \hbox{any} & 11 & 28 \\
				\PSL_2(16) & p\ne 17 & 17 & 16 \\
				&  p=17 & 5 & \le 52 \\
				\PSL_2(17) & \hbox{any} & 17 & 16 \\
				\PSL_2(19) & \hbox{any} & 19 & 14 \\
				\PSL_2(25) & p\ne 13 & 13 & 20 \\
				&  p=13 & 5 & 48 \\
				\PSL_2(29) & \hbox{any} & 29 & \le 10 \\
				\PSL_2(31) & \hbox{any} & 31 &  8 \\
				\PSL_2(32) & \hbox{any} & 31 &  8 \\
				\PSL_2(41) & \hbox{any} & 41 &  8 \\
				\PSL_2(49) & \hbox{any} & 7 &  38 \\
				\PSL_2(61) & \hbox{any} & 61 &  8 \\
				\PSL_3(3) & p\ne 13 & 13 & 20 \\
				&  p=13 & 8 & \le 34 \\
				\PSL_3(5) & \hbox{any} & 31 &  8 \\
				\PSL_4(5) & 2 & 31 &  8 \\
				M_{11} & 3,11 & 5 & 48 \\
				J_3 & 2 & 5 & 48 \\
				Th & 3 & 31 & 8 \\
				\hline
			\end{array}
			\]
		\end{table}

		Next consider $G = E_7(q)$. In this case, the subgroups in class ${\mathcal U}_{ASp'}$ are given in \cite{craven2022E7}, apart from two cases. In these cases, $M_0 = \PSL_2(7)$ or $Alt_6$, and the possible actions on $L(E_7)$ are given in \cite[Sect.6.1]{craven2022E7}, \cite[6.1]{Cralt} respectively. As in the previous cases, we give in Table \ref{xe7} the socles $M_0$ of the subgroups in ${\mathcal U}_{ASp'}$, together with the order of an element $g \in M_0$ for which $\dim C_{\bar G}(g)$ is as in the last column of the table; in all cases (\ref{csgbd}) holds.
		
		Finally, suppose that $S = E_8(q)$. In this case, the subgroups in class ${\mathcal U}_{ASp'}$ are not known, but all candidates can be found in \cite{Litterick}; also in \cite[Chap.6]{Litterick},   with the exception of $(M_0,p) = (Alt_6,3)$, the possible Brauer characters of the representations on $L(E_8)$ of the candidates are given. We shall give a separate argument for alternating groups in the next paragraph, and for the non-alternating groups we give in Table \ref{xe8} the possible socles of subgroups in ${\mathcal U}_{ASp'}$, together with elements $g$ and values $\dim C_{\bar G}(g)$ for which (\ref{csgbd}) holds.
		
		To complete the proof we need to deal with alternating socles in ${\mathcal U}_{ASp'}$. By \cite{Cralt}, the possible socles are $M_0 = Alt_6$ or $Alt_7$. For $M_0 = Alt_7$, the possibilities for the action of $M_0$ on $L(E_8)$ are given in \cite[Sect.7]{Cralt}, from which we see that an element $g \in M_0$ of order 7 satisfies $\dim C_{E_8}(g) \le 40$, and (\ref{csgbd}) holds. Similarly, if $M_0 = Alt_6$ and $p\ne 3$, Section 6 of \cite{Cralt} gives the possibilities for $L(E_8) \downarrow M_0$, and an element $g \in M_0$ of order 5 satisfies $\dim C_{E_8}(g) \le 52$.
		
		It remains to consider $M_0 = Alt_6$ with $p=3$. In this case, the possible restrictions 
		$L(E_8) \downarrow M_0$ are not analysed in either \cite{Cralt} or \cite{Litterick}. We argue as follows that there is an element $g \in M_0$ that satisfies (\ref{csgbd}). Suppose this is not the case, and let $x,y \in M_0$ be elements of orders 5 and 4, respectively. From the list of centralizers and traces of elements of such orders that can be found in \cite[Table 4]{CG}, we see that the assumption that $x,y$ do not satisfy (\ref{csgbd}) forces $C_{E_8}(x)$ and $C_{E_8}(y)$ both to be among the types $E_7T_1$, $D_7T_1$, $A_1E_6T_1$. Moreover, if $\chi$ denotes the Brauer character of $M_0$ on $L(E_8)$, then for the respective centralizers, we have 
		\[
		\begin{array}{l}
			\chi(x) = 78+55\tau,\,28+50\tau \hbox{ or } 28+25\tau, \\
			\chi(y) = 132,\,64 \hbox{ or }28,
		\end{array}
		\]
		where $\tau = \frac{1}{2}(1+\sqrt{5})$. 
		From \cite{atlas_breuer_characters}, we see that the 3-modular irreducibles for $Alt_6$ are $V_1$, $V_3$, $V_3^*$, $V_4$ and $V_9$, where $V_i$ has dimension $i$. Let their multiplicities as composition factors of $L(E_8)\downarrow M_0$ be $a,b,c,d,e$ respectively. Then the values of $\chi(1)$, $\chi(x)$ and $\chi(y)$ give the equations
		\begin{itemize}
			\item[(1)] $a+3b+3c+4d+9e = 248$,
			\item[(2)] $a+b\tau +c\tau -4d-9e = 78+55\tau,\,28+50\tau \hbox{ or } 28+25\tau$,
			\item[(3)] $a+b+c-2d+e = 132,\,64 \hbox{ or }28$.
		\end{itemize}
		From (2) we have $b+c = 55$, 50 or 25, and $a-4d-9e = 78$ or 28, and adding (1) to the latter gives
		\[
		2a+3(b+c) = 326 \hbox{ or }276.
		\]
		Hence $b+c$ is even, so $b+c=50$ and so $a=63$. Also $a-4d-9e = 28$, which gives $4d+9e =35$. Also by (3) we have $2d-e = -19$, 49 or 85. The last two equations do not have integral solutions for $d,e$. This final contradiction completes the proof.
	\end{proof}

	%%%%%%%%%%%%%%%
	
	\section{Classical groups: preliminaries} \label{class1}

	It remains to prove \Cref{t:main_simple} for classical groups. In this section, 
	we establish a number of preliminary results required for the proof, which will be presented in \Cref{sec:geometric_classes,sec:class_S}.
	Throughout this section, $p$ is a prime number and $K = \overline{\F_p}$,  an algebraic closure of $\F_p$.  
	
	\subsection{Some representation theory}
	
	We begin with some representation theoretic lemmas.

	\begin{lemma}
		\label{l:semiregular_eigenvalues}
		Let $H$ be a finite group and let $x\in H$ be a $p'$-element with $\gen x \trianglelefteq H$. Then, the action of $H$ on $\gen x \sm \{1\}$ is equivalent to the action of $H$ on the set of nontrivial irreducible $K\gen x$-modules. In particular, if $H \le \GL_n(K)$ then the action of $H$ on the set of the nontrivial eigenspaces of $x$ is equivalent to the action on a subset of $\gen x \sm \{1\}$. 
	\end{lemma}

	\begin{proof}
		Let $\zeta$ be an element of $K^\times$ of order $|x|$. The map $x^i \mapsto \phi_{\zeta^i}$ gives the desired equivalence of $H$-actions, where $\phi_{\zeta^i}$ is the homomorphism $\gen x \to K^\times$ given by $x\mapsto \zeta^i$.  The last part (``in particular...'') follows immediately. 
	\end{proof}
	
	Typically, we will apply the lemma when $H$  induces a semiregular group on $\gen x \sm \{1\}$, say of order $t$, in which case $H$ induces a semiregular group of order $t$ on the nontrivial eigenspaces of $x$. In particular, in this case the dimensions of the nontrivial eigenspaces of $x$ occur with multiplicities that are multiples of $t$.
	The following variant will sometimes be useful.
	
	\begin{lemma}
		\label{l:prime_order_eigenvalues}
		Let $H\le \GL_n(K)$ and let $x\in H$ be a $p'$-element of prime order. Then $N_H(\gen x)/C_H(x)$ has all orbits of size $|N_H(\gen x)/C_H(x)|=|x^H\cap \gen{x}|$ on the set of nontrivial eigenspaces of $x$.
	\end{lemma}
	
	\begin{proof}
		Since $x$ has prime order, $N_H(\gen x)/C_H(x)$ acts semiregularly on $\gen x \sm \{1\}$. Also, the map $N_H(\gen x)/C_H(x)\to x^H \cap \gen x$ sending $C_H(x)g\mapsto x^g$ is a well-defined bijection. The statement now follows from \Cref{l:semiregular_eigenvalues}. 
	\end{proof}

	Let us also recall the following result due to Scott \cite{scott1977matrices}.
	
	\begin{lemma}
		\label{l:scott}
		Assume that $\gen{x,y}\le \GL_n(K)=\GL(V)$ is irreducible. Then, there exists $z\in\{x,y,xy\}$ such that $\dim(C_V(z))\le n/3$.
	\end{lemma}

	\begin{proof}
		Scott's lemma \cite[Thm. 1]{scott1977matrices} gives $\dim(C_V(x))+\dim(C_V(y))+\dim(C_V(xy))\le n$, from which the conclusion follows.
	\end{proof}
	
	The following lemma uses Green correspondence. We will usually apply it to the case where $S$ is a cyclic torus in a group of Lie type $G$ containing a cyclic Sylow $p$-subgroup $P$.	
	
	\begin{lemma}
		\label{l:green_correspondence}
		Let $G$ be a finite group and let $W$ be an indecomposable $KG$-module. Let $P$ be a cyclic Sylow $p$-subgroup of $G$, let $S$ be an abelian subgroup of $G$ containing $P$, and assume that 
		\begin{itemize}
			\item[($\star$)] $S\trianglelefteq N_G(P)=N_G(P_0)$ for every $1\neq P_0\le P$.
		\end{itemize}
		Then, $W\downarrow P = W_0\oplus U$ where all indecomposable submodules of $U$ have dimension $|P|$, and $W_0$ is the direct sum of at most $|N_G(P):S|$ indecomposable modules of the same dimension.
	\end{lemma}
	
	\begin{proof}
		By \cite[Corollary 3.6.10]{benson1998representations}, every $KG$-module is projective relative to $P$, i.e., it is a direct summand of some module induced from $P$ (see \cite[Proposition 3.6.4]{benson1998representations}). Moreover, assumption ($\star$) implies that 
		\begin{equation}
			\label{eq:green}
			P^g \cap N_G(P)=1 \text{ for every } g\not\in N_G(P).
		\end{equation}
		Now we apply Green correspondence. Let $D$ be a vertex of $W$, see \cite[Definition 3.10.1]{benson1998representations}; up to conjugacy, $D\le P$ by \cite[Proposition 3.6.9]{benson1998representations}.
		We may assume that $W$ is not projective, as otherwise $W\downarrow P$ is projective 
		and the lemma holds with $W_0=0$.
		In particular, $D\neq 1$.
		
		By \cite[Theorem 3.12.2(i)]{benson1998representations}, applied with $H=N_G(P)=N_G(D)$, we have $W\downarrow N_G(P) = W_1\oplus W_2$, where $W_1$ is indecomposable and $W_2$ is projective (note that, with notation as in \cite[Theorem 3.12.2(i)]{benson1998representations}, $\ca Y=\{1\}$ by \eqref{eq:green}).
		
		Setting $U:=W_2 \downarrow P$, we see that $U$ is projective, and so, since $P$ is cyclic, indecomposable submodules of $U$ have dimension $|P|$.
		Moreover, $W_1$ is projective relative to $S$, and so by Clifford theory (see \cite[Section 3.13]{benson1998representations}), $W_1\downarrow S$ is a direct sum of indecomposable modules, conjugate under $N_G(P)/S$. When restricted to $P$, each such module remains indecomposable,  and so setting $W_0:=W_1\downarrow P$, the lemma is proved.
	\end{proof}
	
	\subsection{Some optimization}
	
	The following few lemmas will be used to bound the dimension $d$ of the centralizer of certain semisimple elements $g\in \GL_n(K)=\GL(V)$. For example, assume that $\dim(C_V(g))\le n/3$ and that the dimensions of the nontrivial eigenspaces of $g$ occur with multiplicity at least $4$ (we will use \Cref{l:semiregular_eigenvalues} in order to produce such elements).  Then \Cref{l:optimization_easier}, below, applied with $A=n/3$ and $B=4$, asserts that $d\le n^2/4$; equality is attained when $C_V(g)=0$ and $g$ has four nontrivial eigenspaces of dimension $n/4$. As another example, assuming that $\dim(C_V(g))\le n/2$ and that the dimensions of the nontrivial eigenspaces of $g$ occur with multiplicity at least $4$, we have $d\le 5n^2/16$, with equality attained when $C_V(g)$ has dimension $n/2$ and there are four nontrivial eigenspaces of dimension $n/8$.
	
	\Cref{l:optimization}, below, takes into account also the number of nontrivial eigenvalues of $g$ (named $R(\underline a)$), which will be useful on some occasions.  \Cref{l:optimization_orthogonal} is a useful variant for symplectic and orthogonal groups.
	
	\begin{lemma}
		\label{l:optimization}
		Let $n$ and $B$ be positive integers with $B<n$, let $1\le A<n$ be a real number, and assume $A+B\le n$.
		Let $\ca A\subseteq \R^n$ be the set of $\underline a =(a_1, \ldots, a_n)$ such that for every $i$, $a_i=0$ or $a_i\ge 1$; $a_1\le A$; $\sum_i a_i = n$; for each $i\ge 2$, $a_i$ is either zero or equal to $a_t$ for at least $B$ values of $t\ge 2$. Denoting by $R(\underline a)$ the number of $i\ge 2$ such that $a_i$ is nonzero, we have
		\[
		\max_{\underline a\in \ca A} \left(R(\underline a)+ \sum_{i=1}^n a_i^2\right) = \begin{cases}
			B+ A^2 + (n-A)^2/B & \text{if $A(B+1) \ge 2n$} \\
			B+ n^2/B & \text{if $A(B+1) \le 2n$}
		\end{cases} 
		\]
	\end{lemma}
	
	\begin{proof}
		Let $\underline a =(a_1, \ldots, a_n)\in \ca A$ and denote  $f(\underline a) =R(\underline a)+ \sum_{i=1}^n a_i^2$. We claim that if $R(\underline a)> B$, then there exists some $(b_1, \ldots, b_n)\in \ca A$ with $R(\underline b)<R(\underline a)$ and $f(\underline b)> f(\underline a)$. In particular, this will allow us to assume $R(\underline a)=B$.

		Without loss, $0\neq a_2 = \cdots = a_{r+1}$ for some $r\ge B$, and $a_t\neq a_2$ for $t> r+1$. If $r>B$, then replace $a_i$ by $ra_i/B$ for $i=2, \ldots, B+1$, and replace $a_i$ by $0$ for $i \in \{B+2, \ldots, r+1\}$. Since 
		\[
		r + \sum_{i=2}^{r+1} a_i^2 = r+ ra_2^2 < B+ r^2 a_2^2/B =  B+ \sum_{i=2}^{B+1} (ra_i/B)^2
		\]
		the claim is proved in this case. Assume then $r=B$. If there exists $t>B+1$ with $a_t\neq 0$, then without loss $0\neq a_{B+2}=\cdots a_{2B+1}$ and $a_\ell \neq a_2, a_{B+2}$ for $\ell > 2B+1$. Now for $i=2, \ldots, B+1$ replace $a_i$ by $a_i+a_{B+i}$, and for $i=B+2, \ldots, 2B+1$, replace $a_i$ by $0$. Then 
		\[
		2B + \sum_{i=2}^{2B+1} a_i^2 < B+ \sum_{i=2}^{B+1} (a_i+a_{B+i})^2
		\]
		and so the claim holds in this case also. 
		
		Therefore the initial claim is proved, so assume  $R(\underline a)=B$, and without loss, $a_t=0$ for every $t>B+1$. We may then simply maximize  $a_1^2 + Ba_2^2$ subject to $a_1 + Ba_2 = n$, $1\le a_1 \le A$ or $a_1=0$, and $a_2\ge 1$ or $a_2=0$. This is a parabola in $a_1$ and the maximum is given by either $a_1=0$ or $a_1=A$, which give the values $f(\underline a)=B+ n^2/B$ and $f(\underline a)=B+ A^2 + (n-A)^2/B$ respectively. Noting that $B+A^2 + (n-A)^2/B\ge B+ n^2/B$ if and only if $A(B+1) \ge 2n$, the proof is complete.
	\end{proof}
	
	\begin{lemma}
		\label{l:optimization_easier} 
		With notation as in \Cref{l:optimization}, we have
		\[
		\max_{\underline a \in \ca A} \sum_{i=1}^n a_i^2 = \begin{cases}
			A^2 + (n-A)^2/B & \text{if $A(B+1) \ge 2n$} \\
			n^2/B & \text{if $A(B+1) \le 2n$}
		\end{cases}
		\]
	\end{lemma}
	
	\begin{proof}
		By the proof of the previous lemma, the maximum of $R(\underline a)+\sum_{i=1}^n a_i^2$ is attained when $R(\underline a)$ is as small as possible, and so the same value gives the maximum of $\sum_{i=1}^n a_i^2$.
	\end{proof}

	\begin{lemma}
		\label{l:optimization_orthogonal} 
		With notation as in \Cref{l:optimization}, we have
		\[
		\max_{\underline a \in \ca A} \left(R(\underline a)+\sum_{i=1}^n a_i^2 \pm a_1\right) = \begin{cases}
			B+ A^2 \pm A + (n-A)^2/B & \text{if $A(B+1)\pm B \ge 2n$} \\
			B+ n^2/B & \text{if $A(B+1)\pm B \le 2n$}
		\end{cases}
		\]
	\end{lemma}
	
	\begin{proof}
		We follow the proof of \Cref{l:optimization}, and in the same way we see that the maximum occurs when $R(\underline a)=B$. Then we need to maximize $a_1^2\pm a_1 + Ba_2^2$	subject to the same constraints, and the maximum is attained for $a_1=0$ or $a_1=A$.
	\end{proof}

	In the following lemma, for a partition $a_1  \ge \cdots \ge a_n$ of $n$, we denote by $a'_1  \ge \cdots \ge a'_n$ the transpose partition. The lemma will be applied to bound the dimension $d$ of the centralizer of certain unipotent elements $g\in \GL_n(K)=\GL(V)$, which will be produced thanks to \Cref{l:green_correspondence}.  (Recall that if $a_i$ are the sizes of the Jordan blocks, then $d=\sum (a'_i)^2$ where $d$ is the  dimension of the centralizer of $g$ in $\GL_n(K)$.)
	
	\begin{lemma}
		\label{l:unipotent_centralizer}
		Let $B\ge 2,C\ge 0$, and $n$ be integers, with $B,C\le n$,  and let $a_1  \ge \cdots \ge a_n$ be a partition of $n$. Assume that there exists a subset $I$ of $\{1,\ldots, n\}$ such that $\sum_{i\in I} a_i\le  C$, and $a_j\ge B$ for every $j\not\in I$. Then 
		\[
		\sum_{i=1}^n (a'_i)^2 \le  \frac{n^2}{B} + C^2\left(1-\frac{1}{B}\right).
		\]
	\end{lemma}
	
	\begin{proof}
		The case $C=n$ is trivial, so assume $C<n$. Fix any $C'\le C$, and assume $\sum_{i\in I} a_i = C'$. Assume that $a_i>1$ for some $i\in I$; then consider the partition $b_1\ge \cdots \ge b_n$ obtained by replacing $a_i$ by a part of length $a_i-1$, and adding one part of length $1$. In the transpose partition, this means that $b'_1=a'_1+1$, and $b'_j = a'_j-1$ for some $j>1$. Since
		\[
		(a'_1)^2 + (a'_j)^2 < (a'_1+1)^2 + (a'_j-1)^2,
		\]
		then iterating this procedure, we may assume that $a_i\le 1$ for every $i\in I$. 
		Similarly, we may assume that either $C'=C$, or $a_j=B$ for every $j\not\in I$.
		
		Next we look at the transpose partition; we take $a'_1, \ldots, a'_n$ real numbers and we maximize $\sum_{i=1}^n (a'_i)^2$ under the given constraints. 
		
		Assume first $C'=C$. The condition $a_j\ge B$ for every $j\not\in I$ is equivalent to $a'_1 -C= a'_2 = \cdots = a'_B$. (Note that $I$ is a proper subset of $\{1, \ldots, n\}$ since $C< n$.)
		Assume that there exists $j>B$ with $0<a'_j< a'_1-C$; without loss, $j=B+1$. Note that
		\[
		\sum_{i=1}^{B+1} (a'_i)^2 < \sum_{i=1}^B (a'_i + a_{B+1}/B)^2.
		\]
		In particular, we may assume that $a'_j=0$ for every $j>B$. This choice gives 
		\begin{align*}
			\sum_{i=1}^n (a'_i)^2 &= (\frac{n-C}{B} + C)^2 + (B-1)\frac{(n-C)^2}{B^2}\\
			&=\frac{(n-C)^2}{B} + C^2 + \frac{2C(n-C)}{B} \\
			&=\frac{n^2}{B} + C^2\left(1-\frac{1}{B}\right),
		\end{align*}
		which concludes the proof in this case.
		
		Assume finally $a_j=B$ for every $j\not\in I$, so up to reordering $a'_1 -C' = a'_2 =\cdots a'_B = (n-C')/B$ and $a_j=0$ for $j> B$. Let $\eps = (C-C')/B$, so that $\eps + C' + (B-1)\eps = C$. Then replace each $a'_j$, $j>1$, by $b'_j = a'_j-\eps$, and replace $a'_1$ by $b'_1 = a'_1 + (B-1)\eps = b'_2 + C$. As above, $\sum (a'_i)^2 \le \sum (b'_i)^2$, which reduces to the case $C'=C$, addressed in the previous paragraph.
	\end{proof}

	\subsection{Centralizer bounds}
	
	Next we need some bounds for the dimension and size of centralizers in classical groups. We begin by recording a calculation for the order of the groups.
	
	\begin{lemma} 
		\label{l:order_classical_groups}
		The following estimates hold for every $n\ge 1$ and every prime power $q$.
		
		\begin{align*}
			q^{n^2-2} < \frac{9q^{n^2}}{32} &< |\GL_n(q)|<q^{n^2} \\
			q^{n^2-2} < \frac{9q^{n^2-1}}{16} &< |\SL_n(q)|<q^{n^2-1} \\
			q^{n^2} &< |\GU_n(q)|\le 1.5 q^{n^2} < q^{n^2+1} \\
			q^{n^2-2} < \frac{q^{n^2-1}}{1.5} &< |\SU_n(q)|<  q^{n^2-1} \\
			q^{2n^2+n-1} < \frac{9q^{2n^2+n}}{16} &< |\Sp_{2n}(q)| =|\SO_{2n+1}(q)|<q^{2n^2+n} \\
			q^{2n^2-n-1} < 2\frac{9q^{2n^2-n}}{32} &< |\Or^+_{2n}(q)|<2q^{2n^2-n} \\
			q^{2n^2-n} < 2\frac{9q^{2n^2-n}}{16} &< |\Or^-_{2n}(q)|\le 2 q^{2n^2-n} \le q^{2n^2-n+1} \text{ if $(n,q)\neq (1,2)$}\\
		\end{align*}
		
	\end{lemma}
	
	\begin{proof}
		The estimates follow from \cite[Lemma 5.1(i, iii)]{guralnick2020larsen_tiep_character}, and from the order formulas for the groups under consideration. We just note that the occurrence of the factor $1.5$ follows from use of the bound $q+1\le 1.5 q$.
	\end{proof}

	The next lemma records, in a special case, the fact that the dimension of the centralizer of an element in a symplectic or orthogonal group is approximately half the dimension of the centralizer in the general linear group.

	\begin{lemma}
		\label{l:compare_dim_centralizer_GL_Sp}
		Let $g\in Y:=\Sp_n(K)$ or $Y:=\Or_n(K)$, and let $d'$ and $d$ be the dimension of the centralizer of $g$ in $Y$ and $\GL_n(K)$, respectively. 
		\begin{itemize}
			\item[(i)] Assume $Y=\Sp_n(K)$. If $g$ is unipotent and $p$ is odd, then $d'\le d/2 + R/2$ where $R$ is the number of Jordan blocks of $g$. If $g$ is semisimple, without eigenvalue $-1$ if $p$ is odd, then $d'= d/2 + m/2$ where $m$ is the dimension of the $1$-eigenspace of $g$.
			\item[(ii)]  Assume $Y=\Or_n(K)$. If $g$ is unipotent and $p$ is odd, then $d'\le d/2$. If $g$ is semisimple, without eigenvalue $-1$ if $p$ is odd, then $d'= d/2 - m/2$ where $m$ is the dimension of the $1$-eigenspace of $g$.
		\end{itemize}
	\end{lemma}
	
	\begin{proof}
		If $g$ is unipotent, then by assumption  $p$ is odd and the bounds follow from \cite[Theorem 3.1]{liebeck_seitz_2012unipotent}. Assume then $g$ is semisimple; denote by $V_\lambda$ the $\lambda$-eigenspace of $g$ and by $m_\lambda$ its dimension.  If $\lambda \neq 1$ then $m_\lambda = m_{\lambda^{-1}}$ and the centralizer of $g$ restricted to $V_\lambda \oplus V_{\lambda^{-1}}$ is isomorphic to $\GL_{m_\lambda}(K)$. On the other hand, the centralizer of the restriction of $g$ to $V_1$ is isomorphic to $\Sp_{m_1}(K)$ (if $Y=\Sp_n(K)$) or $\Or_{m_1}(K)$ (if $Y=\Or_n(K)$). Since $\dim(\Sp_{m_1}(K))=m_1^2/2 + m_1/2$ and
		$\dim(\Or_{m_1}(K))=m_1^2/2 - m_1/2$, the conclusion follows.
	\end{proof}

	In part (iii) of the next lemma, a rational canonical block of $g$ refers to an indecomposable $\F_q\gen g$-submodule of the natural module. The natural module decomposes as a direct sum of such submodules; the number of summands is called the number of rational canonical blocks of $g$. Note that (iii) is a generalization of (i), but we prefer to state (i) separately for later reference.

	\begin{lemma}
		\label{l:order_centralizers}
		Let $q$ be a power of $p$, let $G$ be $\GU_n(q)$ or $\Sp_n(q)$ or $\Or^\varepsilon_n(q)$, let $g\in G$ and let $d'$ be the dimension of the centralizer of $g$ in $\GL_n(K)$ or $\Sp_n(K)$ or $\Or_n(K)$.
		\begin{itemize}
			\item[(i)] Assume $g$ is unipotent, and let $R$ be the number of Jordan blocks of $g$. Then $|C_G(g)|<2^Rq^{d'}\le q^{d'+R}$.
			\item[(ii)] Assume $g$ is semisimple, and let $E$ (resp. $E_2$) be the number of distinct irreducible factors (resp. irreducible factors of degree at least $2$) of the characteristic polynomial of $g$. If $G=\Sp_n(q)$ then $|C_G(g)|< 2^{E_2}q^{d'}\le q^{d'+E_2}$; if $G=\Or^\varepsilon_n(q)$ then $|C_G(g)|< 2^{E_2+2}q^{d'}\le q^{d'+E_2+2}$; and if $G=\GU_n(q)$ then $|C_G(g)|< 2^E q^{d'} \le q^{d'+E}$.
			\item[(iii)] Let $R$ denote the number of rational canonical blocks of $g$. Then  $|C_G(g)|< 2^Rq^{d'}\le q^{d'+R}$.
			\item[(iv)] Assume $g$ is regular. Then $|C_G(g)|<h^{(2,q-1)}(q+1)^r$ where $r$ is the untwisted Lie rank of $G$, and $h=1$ if $G=\GU_n(q)$, $h=2$ if $G=\Sp_n(q)$, $h=4$ if $G=\Or^\varepsilon_n(q)$.
		\end{itemize}
	\end{lemma}
	
	\begin{proof}
		(i) This can be found in \cite{liebeck_seitz_2012unipotent}, as we proceed to explain. We address the case $G=\Sp_n(q)$; the case $G=\Or_n(q)$ is identical and the case $G=\GU_n(q)$ is easier. If $p$ is odd, then \cite[Theorem 7.1]{liebeck_seitz_2012unipotent} tells us that $C_G(g)=U\rtimes H$, where $U$ is unipotent and  $H$ is a direct product of at most $R$ groups, each of which is a symplectic or orthogonal group over $\F_q$. By \Cref{l:order_classical_groups}, each of these groups has order at most $2q^m$, where $m$ is the dimension of the corresponding algebraic group; the claimed bound follows. 
		Assume then $p=2$, where there are some complications in the structure of $C_G(g)$. Following \cite{liebeck_seitz_2012unipotent}, we may write	
		\[
		\F_q^n \downarrow \gen g = \bigoplus_{i=1}^\ell W(m_i)^{a_i}\oplus \bigoplus_{j=1}^r V(2k_j)^{b_j}
		\]
		where each $W(m_i)$ is the sum of two Jordan blocks of size $m_i$, and each $V(2k_j)$ is a Jordan block of size $2k_j$. (In particular, $n=2\sum_i a_im_i + 2\sum_j b_jk_j$.) Here the $m_i$ are pairwise distinct, and the same for the $k_j$; assume moreover the $k_j$ are in decreasing order.
		By \cite[Theorem 7.2]{liebeck_seitz_2012unipotent} we have $C_G(g)=U.(H\times K)$ where
		\begin{itemize}
			\item $U$ is unipotent;
			\item $H$ is a direct product of $\ell$ groups, the $i$-th of which is a symplectic or orthogonal group of rank $a_i$ over $\F_q$;
			\item $K\cong C_2^{t+\delta}$, where $t$ is the number of $j$ such that $k_j-k_{j+1}\ge 2$, and $\delta \in \{0,1\}$, with $\delta =1$ only if $r\neq 0$.
		\end{itemize}
		Letting $b$ be the number of $i$ such that $a_i=1$, by \Cref{l:order_classical_groups} we have $|C_G(g)|<2^{\ell + b + t + \delta} q^{d'}$. (The relevance of the parameter $b$ stands in the fact that, for $q=2$, $|\Or^-_2(q)|=2(q+1)$ is not less than $q^2$, cf. the last equation in \Cref{l:order_classical_groups}; it is instead less than  $2q^2$.) By the definitions we have  $\ell + b \le 2 \sum a_i$ and $t+\delta \le \sum b_j$, so $\ell + b + t + \delta \le 2\sum a_i + \sum b_j =R$,  which implies $|C_G(g)|<2^Rq^{d'}$, as desired for part (i).

		We now address (ii) and (iii) largely simultaneously. We assume $G=\Sp_n(q)$; the other cases are similar, as we point out at the end of the proof. For a monic polynomial $f\in \F_q[X]$ of degree $m$, denote $f^*(X) = X^m f(1/X)/f(0)$. Writing $V:=\F_q^{n}$, we have
		\[
		V\downarrow \langle g \rangle =\bigoplus_f V_f
		\]
		where $f$ is an irreducible monic polynomial in $\F_q[X]$ and $V_f$ is the generalized $f$-eigenspace of $g$, and moreover $W_{f,f^*}:=V_f + V_{f^*}$ is nondegenerate. Let $g_{f,f^*}$ be the restriction of $g$ to $W_{f,f*}$ and let $S_{f,f^*}=\Sp(W_{f,f^*})$, so
		\[
		C_G(g)=\prod_{f,f^*}C_{S_{f,f^*}}(g_{f,f^*}).
		\]
		Let now $g_f$ be the restriction of $g$ to $V_f$. Assume $f$ has degree $m$ and $V_f$ has dimension $k$. Note that $g_f$ acts $\F_{q^m}$-linearly; let $u_f$ be the unipotent part of $g_f$ as an element of $\GL_{k/m}(q^m)$. Write $S=S_{f,f^*}$, $g'=g_{f,f^*}$,  $u=u_f$, for ease of notation. Then $C_S(g')\cong C_L(u)$, where $L= \GU_{k/m}(q^{m/2})$ if $f=f^*$ and $\deg(f)>1$; $L=\GL_{k/m}(q^m)$ if $f\neq f^*$; $L=\Sp_k(q)$ if $f=f^*$ and $\deg(f)=1$ (i.e., $f=X\pm 1$). Assume next $g$ is as in (ii). Then $g$ is semisimple, and so $u=1$ and $C_L(u)=L$.  In particular, we see that the bound follows from \Cref{l:order_classical_groups}, noting that the number of $f$ with $f=f^*$ and $f\neq X+1$ and $V_f\neq 0$ is at most $E_2$ (since each such $f$ has degree at least $2$). Assume now $g$ is as in (iii), so $u$ may be nontrivial. The bound follows from (i), noting that the number of rational canonical blocks of $g'$ is at least the number of Jordan blocks of $u$. (More precisely, if $f=f^*$ then it is equal to the number of Jordan blocks of $u$, and otherwise it is twice that number.)

		The case $G=\Or^\varepsilon_n(q)$ is essentially the same; in (ii), we pay an additional factor $4\le q^2$
		to account for the case $f=X\pm 1$, in which case the centralizer of the restriction to $V_f$ is an orthogonal group rather than a symplectic one. In the case $G=\GU_n(q)$ we replace $f^*(X)$ by $f^\dagger(X)=X^n\bar f(1/X)/\bar f(0)$ where $x \mapsto \bar x = x^q$ is the involutory automorphism of $\F_{q^2}$. We have a similar decomposition of $V=\F_{q^2}^n$ into nondegenerate subspaces, the only difference being that if $f=f^\dagger$ and $\deg(f)=1$ (i.e., $f=X-\lambda$ with $\lambda^{q+1}=1$) then the centralizer of the restriction of $g$ to $V_f$ is $\GU_k(q)$, and so in (ii) we pay a factor $2\le q$ for each such $f$.
		
		(iv) We use the same notation as in the proof of (ii) and (iii). We have that $g$ regular if and only if $g_{f,f^*}$ is regular for every $f,f^*$, which is equivalent to $u=u_f$ being regular. 
		If $L=\GL_{k/m}(q^m)$ (resp. $\GU_{k/m}(q^{m/2})$) then $|C_L(u)|<(q+1)^{k}$ (resp. $|C_L(u)|<(q+1)^{k/2}$). If $L=\Or^\varepsilon_k(q)$ or $\Sp_k(q)$, then $u$ has at most $h/2$ Jordan blocks, so by (i) we deduce that $|C_L(u)|<2^{h/2}(q+1)^{k/2}=h(q+1)^{k/2}$. There are at most $(2,q-1)$ choices of $f$ with $L=\Or^\varepsilon _k(q)$ or $\Sp_k(q)$ (namely $f=X\pm 1$), and so we deduce $|C_G(g)|<h^{(2,q-1)}(q+1)^{n/2}$, as desired. If $G=\GU_n(q)$, we have $|C_G(g)|<(q+1)^n$ since $L$ is never $\Or^\varepsilon_k(q)$ or $\Sp_k(q)$.
	\end{proof}

	\subsection{Notation and remarks on \Cref{table:new_unique,table:sporadics}}
	\label{subsec:notation}
	
	Let $S$ be a simple group of Lie type over a field $\F_r$, or a sporadic group. In view of various isomorphisms, we  assume $m\ge 3$ if $S=\PSU_m(r)$, $m\ge 4$ if $S=\PSp_m(r)$, $m\ge 7$ if $S=\POm^\varepsilon_m(q)$. Moreover, in order to avoid some technical issues, we assume $S\neq \PSL_2(r), \PSL^\pm_3(r)$.
    
    In \Cref{table:new_unique,table:sporadics} we define certain elements $x_i \in S$, and we now explain our notation for these elements. The notation in \Cref{table:sporadics} is as in the ATLAS \cite{atlas}. We mention at once that the table is \cite[Table 9]{guralnick2012malle}, except that for $S=Co_1$ we chose $23A$ instead of $13A$.

	\begin{table}
		\centering
		\caption{In the table, $n_i:=|N_S(\gen{x_i})/C_S(x_i)|$. See \Cref{subsec:notation} for remarks and notation.}
		\label{table:new_unique}
		\begin{tabular}{llllll}
			\hline\noalign{\smallskip}
			$S$ & $|x_1|$ & $n_1$ & $|x_2|$ & $n_2$   &  Conditions \\
			\noalign{\smallskip}\hline\noalign{\smallskip}
			$\PSL_m(r)$ & $\Phi^*_m(r)$  & $m$ &   $\Phi^*_{m-1}(r)$ & $m-1$ & $m\ge 4$  \\
			$\PSU_m(r)$ & $\Phi^*_{m-1}(r^2)$ & $m-1$  &  $\Phi^*_{m/2}(r^2)$ & $m$ & $m\equiv 0\pmod 4$ \\
			$\PSU_m(r)$ & $\Phi^*_{m-1}(r^2)$ & $m-1$  &  $\Phi^*_{m/2}(r^2)$ & $m/2$ & $m\equiv 2\pmod 4$  \\
			$\PSU_m(r)$ & $\Phi^*_m(r^2)$ & $m$ & $\Phi^*_{m-2}(r^2)$ & $m-2$ & $m\equiv 3\pmod 4, \, m>3$  \\
			$\PSU_m(r)$ & $\Phi^*_m(r^2)$ & $m$ & $\Phi^*_{(m-1)/2}(r^2)$ & $m-1$ & $m\equiv 1\pmod 4$  \\
			$\PSp_{2m}(r)$  & $\Phi^*_{2m}(r)$ & $2m$ &  $\Phi^*_{2(m-1)}(r)$ & $2(m-1)$ & $(m,r)\neq (3,2)$ \\
			$\PSp_6(2)$  & $7$ & $6$ & & &\\  
			$\POm^+_{2m}(r)$ & $\Phi^*_{2(m-1)}(r)$ &  $2(m-1)$   & $\Phi^*_m(r)$ & $m$   &  \\
			$\POm^-_{2m}(r)$ & $\Phi^*_{2m}(r)$ & $m$ & $\Phi^*_{2(m-1)}(r)$ & $2(m-1)$ &   \\
			$\POm_{2m+1}(r)$ & $\Phi^*_{2m}(r)$ & $2m$ &  $\Phi^*_m(r)$ & $2m$ & $m$ odd  \\
			$\POm_{2m+1}(r)$ & $\Phi^*_{2m}(r)$ & $2m$ &  $\Phi^*_m(r)$ & $m$ & $m$ even \\
			$^2B_2(r)$ & $\Phi'_8$ & $4$  &&& \\
			$^2G_2(r)$  & $\Phi'_{12}$ & $6$  &&& \\ 
			$G_2(r)$ & $r^2+\varepsilon r +1$ & $6$  &&& $3\mid (r+\varepsilon)$ \\
			$G_2(r)$ & $r^2+r+1$ & $6$  &&& $3\mid r$  \\
			$^3D_4(r)$  & $r^4-r^2+1$ & $4$  &&& \\
			$F_4(r)$ & $r^4-r^2+1$ & $12$  &&&  \\
			$^2F_4(r)$ & $\Phi'_{24}$ & $12$  &&& \\	
			$E_6(r)$ & $\Phi^*_9(r)$ & $9$  &  $\Phi^*_8(r)$ &  $8$ &   \\
			$^2E_6(r)$ & $\Phi^*_{18}(r)$  & $9$  &   $\Phi^*_{10}(r)$ &  $10$ &   \\
			$E_7(r)$ &  $\Phi^*_{14}(r)$ & 14 &   &  &  \\
			$E_8(r)$ & $\Phi_{30}(r)$ & $30$  &&& \\
			\noalign{\smallskip}\hline\noalign{\smallskip}
		\end{tabular}
	\end{table}
	
	\begin{table}
		\centering
		\caption{In the table, $n_1:=|N_S(\gen{x_1})/C_S(x_1)|$.}
		\label{table:sporadics}
		\begin{tabular}{lll}
			\hline\noalign{\smallskip}
			$S$ & $x_1$ & $n_1$  \\
			\noalign{\smallskip}\hline\noalign{\smallskip}
			$M_{11}$ & $11A$ & $5$ \\
			$M_{12}$ & $11A$ & $5$  \\
			$J_1$ & $19A$ & $6$ \\
			$M_{22}$ & $11A$ & $5$ \\
			$J_2$ & $7A$ & $6$ \\
			$M_{23}$ & $23A$ & $11$ \\
			$^2F_4(2)'$ & $13A$ & $6$ \\
			$HS$ & $11A$ & $5$ \\
			$J_3$ & $19A$ & $9$ \\
			$M_{24}$ & $23A$ & $11$ \\
			$McL$ & $11A$ & $5$ \\
			$He$ & $17A$ & $8$ \\
			$Ru$ & $29A$ & $14$ \\
			$Suz$ & $13A$ & $6$ \\ 
			$ON$ & $31A$ & $15$ \\ 
			$Co_3$ & $23A$ & $11$ \\ 
			$Co_2$ & $23A$ & $11$ \\ 
			$Fi22$ & $13A$ & $6$ \\ 
			$Fi22$ & $13A$ & $6$ \\ 
			$HN$ & $19A$ & $9$  \\
			$Ly$ & $67A$ & $22$  \\
			$Th$ & $19A$ & $18$  \\
			$Fi_{23}$ & $17A$ & $16$  \\
			$Co_1$ & $23A$ & $11$  \\
			$J_4$ & $43A$ & $14$  \\
			$Fi'_{24}$ & $29A$ & $14$  \\
			$B$ & $47A$ & $23$  \\
			$M$ & $71A$ & $35$  \\
			\noalign{\smallskip}\hline\noalign{\smallskip}
		\end{tabular}
	\end{table}

Let us then focus on \Cref{table:new_unique}.	We denote by $\Phi_m(X)$ the $m$-th cyclotomic polynomial, and by $\Phi^*_m(r)$ the product of all primitive prime divisors (ppds) of $r^m-1$.  Recall that by Zsigmondy's theorem, $\Phi^*_m(r)>1$ unless $m=2$ and $r+1$ is a $2$-power, or $(m,r)=(6,2)$. (The notation $\Phi^*_m(r)$, which is borrowed from \cite{guralnick2012malle}, is not accidental, in the sense that $\Phi^*_m(r)$ always divides the $m$-th cyclotomic polynomials $\Phi_m$ evaluated at $r$.)
	We also denote $\Phi'_8 = r+\sqrt{2r}+1$, $\Phi'_{12} = r+\sqrt{3r}+1$, $\Phi'_{24} = r^2 + \sqrt{2r^3}+ r+\sqrt{2r}+1$.
	
	\subsection*{Definition of $x_i$}
	In \Cref{table:new_unique}, we define elements $x_1$ and $x_2$ via their orders, and for classical groups in each case is straightforward to deduce their action on the natural module. For example, $x_1\in \POm^+_{2m}(r)$ acts irreducibly on a nondegenerate $2m-2$-space, and trivially on the perpendicular complement. There are cases where we define only an element $x_1$, and the corresponding entry for $x_2$ is empty. 
	
	Note that in some cases the order of $x_i$ is $1$. For $x_1$, this happens when $S$ is $\PSL_6(2)$ or  $\POm^+_8(2)$. For $x_2$, this happens when $S$ is $\PSL_7(2)$, $\PSp_4(r)$ (and $r+1$ is a $2$-power), $\PSp_8(2)$, $\POm^+_{12}(2)$, or $\POm^-_8(2)$. In particular, in all these cases we regard $x_i$ as being not defined.

	As in \Cref{table:new_unique}, we set $n_i:=|N_S(\gen{x_i})/C_S(x_i)|$. In the following lemma, we verify this value (item (vi)), at the same time pointing out that the value does not change when passing to a quasisimple cover (item (ii)). Letting $X$ be a simple algebraic group over $\overline{\F_r}$ with a Frobenius endomorphism $F$ such that $S=[X^F, X^F]$, we note
	 that if $x_i$ is regular then $n_i=|N_{X^F}(T)/T^F|$, where $T$ is the maximal torus containing $x_i$ in $X$ (item (iv)); and that $N_S(\gen{x_i})/C_S(x_i)$ acts semiregularly on $\gen{x_i}\sm\{1\}$ (item (v)). We also include the case of sporadic groups (\Cref{table:sporadics}), where all claims can be verified by inspection of the ATLAS \cite{atlas}.

	\begin{lemma}
		\label{l:coprime_schur} 
		Assume $S\neq \PSL_2(r), \PSL^\pm_3(r)$. Let $L$ be a quasisimple cover of $S$ and let $x_i$ be an element appearing in \Cref{table:new_unique} or \Cref{table:sporadics}.
		\begin{itemize}
			\item[(i)]  $|x_i|$ is coprime to the order of the Schur multiplier of $S$, and so there exists a unique lift $y_i\in L$ of $x_i$ such that $|y_i|=|x_i|$. (In the items below, $y_i$ denotes such a lift.)
			\item[(ii)] We have $C_S(x_i) =C_L(y_i)/Z(L)$ and $N_S(\gen{x_i}) = N_L(\gen{y_i})/Z(L)$.
			
			\item[(iii)] Assume that $S$ is of Lie type. Then $x_i$ is non-regular in $S$ if and only if $i=2$ and one of the following holds: $S=\PSU_m(r)$ with $m\equiv 2,3\pmod 4$; $S=\PSp_{2m}(r)$; $S= \POm^+_{2m}(r)$ or $\POm_{2m+1}(r)$ with $m$ even.
			\item[(iv)] Assume that $S$ is sporadic or $S$ is of Lie type and $x_i$ is regular. Then $C_L(y_i)$ is abelian. Moreover, if $X$ is a simple  algebraic group with a Frobenius endomorphism $F$ such that $L=[X^F, X^F]$, and if  $T$ is the unique maximal torus of $X$ containing $y_i$, we have $C_L(y_i) = T\cap L$ and $N_L(\gen{y_i}) = N_L(T\cap L)$. 
			\item[(v)] For every nontrivial power $y$ of $y_i$, we have $C_L(y)=C_L(y_i)$ and $N_L(\gen y)=N_L(\gen{y_i})$, and so $N_L(\gen{y_i})/C_L(y_i)$ acts semiregularly on $\gen{y_i} \sm\{1\}$.
			\item[(vi)] $|N_L(\gen{y_i})/C_L(y_i)|$ is equal to $n_i$ as listed in \Cref{table:new_unique} or \Cref{table:sporadics}.
			\item[(vii)] Assume that $S$ is sporadic or $S$ is of Lie type and $x_i$ is regular. If $\ell$ is a prime divisor of $|x_i|$, then $\ell\ge 5$ and $\gen{x_i}$ contains a Sylow $\ell$-subgroup of $S$. 
		\end{itemize}
	\end{lemma}
	
	\begin{proof}
		(i) The last statement (``and so...'') follows from the first. The first statement follows from inspection of the Schur multiplier $M$ of $S$ and, in groups of Lie type, from the fact that each prime divisor of $\Phi^*_c(r)$ is congruent to $1$ mod $c$. Let us handle for example $S=\PSL_m(r)$. Each prime divisor of $M$ divides $m$, 
		and each prime divisor of $|x_1|$ is at least $m+1$, so $(|x_1|,M)=1$. As for $x_2$, the only possible exception occurs when $m$ is a prime divisor of $\Phi^*_{m-1}(r)$. But then by definition of ppd, we have that $m$ does not divide $r-1$ and so does not divide $M$.
		
		(ii) follows immediately from (i).
		
		(iii) Assume first $S$ is classical, let $\tilde S$ be the universal covering group of $S$ and let $\tilde x_i$ be a lift of $x_i$ as in (i). We readily see that  $\tilde x_i$ has distinct eigenvalues on the natural module (in which case $x_i$ is regular) except possibly in the following cases: $i=2$ and $S=\PSU_m(r),\PSp_{2m}(r),\POm^\pm_{2m}(r),\POm_{2m+1}(r)$; or $i=1$ and $S=\POm^+_{2m}(r)$. In the case $i=2$ and $S=\POm^-_{2m}(r)$, or $i=1$ and $S=\POm^+_{2m}(r)$, we have that $\tilde x_i$ centralizes a $2$-space and acts irreducibly on a complement, so $x_i$ is regular. In the case $i=2$ and $S=\PSp_{2m}(r)$, or $S=\PSU_m(r)$ with $m\equiv 3\pmod 4$, $\tilde x_i$ centralizes a $2$-space and so $x_i$ is not regular. Let us go through the remaining cases. If $S=\PSU_m(r)$ with $m$ even or $m\equiv 1 \pmod 4$, a straightforward calculation shows that $\tilde x_2$ has distinct eigenvalues  if and only if $m\equiv 0,1\pmod 4$ (that is, if and only if a totally singular irreducible $\F_r\gen{\tilde x_2}$-submodule has even dimension).  On the other hand, if $S=\POm^+_{2m}(r),\POm_{2m+1}(r)$ then $\tilde x_2$ has distinct eigenvalues  if and only if $m$ is odd. This concludes the proof if $S$ is classical.
		
		Assume now $S$ is exceptional. Since $C_S(x_i)$ is a subgroup of $S$ of maximal rank, it is enough to prove that $C_M(x_i)$ is an $r'$-group for every maximal subgroup of maximal rank $M$ containing $x_i$. These are listed in \cite{LSS}, and the check is straightforward. For example, assume $S=E_6(r)$.  For $x_1$, by order considerations the only option is $M=\PSL_3(r^3).3$; for $x_2$, the only option is a parabolic with Levi subgroup $D_5(r)$. In both cases $C_M(x_i)$ is an $r'$-group and we are done. Assume now $S=\,^2E_6(r)$. For $x_1$, it must be $M=\PSU_3(r^3).3$; for $x_2$, the only option for $M$ is a reductive subgroup of type $^2D_5(r)$, and we are done. Assume now $S=E_7(r)$. The options for $M$ are reductive subgroups of type $^2A_7(r)$ and $A_1(r^7)$. We see that in both cases $C_M(x_1)$ is an $r'$-group.

		(iv) If $S$ is sporadic then by inspection of \cite{atlas} we have that $\langle x_i \rangle$ is self-centralizing, from which the claim follows. Assume now $S$ is of Lie type and $x_i$ is regular. 
		There exists a simple algebraic group $X$ and a prime power $s$ with $(s,r)\neq 1$ such that $L$ is an $s$-cover of $[X^F,X^F]$ (see \cite[Table 5.1.D]{kleidman1990liebeck}). In particular, it is sufficient to prove the assertion in the case where $L=[X^F, X^F]$. Let us begin with centralizers. A theorem of Steinberg (see \cite[Theorem 14.16]{malle_testerman_2011linear}) asserts that $C_L(y_i) = T$ if $X$ is simply connected, and the general case follows from this and (ii). For what concerns normalizers, certainly $N_L(\gen{y_i}) \le  N_L(C_L(y_i)) = N_L(T\cap L)$. For the other inclusion, we may assume $S=L$. Except for $x_2$ in $E^\pm_6(r)$, we see that $T\cap L$ is cyclic and so the inclusion holds. For $x_2$ in $E^\pm_6(r)$,  we readily see that $T\cap L$ has a unique subgroup of order $|x_i|$, and the inclusion holds also in this case.

		(v) If $S$ is sporadic, this is immediate since $|x_i|$ is prime. If $S$ is of Lie type and $x_i$ is regular, the statement follows from the same argument of (iii),(iv) applied to $y$. Assume then $x_i$ is not regular; the cases are listed in (iii). In all cases, $x_i$ and every nontrivial power of it fix the same subspaces of the natural module, from which we see that
		$C_L(y_i) = C_L(y)$ and $N_L(\gen{y_i}) = N_L(\gen y)$, as desired.

		(vi) If $S$ is sporadic the information is contained in \cite[Table 9]{guralnick2012malle}, except for $S=Co_1$, in which case we consult \cite{atlas}. Assume then $S$ is of Lie type. See \Cref{rem:weyl_group}, below, for the case where $x_i$ is regular. In the other cases we can check the value of $n_i$ directly, as follows. Consider $x_2\in \PSU_m(r)$ with $m\equiv 2 \pmod 4$; we  work in $S=\SU_m(r)$ for convenience. Then $N_S(\gen{x_2})$ is an extension field subgroup of type $\GU_2(r^{m/2}).(m/2)$, and $N_S(\gen{x_2})/C_S(x_2)\cong C_{m/2}$. Consider now $x_2\in \POm^+_{2m}(r)$ with $m$ even; we work in $S=\Omega^+_{2m}(r)$. Put $A=\SO^+_m(r)$ or $A=S$ according to whether $r$ is odd or even, and note that if $r$ is odd then $C_S(x_2) <C_A(x_2)$; in particular $N_S(\gen{x_2}) <N_A(\gen{x_2})$ and so we may compute $N_A(\gen{x_2})/C_A(x_2)$. Now we have that  $N_A(\gen{x_2})$ is an extension field subgroup of type $\GU_2(r^{m/2}).m$ and $N_A(\gen{x_2})/C_A(x_2)\cong C_m$. The case  $x_2\in \POm_{2m+1}(r)$ with $m$ even is entirely analogous.
		
		(vii) If $S$ is sporadic we use \cite{atlas}, and if $S$ is of Lie type this is a straightforward check. We only note that the parameter $\varepsilon$ for $S=G_2(r)$ ensures that $|x_1|$ is not divisible by $3$.
	\end{proof}
	
	\begin{remark}
		\label{rem:weyl_group}
		Let us recall that if $x_i$ is regular, then the value $N_L(\gen{y_i})/C_L(y_i)$ can be read off from the Weyl group, as follows. 
		Assume for convenience that $L=X^F$ is of simply connected type, so by \Cref{l:coprime_schur}(iv) we have $N_L(\gen{y_i})/ C_L(y_i) = N_{X^F}(T^F)/T^F$. Write $T=R_w$ where $R$ is a fixed $F$-stable  maximal torus and $w\in W:=N_X(R)/R$ (see \cite[Section 25.1]{malle_testerman_2011linear} for the notation used).  Since in our case $N_{X^F}(T^F)= N_{X^F}(T)$, we have that $|N_{X^F}(T^F)/T^F|=|C_W(F w)|$, where we view $F w$ as an element of the coset $F W$ of $W\rtimes \gen{F}$; see for example \cite[Proposition 25.3]{malle_testerman_2011linear}.
	\end{remark}

	We conclude this subsection with another lemma. Recall that $p$ is a prime and $K = \overline{\F_p}$.
	
	\begin{lemma}
		\label{l:landazuri_seitz}
		Let $S$ be sporadic or in $\mathrm{Lie}(p')$, let $x=x_i$ be an element in \Cref{table:new_unique} or \Cref{table:sporadics}, and assume that $m\ge 4$ if $S=\PSL^\pm_m(r)$.
		If there exists an irreducible projective $KS$-representation of dimension $d<10n_i$, then $S$ is either a sporadic group different from $ON, He,Th,Fi_{23}, Fi'_{24}, B, M$, or $S$ is one of the following:
		\begin{align*}
			&\PSL_m(r),\;(m,r)=(4,2),(4,3),(5,2) \\
			&\PSU_m(r),\;(m,r)=(4,2),(4,3),(5,2),(6,2),(7,2)\\
			&\PSp_{2m}(r),\;(m,r)=(2,3),(2,4),(2,5),(2,7),(3,2),(3,3),(4,2),(4,3) \\
			&\POm^{\pm}_{8}(2), \POm_7(3), \, ^2\!B_2(8), G_2(3), \,^3\!D_4(2), F_4(2). \\
		\end{align*}
	\end{lemma}
	
	\begin{proof}
		For sporadic groups we consult \cite{hiss2001malle}. For groups of Lie type, we consult \cite{landazuri1974minimal} and we are reduced to the groups in the statement, with the addition of $\PSL_6(2), \PSL_7(2)$, which can be excluded with \cite{hiss2001malle}. 
	\end{proof}

	\subsection{Invariable generation of groups of Lie type} We need a result on invariable generation of simple groups of Lie type. Recall that elements $x_1, \ldots, x_t$ \textit{invariably generate} a group $G$ if $\gen{x_1^{g_1}, \ldots, x_t^{g_t}}=G$ for every $g_1, \ldots, g_t\in G$. We write $\gen{x_1,\ldots, x_t}_I=G$ in this case. We consider the set of the following finite simple groups:
	
	\begin{align}
		\label{eq:condition_invariable_generation}
		&E_6(r), \,^2E_6(r); \\
		&\text{$\PSL_m(r)$, $m\ge 5$ and $(m,r)\neq (5,2),(6,2),(7,2),(11,2),(13,2),(19,2),(5,3),(7,3), (7,5)$;} \nonumber \\
		&\text{$\PSU_m(r)$, $m\ge 5$ and $(m,r)\neq (5,2),(5,3),(13,2)$;}\nonumber\\
		&\text{$\PSp_{2m}(r)$, $m\ge 3$, $r$ odd;}\nonumber\\
		&\text{$\POm^+_{2m}(r)$, $m\ge 5$ odd;}\nonumber\\
		&\text{$\POm^-_{2m}(r)$, $m\ge 4$, $r$ odd;}\nonumber\\
		&\text{$\POm_{2m+1}(r)$, $m\ge 4$.}\nonumber
	\end{align}
	
	\begin{lemma}
		\label{l:invariable generation}
		Let $S$ be as in \eqref{eq:condition_invariable_generation} and let $x_1, x_2$ be as in \Cref{table:new_unique}. Then $\gen{x_1, x_2}_I=S$.
	\end{lemma}
	
	\begin{proof}
		Our main tool is \cite[Corollary 3.4]{guralnick2012beauville}. This result (mostly relying on \cite{guralnick1999praeger}) classifies the irreducible subgroups of $\GL_m(r)$, $m\ge 5$, containing elements of order $\Phi^*_{e_i}(r)$, $i=1,2$, with $e_1>e_2>m/2$. Sometimes this result will not apply, because we have at our disposal only one such element. Whenever this is the case, we will use \cite[Theorem 2.2]{guralnick2012malle}, which classifies the irreducible subgroups of $\GL_m(r)$ containing an element of order $\Phi_e^*(r)$ with $\Phi_e^*(r)>2e+1$. By  \cite[Lemma 2.1]{guralnick2012malle}, the inequality $\Phi_e^*(r)>2e+1$ is satisfied in all but a handful of cases, which either do not appear in \eqref{eq:condition_invariable_generation} or will be dealt with separately. (We point out that item (2) in \cite[Theorem 3.3 and Corollary 3.4]{guralnick2012beauville} should be amended, to include the case where $m$ is prime and $\Phi^*_{m-1}(r)=m$, in which case $\GL_1(r^m).m$ contains elements of orders both $\Phi^*_m(r)$ and $\Phi^*_{m-1}(r)$. By \cite[Lemma 2.1]{guralnick2012malle}, the only such cases are the ones listed in the second line  of \Cref{eq:condition_invariable_generation}, other than $(6,2),(7,2)$. This is why we excluded these cases for $S=\PSL_m(r)$.)

		In all cases below, we denote by $H$ the subgroup of $S$ generated by any conjugates of $x_1$, $x_2$ and $x_3$, so our aim is to show that $H=S$. Assume first $S=\PSL_m(r)$, so by assumption $m\ge 5$. Note $H$ is irreducible and $H$ does not preserve an extension field subgroup (see the last sentence in the previous paragraph).
		Moreover, if $r=\ell^a$ where $\ell$ is prime then $|x_1|$ is divisible by $\Phi^*_{am}(\ell) >1$, which rules out subfield subgroups. A classical subgroup does not have order divisible by both $|x_1|$ and $|x_2|$, so it is ruled out.
		Since by assumption $(m,r)\neq (6,2),(7,2)$, by \cite[Corollary 3.4]{guralnick2012beauville} there are no other possibilities (note that cases (4), (5), and (6) in that corollary are ruled out at once), so $H=S$, as desired.
		
		Assume now $S=\PSU_m(r)$, so $m\ge 5$.
		Assume first $m$ is odd. Then $H$ is irreducible, and we claim that $H$ does not preserve an extension field subgroup. If $m\equiv 3\pmod 4$ this holds since $m$ is odd and so $m$ and $m-2$ are coprime. If $m\equiv 1\pmod 4$, the only option would be a subgroup $\GU_1(r^m).m$. But since we excluded the cases $(m,r)=(5,2), (5,3), (13,2)$, we deduce from \cite[Lemma 2.1]{guralnick2012malle} that $\Phi^*_{m-1}(r)>m$ and so this possibility is excluded (note that $\Phi^*_{m-1}(r)$ divides $|x_2|$).
		Moreover, as above, a subfield subgroup does not have order divisible by $|x_1|$, so it is ruled out. Again by \cite[Lemma 2.1]{guralnick2012malle}, we have that $\Phi^*_m(r^2)>2m+1$ (since $m$ is odd), so we deduce by  \cite[Theorem 2.2]{guralnick2012malle} that there are no remaining possibilities and so $H=S$. Assume then $m$ is even. Again $H$ is irreducible. If $\Phi^*_{m-1}(r^2) >2(m-1)+1$, then by \cite[Theorem 2.2]{guralnick2012malle} we deduce $H=S$.
		And if $\Phi^*_{m-1}(r^2) \le 2(m-1)+1$, then by \cite[Lemma 2.1]{guralnick2012malle},  the only possibility is that $r^2=4$ and $m-1 =3$ or $m-1=6$, none of which is under consideration in 
		\eqref{eq:condition_invariable_generation}, so again $H=S$.
		
		Assume next $S=\PSp_{2m}(r)$, so by assumption $m\ge 3$ and $r$ is odd. Since $(m,r)\neq (3,2), (4,2)$, \cite[Corollary 3.4]{guralnick2012beauville} applies. Classical subgroups do not occur ($\Or^\pm_{2m}(r)$ is ruled out since $r$ is odd). An extension field subgroup would necessarily preserve a structure over $\F_{r^2}$. A subgroup of type $\Sp_m(r^2)$ is excluded by looking at $x_2$; a subgroup of type $\GU_m(r)$ is excluded because one of $m$ and $m-1$ is even. Subfield subgroups are excluded by looking at $|x_1|$, as above. By \cite[Corollary 3.4]{guralnick2012beauville} if $(m,r)\neq (3,3)$ there are no remaining possibilities and so $H=S$. If $(m,r)=(3,3)$, the conclusion follows from \cite[Corollary 3.4]{guralnick2012beauville}, since $A_7$ (in the fully deleted permutation module) embeds in $\Omega^\pm_6(3)$ rather than $\Sp_6(3)$.

		Assume now $S=\POm^-_{2m}(r)$ with $m\ge 4$ and $r$ odd. Since $(m,r)\neq (4,2)$,  \cite[Corollary 3.4]{guralnick2012beauville} applies and we argue as in the symplectic case.

		Assume now $S=\POm_{2m+1}(r)$, so $m\ge 4$ and $r$ is odd.
		Note that $H$ is irreducible. (Indeed, both $x_1$ and $x_2$ fix a nondegenerate $1$-space, but one has square discriminant and the other has non-square discriminant.) Here \cite[Corollary 3.4]{guralnick2012beauville} does not apply and we use \cite[Lemma 2.1 and Theorem 2.2]{guralnick2012malle} for the element $x_1$. 
		By \cite[Lemma 2.1]{guralnick2012malle} we have $\Phi^*_{2m}(r) > 4m+1$; and since $m\ge 4$, by \cite[Theorem 2.2]{guralnick2012malle} we deduce that $H=S$.

		Assume now $S=\POm_{2m}^+(r)$, so $m\ge 5$ is odd. Note that $H$ is irreducible. We use \cite[Lemma 2.1 and Theorem 2.2]{guralnick2012malle}. Since $m$ is odd, $H$ does not preserve an extension field subgroup. (The crux is that $x_1$ cannot belong to a subgroup $\GU_m(r)$, which happens instead for $m$ even.)  
		Assume first $\Phi^*_{2(m-1)}(r) > 4(m-1) +1$. Then by \cite[Theorem 2.2]{guralnick2012malle} we see that $H=S$. Assume finally $\Phi^*_{2(m-1)}(r) \le 4(m-1) +1$. Since $2(m-1) \ge 8$ and $m$ is odd, by \cite[Lemma 2.1]{guralnick2012malle} one of the following holds: $(m,r) = (7,2)$ (and $\Phi^*_{2(m-1)}(r) = 2(m-1) +1$) or $(m,r)=(5,2),(11,2)$ (and $\Phi^*_{2(m-1)}(r) = 4(m-1) +1$). 
		If $(m,r)=(7,2)$, we have $|x_1|=13$ and $|x_2| = \Phi^*_7(2) = 127$. Now \cite{guralnick1999praeger} lists the possibilities for the overgroups of $x_1$. With their notation, we have already excluded Examples 2.1, 2.2, 2.4. Example 2.5 does not arise as we have $r=2$, and the groups in Example 2.3 do not contain an element of order $127$. Finally, Examples 2.6--2.9 consist of an explicit list of  almost simple acting (projectively) absolutely irreducibly. All these can be ruled out, as either they do not contain an element of order $127$, or the representation does not have degree $14$. The cases  $(m,r)=(5,2), (11,2)$ can be handled in the same way, using \cite{guralnick1999praeger} (note that in the first case $|x_2|=\Phi^*_5(2) =31$, and in the second case $|x_2| = \Phi^*_{11}(2)=2047 = 23\cdot 89$).
		
		Assume now $S=E_6(r)$; we look at the list of maximal subgroups in \cite[Tables 2 and 9]{craven2023maximal}. We have $|x_1|=\Phi^*_{9}(r)\ge 37$ (by \cite[Lemma 2.1]{guralnick2012malle}), and we see that the overgroups of $x_1$ are of type $\PSL_3(r^3)$ and $^2E_6(r^{1/2})$.
		None of these can contain $x_2$ and we are done. 
		
		Assume finally $S=\,^2E_6(r)$; we look at  \cite[Tables 3 and 10]{craven2023maximal}. Note $|x_1| = \Phi^*_{18}(r)$ is equal to $19$ for $r=2$, and is at least $73$ for $r>2$ (by \cite[Lemma 2.1]{guralnick2012malle}). The overgroups of $x_1$ are a subgroup of type $\PSU_3(r^3)$, and $\PSL_2(19)$ (for $r=2$).
		None of these can contain $x_2$ and $H=S$.
	\end{proof}

	\subsection{Generation of groups of Lie type by conjugates}
	Next, we need a result on generation of simple groups of Lie type by two conjugate elements. 
	Consider the set of the following finite simple groups:
	\begin{align}
		\label{eq:more_groups}
		&\PSL^\pm_4(r)\\ 
		&\POm^+_{2m}(r) \text{ with $m$ even and $(m,r)\neq (4,2)$ }\nonumber \\
		&E_7(r) \nonumber
	\end{align}

	\begin{lemma}
		\label{l:generation_by_conjugates}
		Let $S$ be as in $\eqref{eq:more_groups}$. If $S\neq \PSL_4^-(r)$ then let $x=x_1$ be as in \Cref{table:new_unique}; if $S= \PSL_4^-(r)$ then let $x=x_2$ be as in \Cref{table:new_unique}. Then $S$ is generated by two conjugates of $x$.
	\end{lemma}

	\begin{proof}
		Let $\ca M_c=\ca M_c(x)$ be a set of representatives for the conjugacy classes of maximal subgroups containing $x$. For $M\in \ca M_c$, the number of conjugates of $M$ containing $x$ is $\fp(x,S/M)$. For each such conjugate, say $M^g$, the probability that a random conjugate of $x$ belongs to $M^g$ is $|x^G\cap M|/|x^G| = \fp(x,S/M)/|S:M|$. Letting  $P$ the probability that $x$ and a random conjugate do not generate, we deduce by a union bound that
		\begin{equation}
			\label{eq:generation_conjugates}  
			P\le \sum_{M\in \ca M_c}\frac{\fp(x,S/M)^2}{|S:M|}.
		\end{equation}
		In particular, it will be enough to show that the right-hand side of \eqref{eq:generation_conjugates} is less than $1$, and in order to achieve this we will mostly use \Cref{l:fpr_enough}. Let also $\ca M=\ca M(x)$ be the set of maximal subgroups of $S$ containing $x$. Note that if $|\ca M(x)|=1$, say $\ca M(x) = \{M\}$, then the right-hand of \eqref{eq:generation_conjugates} is at most $1/|S:M|<1$. When $S$ is classical, we will work with the cover of $S$ acting faithfully on the natural module, without changing notation.

		Let us start from $S=\SL_4(r)$ where $r=\ell^a$, so $|x|=\Phi^*_4(r)$ and each prime divisor of $|x|$ is $\equiv 1\pmod 4$. The case $r=2$ can be checked with GAP. Assume then $r\ge 3$. Note that $x$ is contained in a unique extension field subgroup $M_1$ of type $\GL_2(r^2)$, see for example \cite[Lemma 2.12]{breuer2008probabilistic}, so we have
		\[
		\frac{\fp(x,S/M_1)^2}{|S:M_1|}=\frac{1}{|S:M_1|} = \frac{2}{r^4(r^3-1)(r-1)} < r^{-6}.
		\]
		Moreover, $x$ is contained in a subgroup $M_2=\Sp_4(r).(2,r-1)$, and in a subgroup $M_3=\SO^-_4(r).(4,r-1)$, the latter being maximal only if $r$ is odd. In both cases, we have $x^S\cap M_i = x^{M_i}$, and we see from \Cref{l:fixed_points_precise_normal_subgroup} that
		\begin{align*}
			\frac{\fp(x,S/M_i)^2}{|S:M_i|} \le r^{-2}.
		\end{align*}
		For example, for $i=2$, setting $M_0=\Sp_4(r)$ and $d=(2,r-1)$ we have
		\begin{align}
			\label{eq:fpr_symplectic_subgroup}
			\frac{\fp(x,S/M_2)^2}{|S:M_2|} &= \frac{|C_S(x)|^2}{d^2|C_{M_0}(x)|^2}\cdot \frac{|M_2|}{|S|}\\ &=\frac{(r^4-1)^2}{(r-1)^2(r^2+1)^2d^2}\cdot \frac{d}{r^2(r^3-1)} \le r^{-2}.\nonumber
		\end{align}
		In particular, we have
		\begin{equation}
			\label{eq:fpr_SL_4}
			\sum_{i=1}^3 \frac{\fp(x,S/M_i)^2}{|S:M_i|} \le dr^{-2}+r^{-6} <1.
		\end{equation}
		Next, recalling that $r=\ell^a$, we have that $\Phi^*_{4a}(\ell)$ divides $\Phi^*_{4}(r)$. Since $\Phi^*_{4a}(\ell)>1$, subfield subgroups are ruled out. (We note at once that in this proof, subfield subgroups will always be ruled out with this argument.) Consulting the list of maximal subgroups of $S$ in \cite[p. 381]{bray2013holt_colva}, we see that the only possible other overgroups of $x$ are symplectic type subgroups and covers of $A_7$ and  $\PSU_4(2)$, occuring only (possibly) when $|x|=5$. By \cite[Lemma 2.1]{guralnick2012malle}, we have $|x|=5$ only if $r=2,3$; and in these cases, only $A_7$ arises among the overgroups above, for $r=2$. We are assuming $r\ge 3$ and so we conclude from \eqref{eq:generation_conjugates} and \Cref{eq:fpr_SL_4}.

		Assume now $S=\SU_4(r)$ and $|x|=\Phi^*_2(r^2)$. We check the cases $r=2,3$ with GAP so assume $r\ge 4$. Then $|x| \ge 13$ by \cite[Lemma 2.1]{guralnick2012malle}. We deduce from \cite[p. 382]{bray2013holt_colva} that the only overgroups of $x$ are the stabilizer $M_1$ of a totally singular $2$-space, the stabilizer $M_2$ of a decomposition into totally singular $2$-spaces, and subfield subgroups  $\Sp_4(r).d$ and $\SO^-_4(r).e$, where $d:=(2,r-1)$ and $e:=(4,r+1)/2$ (there are $d$ and $e$ classes, respectively); the orthogonal subgroup being maximal only for $r$ odd. We have $\fp(x,S/M_i)=1$ for $i=1,2$. Setting $M_0=\Sp_4(r)$, we can calculate $\fp(g,S/M_3)^2/|S:M_3|$ precisely, similarly to \eqref{eq:fpr_symplectic_subgroup}; we see that this quantity is at most $1/r^3$, and the same holds for $M_4$. The same holds also for $M_3^a$ and $M_4^a$ with $a\in \Aut(S)$, therefore
		\begin{align*}
			P&\le \sum_{i=1}^2 \frac{\fp(x,S/M_i)^2}{|S:M_i|} + d\cdot \frac{\fp(x,S/M_3)^2}{|S:M_3|} + e\cdot \frac{\fp(x,S/M_3)^2}{|S:M_3|}\\
			&\le  |S:M_1|^{-1} + |S:M_2|^{-1} + 2r^{-3} <1.\\
		\end{align*}

		Assume now $S=\Om^+_{2m}(r)$ with $m\ge 4$ even and $(m,r)\neq (4,2)$, so $|x|=\Phi^*_{2(m-1)}(r)$. Then $x$ stabilizes a unique nondegenerate $2$-space and $r+1$ nonsingular $1$-spaces. Moreover, $x$ lies in a unique extension field subgroup of type $\GU_m(r)$ if $m$ is even, and of type $\Or_m(r)$ if $m$ is odd.
		
		Let us first address the case $m=4$; see \cite[pp. 402--403]{bray2013holt_colva} for the list of maximal subgroups of $G$, taken from \cite{kleidman1987maximal_omega8}. Letting $A=\gen{\text{Inndiag}(\overline S),\tau}$ where $\overline S=\POm^+_8(r)$ and $\tau$ is a triality automorphism, we note  that for $r$ odd (resp. $r$ even) there are six (resp. three) $\overline S$-classes of subgroups $\Omega_7(r)$ (resp. $\Sp_6(r)$) conjugate under $A$, four of which (resp. two of which) act (projectively) irreducibly via the spin module. (Note that $\tau$ does not lift to an automorphism of $S$; this is why we considered $\overline S$ here.) Let $M_1\le S$ be the preimage of a representative for this $A$-class; we can take $M_1$ to be the stabilizer of some nonsingular $1$-space. Furthermore, there are three $\overline S$-classes of subgroups isomorphic to the stabilizer of a nondegenerate $2$-space of minus type, which are conjugate under $A$, and two of which are extension field subgroups of type $\GU_4(r)$. Let $M_2\le S$ be the preimage of a representative for this $A$-class. 
		
		Finally, for $2<r\equiv 2\pmod 3$, setting $d=(2,r-1)$  there are $d^2$ $S$-classes of subgroups $d\times\PSU_3(r).3$, acting irreducibly via the adjoint module. 
		Let $M_3$ be one such subgroup. Notice that $x^S\cap M_3$ is the union of at most six  $M_3$-classes. Indeed, if $y\in x^S\cap M_3$, then $y$ is irreducible as an element of $\mathrm{PGU}_3(r)$; let $\alpha$ be an eigenvalue of a lift of $y$ to $\GU_3(r)$.  Then $\beta:=\alpha^{r^2-1}$ is an eigenvalue of $x$, 
		and notice that $\alpha$ can be uniquely recovered from $\beta$.
		Since $x$ has six nontrivial eigenvalues, it follows that $x^S\cap M_3$ is indeed the union of at most six $M_3$-classes. Letting $x_1, \ldots, x_c$  ($c\le 6$) be representatives for such classes, we have $|C_S(x_i)|/|C_{M_3}(x)|\le r+1$
		and so by \Cref{l:fpr_enough} we deduce
		\[
		\frac{\fp(x,S/M_3)^2}{|S:M_3|}  \le \frac{36(r+1)^2}{|S:M_3|}.
		\]
		Next, letting $x\mapsto \overline x$  denote the map $S\to \overline S$, notice that $\overline x^A = \overline x^{\overline S}$. (Indeed, $\overline x$ belongs to a subgroup $G_2(r)$ and so is centralized by a triality automorphism.) In particular, for every $a\in A$ and every $i=1,2,3$, we have $\fpr(\overline x,\overline S/\overline{M_i}) = \fpr(\overline x,\overline S/\overline{M_i}^a)$. Consulting \cite{bray2013holt_colva}, we see that there are no other overgroups. For $r$ odd, by \eqref{eq:generation_conjugates} we then get 
		\[
		P\le  6\frac{(r+1)^2}{4|S:M_1|} + 3\frac{1}{|S:M_2|} + 144\frac{(r+1)^2}{|S:M_3|} <1.
		\]
		
		Assume then $m>4$. By \cite[Theorem 2.2]{guralnick2012malle}, if $\Phi^*_{2(m-1)}(r)> 4(m-1)+1$, then the only overgroups are the ones mentioned in the first paragraph handling $\POm^+_{2m}(r)$, and we conclude similarly to the case $m=4$.

		Assume finally $\Phi^*_{2(m-1)}(r)\le  4(m-1)+1$, so by \cite[Lemma 2.1]{guralnick2012malle}, recalling that we are assuming $m$ even, we see that $(m,r)=(6,2), (10,2)$, with $\Phi^*_{2(m-1)}(r)=2(m-1)+1$.
		For $m=6$ we can consult \cite{bray2013holt_colva} and there are no other overgroups. (Recall that the imprimitive subgroup of type $\Or_1(r)\wr S_{2m}$ is not maximal.) For $m=10$, we can consult
		\cite{guralnick1999praeger}; we find subgroups $J_1$ and $\PSL_2(19)$. (Note that $A_{22}$ embeds into $\Sp_{20}(2)$, but not into $\Omega^+_{20}(2)$, via the fully deleted permutation module; see for example \cite[p. 187]{kleidman1990liebeck}.)
		In these cases, from \Cref{l:fpr_enough} we have 
		\[
		\frac{\fp(x,S/M)^2}{|S:M|} \le \frac{|C_S(x)|^2}{|S:M|} < \frac{2^{24}}{|S:M|},
		\]
		from which the result follows.
		
		Assume finally $S=E_7(r)$. We have $|x_1| = \Phi^*_{14}(r) \ge 43$ by \cite[Lemma 2.1]{guralnick2012malle}. By \cite[Tables 1.1, 1.2 and 4.1]{craven2022E7}, the only maximal subgroups containing $x_1$ are $M_1=\PSL_2(r^7).7$ and a subgroup $M_2$ of type $\PSU_8(r)$. We have $|C_S(x_1)|\le r^7+1$ and so 
		\[
		P\le \frac{|C_S(x_1)|^2}{|S:M_1|} + \frac{|C_S(x_1)|^2}{|S:M_2|}< 1.
		\]
		The proof is concluded.
	\end{proof}
	
	\subsection{Generation of alternating groups} We also need a simple result on the generation of alternating groups.

	\begin{lemma}
		\label{l:generation_a_n}
		Let $\ell$ and $s$ be (not necessarily distinct) primes in $(m/2,m-3]$. Then $A_m$ is generated by one $\ell$-cycle and one $s$-cycle.
	\end{lemma}
	
	\begin{proof}
		Choose $x$ an $\ell$-cycle and $y$ an $s$-cycle such that $\gen{x,y}$ is transitive. Since $x$ cannot preserve nontrivial blocks, $\gen{x,y}$ is primitive. Finally $\gen{x,y}$ contains a cycle of prime length fixing at least three points and so $\gen{x,y}=A_m$ by Jordan's theorem. 
	\end{proof}
	
	\begin{lemma}
		\label{l:two_distinct_primes}
		Assume that $m\ge 16$. Then there are at least two distinct primes in $(m/2,m-3]$.
	\end{lemma}
	
	\begin{proof}
		If $m\ge 50$, then by \cite{nagura1952interval} there is a prime in $(m/2,\floor{3m/5}]$,
		and a prime in $(\floor{3m/5}, 18m/25]$. Since $18m/25 \le m-3$, the proof in this case is complete. If $16\le m\le 49$, then one checks the statement directly.
	\end{proof}
	
	\begin{lemma}
		\label{l:generation_a_n_second}
		For $m \ge 11$, there are two distinct primes $\ell$ and $s$ such that $\ell,s\ge 7$ and $A_m=\gen{x_1,y_1}=\gen{x_2,y_2}$ where $|x_1|=|y_1|=\ell$ and $|x_2|=|y_2|=s$.
	\end{lemma}
	
	\begin{proof}
		For $m\ge 16$ the statement follows from \Cref{l:generation_a_n,l:two_distinct_primes}. For $11\le m \le 15$ we can use GAP.
	\end{proof}

	\section{Classical groups: classes $\ca C_1, \ldots, \ca C_8$} 
	
	\label{sec:geometric_classes}
	
	Throughout this section, $p$ is a prime number, $q$ is a power of $p$, and $K=\overline{\F_p}$. Moreover, $G$ is a classical group with natural module $V=\F_{q^u}^n$, where $u=2$ if $G$ is unitary and $u=1$ otherwise. In view of various isomorphisms, we  assume $n\ge 3$ if $G=\PSU_n(q)$, $n\ge 4$ if $G=\PSp_n(q)$, $n\ge 7$ if $G=\POm^\varepsilon_n(q)$.  Recall also that $\Sp_4(2)'=\PSL_2(9)=A_6$, $\PSL_4(2)=A_8$, $\PSU_4(2)=\PSp_4(3)$, $\PSL_2(4)=\PSL_2(5)=A_5$ (see for example \cite[Proposition 2.9.1]{kleidman1990liebeck}).

	Aschbacher \cite{aschbacher1984maximal_subgroups} partitioned the members of $\ca A(G)$ into nine classes, which are generally denoted $\ca C_1, \ldots, \ca C_8, \ca S$ (see also \cite{kleidman1990liebeck,liebeck1998subgroup}). It will be convenient for us to introduce a further class $\ca N$, which we  define below. In this section, we handle classes $\ca C_1, \ldots, \ca C_8$ and  $\ca N$. We will handle class $\ca S$ in \Cref{sec:class_S}, and we postpone a brief description of it to that section.
	
	Classes $\ca C_1, \ldots, \ca C_8$ are sometimes called ``geometric'', in that the subgroups in these classes preserve natural structures on the natural module. We devote one subsection to each class, and at the beginning of each subsection, we give brief information on the class. We refer to \cite[Chapter 4]{kleidman1990liebeck} for detailed descriptions of the subgroups in each class. 
	
	Class $\ca N$ is non-empty only for $G=\POm^+_8(q)$, and $G=\Sp_4(q)$ with $q$ even. Put $\delta=3$ if $G=\POm^+_8(q)$, and $\delta=2$ if $G=\Sp_4(q)$. Then $\Aut(G)$ has a normal subgroup $\Gamma$ of index $\delta$, such that $\Aut(G)\sm \Gamma$ contains a graph automorphism of order $\delta$. We let $\ca N$ be the set of members of $\ca A(G)$ that do not extend to maximal subgroups of almost simple groups contained in $\Gamma$. (In particular, by the definition of $\ca A(G)$ each of them extends to a maximal subgroup of some almost simple group not contained in $\Gamma$.)

	For convenience, in the proofs we will replace $G$ by its quasisimple cover acting faithfully on $V$. We first handle some groups computationally (also for class $\ca S$).
	
	\begin{lemma}
		\label{l:small_groups}
		\Cref{t:main_simple} holds if $G$ is one of the following:
		\begin{align*}
			&\PSL_2(q) \emph{ with $q\le 9$} \\
			&\PSL_3(q) \emph{ with $q\le 8$} \\
			&\PSL_4(q) \emph{ with $q\le 3$} \\
			&\PSU_3(q) \emph{ with $q\le 8$} \\
			&\PSU_4(q) \emph{ with $q\le 3$}\\
			&\PSp_4(q) \emph{ with $q\le 4$} \\
			& \Sp_6(2)
		\end{align*}
	\end{lemma}
	
	\begin{proof}
		If $G\neq \Sp_6(2)$ we use GAP.	For each maximal subgroup $M$ of $G$, we perform a random search in $M$ and find $g\in M$ with $\fp(g,G/M)<|G:M|^{1/3}$. For the cases where $M\in \ca A(G)$ but $M$ is not maximal in $G$, as well as for the case $G=\Sp_6(2)$, we directly inspect \cite{atlas} and find $g\in M$ with $\fp(g,G/M)<|G:M|^{1/3}$.
	\end{proof}

	\subsection{Class $\ca C_1$}
	
	Subgroups $M$ in class $\ca C_1$ are stabilizers of certain subspaces $U$ of the natural module $V$. If $G$ preserves some nondegenerate form, then $U$ is nondegenerate, or totally singular, or a nonsingular $1$-space in orthogonal groups in even characteristic. What is more, if $U$ is nondegenerate, then $U$ is not similar to $U^\perp$.
	We also include the case where $G=\SL_n(q)$ and $M$ is the stabilizer of a flag or antiflag.

	\begin{lemma}
		\Cref{t:main_simple} holds if $M$ is in class $\ca C_1$.
	\end{lemma}
    
	\begin{proof}
		If $M$ is parabolic, let $g\in M$ be a regular unipotent element, so $\fp(g,G/M)=1$ by \Cref{l:parabolic}.
		Assume then $M$ is not parabolic. We go through all remaining cases.
		
		Assume $G=\SU_n(q)$, so $M$ is the stabilizer of a nondegenerate subspace $U_1$ of dimension $m_1<n/2$, and set $m_2=n-m$. Write $V=U_1\perp U_2$ and, according to this decomposition, let $g=g_1\perp g_2$ where $g_i$ is as follows. 
		If $m_i$ is odd, then $g_i$ acts irreducibly.
		If $m_i=2$ then $g_i$ is diagonal with distinct eigenvalues.
		If $m_i\neq 2$ is even, then $g_i$  acts irreducibly on each of a pair of complementary totally singular subspaces $A$ and $B$, and moreover $A\not\cong B$ as $\F_{q^2}\gen{g_i}$-modules. It is possible to choose $g_1$ and $g_2$ so that $\det(g)=1$, and such that $g_1$ and $g_2$ have no common eigenvalue, so $\fp(g,G/M)=1$. 
		
		Assume now $G=\Sp_n(q)$, so $M$ is the stabilizer of a nondegenerate subspace $U_1$ of dimension $m_1<n/2$. Write $V=U_1\perp U_2$, and accordingly take $g=g_1\perp g_2$ where $g_i$ acts irreducibly, so $\fp(g,G/M)=1$.
		
		Assume now $G=\Omega^\varepsilon_n(q)$. Assume first $M$ is the stabilizer of a nondegenerate subspace $U_1$ of dimension $m_1\le n/2$ and type $\delta_1$; set $m_2=n-m$ and let $\delta_2$ be the type of $U_2:=U_1^\perp$. We have that $U_1$ is not similar to $U_2$, and so $m_1<n/2$ unless $\varepsilon =-$ and $n\equiv 0\pmod 4$.
		According to the decomposition $V=U_1\perp U_2$, we choose $g=g_1\perp g_2$ where $g_i$ is as follows.

		If $m_i$ is odd, then $g_i$ acts irreducibly on a nondegenerate hyperplane of minus type and trivially on a complement. If $m_i$ is even and $\delta_i=-$, then $g_i$ acts irreducibly.
		If $(m_i,\delta_i,q)=(4,+,2)$ then $g_i$ acts irreducibly on a nondegenerate $2$-space and trivially on a complement. If $m_i$ is even and $\delta_i=+$ and $(m_i,q)\neq (4,2)$ then $g_i$ acts irreducibly on each of a pair of complementary totally singular subspaces $A$ and $B$, and moreover either $(m_i,q) = (2,3)$ and $g_i=-1$, or $(m_i,q)=(2,2)$ and $g_i=1$, or $A\not\cong B$ as $\F_q\gen{g_i}$-modules. It is possible to choose $g_1$ and $g_2$ so that $g\in \Omega_n^\varepsilon(q)$, and such that moreover $g_1$ and $g_2$ have no nontrivial irreducible submodule in common. In particular, we have $\fp(g,G/M)=1$ unless we are in one of the following cases:
		\begin{itemize}
			\item[(i)] $\delta_1=-$ and $m_2=m_1+1$, in which case $\fp(g,G/M)\le 2$. 
			\item[(ii)] $m_1$ and $m_2$ are odd, in which case $\fp(g,G/M)\le q+1$ (the number of $1$-subspaces of a $2$-space).
			\item[(iii)] $(m_1,\delta_1,q)=(2,+,3)$ and $m_2$ is odd, in which case $\fp(g,G/M)\le 3$ (one plus the number of nondegenerate $1$-subspaces of $U_1$).
		\end{itemize}
		Assume now $M=\Sp_{n-2}(q)$ is the stabilizer of a nonsingular $1$-space with $q$ even. We take $g=g_1\perp g_2$ where $g_1$ acts irreducibly on a nondegenerate $(n-2)$-space and $g_2$ is trivial on a nondegenerate $2$-space $U$. 
		Then $\fp(g,G/M)$ is the number of nonsingular $1$-spaces of $U$, so $\fp(g,G/M)\le q+1$, 
		and $(q+1)^3 < q^{n-1}(q^n-1) \le |G:M|$.

		The only remaining case is when $G=\SL_n(q)$ with $n\ge 3$ and $M$ is the stabilizer of an antiflag $V=U_1\oplus U_2$ with $\dim(U_1)\neq \dim(U_2)$. (Note that the stabilizer of a flag in $G$ is parabolic and so it has been considered in the first paragraph of the proof.)
		We choose $g=(g_1,g_2)$ where $g_i$ is irreducible on $U_i$, 
		so $\fp(g,G/M)=1$.
	\end{proof}
	
	\subsection{Class $\ca C_2$}
	
	Subgroups $M$ in class $\ca C_2$ are stabilizers of direct sum decompositions $V=V_1\oplus \cdots \oplus V_t$ of the natural module, where $\dim(V_i)=m$ and $t>1$. If $G$ preserves a nondegenerate form, then the $V_i$ are either nondegenerate and similar, or totally singular; in the latter case $t=2$. In all cases we have $n=mt$ and $M\le \GL(V_1)\wr S_t$.

	\begin{lemma}
		\label{l:determinant_imprimitive}
		Let $g\in \GL_m(K)\wr C_t$ preserve the decomposition $V=V_1\oplus \cdots \oplus V_t$, inducing a $t$-cycle on the spaces, and denote by $g'$ the restriction of $g^t$ to $V_1$. Then $\det(g)=(-1)^{m(t-1)}\det(g')$.
	\end{lemma}
	
	\begin{proof}
		We have that $g$ is conjugate in $\GL_m(K)\wr C_t$ to $(g', 1, \ldots, 1)\tau$, where $\tau$ is an element of order $t$ inducing a $t$-cycle on the spaces. The lemma follows from the fact that $\det(\tau)=(-1)^{m(t-1)}$.
	\end{proof}

	\begin{lemma}
		\Cref{t:main_simple} holds if $M$ is in class $\ca C_2$.
	\end{lemma}
	
	\begin{proof} 
		Assume $G=\SL_n(q)$ and $M$ of type $\GL_m(q)\wr S_t$. Let $g=(g', 1, \ldots, 1)\tau\in G$, where $g'\in \GL_m(q)$ is irreducible, $\tau$ is a $t$-cycle, and $\det(g')\det(\tau)=1$. (By a $t$-cycle, we mean an element of order $t$ inducing a $t$-cycle on the spaces.)
		Then $g$ is regular, and therefore $|C_G(g)|\le q^n-1<q^n$. 
		Now by \Cref{l:order_classical_groups}, $|G:M| \ge q^{n(n-m)-2}/t!$. 
		Bounding $t! \le q^{n\log_2(n)}$, $m\le n/2$, and $n\le q^{\log_2(n)}$, we see that if $n\ge 14$ then
		\[
		|C_G(g)|^3 <q^{3n} < q^{n(n-m)-2}/t! \le |G:M|,
		\]
		and the result follows from \Cref{l:fpr_enough}.
		Analyzing now the cases $n\le 13$, 
		we see that the inequality $q^{3n} < q^{n(n-m)-2}/t!$ holds unless $(n,m)=(6,3)$ or $2\le n\le 5$. Moreover, replacing $q^{n(n-m)-2}/t!$ with the precise value of $|G:M|$, we further exclude the case $n=5$.
		
		Now, by \cite{kleidman1990liebeck,bray2013holt_colva}, since $M\in \ca A$, if $m=1$ then $q\ge 5$; if $m=2$ then $q\ge 3$; if $n=2$ then $q\neq 5$. We also have $(n,q)\neq (3,2),(3,3)$ in view of \Cref{l:small_groups}.

		In the remaining cases with $m=n/2$ (namely, $(n,m)=(6,3),(4,2),(2,1)$), choose $1\neq g=(g_1,g_2)\in\GL_{n/2}(q)^2 \cap G$ such that $g_1$ is a Singer cycle and $g_2$ is a diagonal element satisfying $\det(g_1)\det(g_2)=1$.
		Note that $V$ does not decompose as the sum of two equivalent $\F_q\gen{g^2}$-modules. 
		In particular, a conjugate of $g$ cannot belong to $M\sm \GL_{n/2}(q)^2$, so $\fp(g,G/M)=1$.
		
		The remaining cases are $(n,m)=(3,1), (4,1)$, with $q\ge 5$. 
		Assume $n=3$, so $q\ge 7$ in view of \Cref{l:small_groups}.
		Let $g$ be a diagonal element with eigenvalues $(\lambda, \lambda^{-1},1)$ where $\lambda$ generates $\F_q^\times$. Then $\fp(g,G/M)=1$, since $g^j$ does not act homogeneously on a $j$-space for $j=2,3$.
		
		Assume finally $n=4$. For $q=5$, let $g\in M$ be a regular (semisimple) element inducing a $3$-cycle on the spaces; concretely, for example, $g=(\lambda,1,1,-1)\tau$ where $\tau=(1,2,3)$ and $|\lambda|=4$.  For $q\ge 7$, let $g\in \GL_1(q)^4\cap G$ be regular (semisimple).
		In both cases we have $|C_G(g)| < (q^4-1)/(q-1)$, and  $(q^4-1)^3/(q-1)^3 < |G:M|$. 
		
		Assume now $G=\Omega^+_n(q)$ and $M$ of type $\Or^+_m(q)\wr S_t$. In most cases, it will be sufficient to note that $M$ contains a suitable regular element. First note that a $t$-cycle $\tau$ of $S_t$ belongs to $\SO^+_n(q)\sm \Omega^+_n(q)$ if and only if $q\equiv 3 \pmod 4$ and
		$m\equiv 2\pmod 4$ and $t$ is even, and otherwise $\tau\in \Omega^+_n(q)$. This follows from the fact that $\tau$ fixes complementary totally singular subspaces $A$ and $B$, and has determinant $(-1)^{m(t-1)/2}$ on both; recall that in this situation,  $\Omega^+_n(q)$ consists of the elements whose determinant on $A$ (and $B$) is a square in $\F_q$ (see for example \cite[Lemma 4.1.9]{kleidman1990liebeck}). Next, we have $(m,q)\neq (2,2),(2,3),(2,4),(4,2)$ by \cite[Proposition 2.3.6]{bray2013holt_colva}. In particular, if $(m,q)\neq (2,5), (4,3)$, we can choose $g=(g', 1, \ldots, 1)\tau \in \Omega^+_n(q)$ where as an $\F_q\gen{g'}$-module $V_1=A\oplus B$, where $A$ and $B$ are totally singular and irreducible, and $A\not\cong B$. Then $g$ has distinct eigenvalues, and the centralizer of $g$ in $\Or^+_n(q)$ is isomorphic to the centralizer of a regular element of $\GL_{n/2}(q)$; 
		therefore $|C_G(g)| < q^{n/2}$. Similarly to the $\SL_n(q)$ case, it can be easily checked that $q^{3n/2}<|G:M|$, which concludes the proof. We are left with $(m,q)= (2,5), (4,3)$. 
		In these cases, setting $W=V_1\perp \cdots \perp V_{t-1}$, according to the decomposition $V=W\perp V_t$ we choose $g=g_1\perp g_2$, where $g_1 = (g', 1,\ldots, 1)\tau\in \SO^+_{n-m}(q)\sm \Omega^+_{n-m}(q)$ is defined as above (in particular, note that $g'\in \SO^+_m(q)\sm \Omega^+_m(q)$ and $|g'|=4$ for $(m,q)=(2,5)$, and $|g'|=8$ for $(m,q)=(4,3)$), and $g_2$ is defined as follows: If $(m,q)=(2,5)$ then $g_2=g'$, and if $(m,q)=(4,3)$ then as an $\F_q\gen{g_2}$-module $V_t = C\perp D$, where $C$ and $D$ are of minus type, $g_2$ is trivial on $D$ and has order $4$ on $C$. Next we divide into the two cases. Assume first $(m,q)=(4,3)$; note that $g$ is regular, and $|C_G(g)|<2q^{n/2-2}(q+1)^2$, which gives the conclusion for $n\ge 12$. For $n=8$, it is easy to see that $\fp(g,G/M)=1$ (since $V$ does not decompose as the sum of two equivalent $\F_q\gen{g^2}$-modules). 
		Assume then $(m,q)=(2,5)$ and denote by $\lambda$ a generator of $\F_5^\times$. Then $g$ is not necessarily regular; specifically, $g$ is regular if and only if $t$ is odd. 
		If $g$ is not regular, then for $\alpha=\lambda^{\pm 1}$ there are two $\alpha$-Jordan blocks, of sizes $1$ and the $5$-part of $t-1$. In particular, by working in $\GL_{n/2}(q)$ we see that $|C_G(g)|<q^{n/2+2}$. We have $q^{3n/2+6}<|G:M|$ and the result follows.
		
		The other cases are similar; we will list the choice of the elements, generally omitting the details of the calculations. If $G=\SU_n(q)$ and $M$ is of type $\GU_m(q)\wr S_t$, we have $(m,q)\neq (2,2)$ by \cite[Proposition 2.3.6]{bray2013holt_colva}. We choose $g=(g', 1\ldots, 1)\tau$ where $g'$ is irreducible for $m$ odd, and irreducible on a hyperplane for $m$ even, and $\tau$ is a $t$-cycle. We have that $g$ is regular and so $|C_G(g)|<(q+1)^n$ by \Cref{l:order_centralizers}(iv). Similarly to the $\SL_n(q)$ case we  reduce to the cases $n\le 4$ or $(n,m)=(6,3)$. If $m=n/2$ then we choose $g=g_1\perp g_2 \in \GU_{n/2}(q)^2\cap G$ where $|g_1|=q^3+1$ for $n=6$ (resp. $|g_1|=q^4-1$ for $n=4$) and $g_2$ is diagonal; we have $\fp(g,G/M)=1$. The remaining cases are $(n,m)=(3,1),(4,1)$; in view of \Cref{l:small_groups} we have $q\ge 7$ and $q\ge 5$ in the respective cases and we can choose a diagonal regular (semisimple) element $g$, such that moreover $\fpr(g,G/M)=1$ for $n=3$.
		If $G=\Sp_n(q)$ and $M$ of type $\Sp_m(q)\wr S_t$, then $(m,q)\neq (2,2)$ by \cite[Proposition 2.3.6]{bray2013holt_colva}. We choose $g=(g', 1, \ldots, 1)\tau$ where $g'$ is irreducible and $\tau$ is a $t$-cycle; then $g$ is regular and $|C_G(g)|<(q+1)^{n/2}$ (since $g$ does not have eigenvalues $\pm 1$; see the proof of \Cref{l:order_centralizers}(iv)). 
		As above, we reduce to the case $(n,m)=(4,2)$.
		We have $q\ge 5$ by \Cref{l:small_groups} and 
		we can choose a regular semisimple element $g$ of $\Sp_m(q)^t$ of order $q+1$ with $\fp(g,G/M)=1$.
		Assume now $G=\Omega^\varepsilon_n(q)$ and $M$ of type $\Or^-_m(q)\wr S_t$, so $\varepsilon =(-1)^t$. If $q$ is odd and $(m,q)\neq (2,3)$, we choose $g=(g', 1, \ldots, 1)\tau \in \Omega^\varepsilon_n(q)$ where $g'$ irreducible, so $g$ is regular. (Note $\det(\tau)=1$, and since $(m,q)\neq (2,3)$, we can choose $g\in \Omega^-_m(q)$ or $\SO^-_m(q)\sm \Omega^-_m(q)$ so that $g\in \Omega_n^\varepsilon(q)$.)
		If  $(m,q)= (2,3)$, we write $V=W\perp V_t$ where $W=V_1\perp \cdots \perp V_{t-1}$, and accordingly we choose $g=g_1\perp g_2$ where $g_1=(g', 1, \ldots, 1)\tau\in \SO^\varepsilon_{n-2}(q)$ is as above, and either $g_2=1$ or $|g_2|=4$, so that $g\in \Omega^\varepsilon_n(q)$. If $g_2=1$ then $g$ is regular; if $|g_2|=4$ then the dimension of the centralizer in $\Or_n(K)$ is at most $n/2+2$, which gives the conclusion. Assume then $q$ is even. Then we choose $g=(g', 1, \ldots, 1)\tau\in \Omega^\varepsilon_n(q)$ where $g'$ is either irreducible, or irreducible on a nondegenerate codimension $2$ subspace and a reflection on the perpendicular complement (according to whether $\tau\in \Omega^\varepsilon_n(q)$ or not). Then $g$ is regular.  Assume now $M$ is of type $\Or_m(q)\wr S_t$ where $mq$ is odd and $G=\Omega^\varepsilon_n(q)$ with $\varepsilon \in \{+,-,\circ\}$. Suppose first $m\ge 3$; we have $(m,q)\neq (3,3)$ by \cite[Proposition 2.3.6]{bray2013holt_colva}. We may then choose $g=(g', 1, \ldots, 1)\tau \in \Omega^\varepsilon_n(q)$ where $g'$ is irreducible on a hyperplane and $\pm 1$ on the complement. Assume now $m=1$. Then the $t$-cycle $\tau$ has determinant $(-1)^{t-1}$. If $t$ is odd then choose $g=\tau\in \Omega^\varepsilon_n(q)$, so $g$ is regular.
		If $t$ is even then choose $g=\tau'$ a $(t-1)$-cycle, so $g$ is also regular since it has at most two $1$-Jordan blocks, one of which of size one (see \cite[Theorem 3.1]{liebeck_seitz_2012unipotent}). Assume now $M$ is of type $\Or_{n/2}(q)^2$ where $nq/2$ is odd. Then we may choose $g=g_1\perp g_2$ where $g_1$ and $g_2$ are regular and $g$ is also regular. Finally, if $M$ is of type $\GL_{n/2}(q^u).2$ with $G=\Sp_n(q),\Omega^+_n(q)$ and $u=1$, or $G=\SU_n(q)$ and $u=2$, then we may choose an irreducible element $g$ of $\GL_{n/2}(q^u)$ that is regular in $G$ and the result follows.
	\end{proof}
	
	\subsection{Class $\ca C_3$} Subgroups in class $\ca C_3$ are normalizers of subfields $E$ of $\End(V)$, where $E$ is a field extension of $\F_{q^u}$ of degree $r$ and $r$ divides $n$. (Equivalently, these are the normalizers of $E^\times$, which is a cyclic group of order $(q^u)^r-1$ acting homogeneously on $V$ with irreducible submodules of dimension $r$.)  We have $C_{\GL_n(q)}(E)\cong \GL_{n/r}(q^r)$ and $N_{\GL_n(q)}(E)\cong \GL_{n/r}(q^r)\rtimes \text{Gal}(\F_{q^r}/\F_q)$. We call $N_{\GL_n(q)}(E)$ an $r$-extension field subgroup of $\GL_n(q)$. We record the following

	\begin{lemma}
		\label{l:basic_extension}
		Let $g\in G=\GL_n(q)$ and assume $C_G(g)=C_G(g^r)$. Then, the number of $r$-extension field subgroups of $G$ containing $g$ is equal to the number of cyclic subgroups $R$ of $C_G(g)$ order $q^r-1$ acting homogeneously on $V=\F_q^n$.
	\end{lemma}
	
	\begin{proof}
		Simply note that if $g\in N_G(R)$  then $g^r \in C_G(R)$, and so $R\le C_G(g^r)=C_G(g)$.
	\end{proof}

	\begin{lemma}
		\label{l:classical_C3}
		\Cref{t:main_simple} holds if $M\in \ca C_3$.
	\end{lemma}
	
	\begin{proof}
		Let us start with $G=\SL_n(q)$, so $M$ is of type $\GL_{n/r}(q^r)$ for a prime divisor $r$ of $n$. Let $g\in M$ act irreducibly and of order $(q^n-1)/(q-1)$. If $n\ge 3$ or $(n,q)\neq (6,2)$, then by Zsigmondy's theorem $|g|$ is divisible by a ppd of $q^n-1$, so $g^r$ acts irreducibly. This holds also if $n=2$ or $(n,q)=(6,2)$, so $\fp(g,G/M) = 1$ by \Cref{l:basic_extension} (or \cite[Lemma 2.12]{breuer2008probabilistic}). 
		
		Assume now $G=\SU_n(q)$ and $M$ of type $\GU_{n/r}(q^r)$ for an odd prime divisor $r$ of $n$. If $n/r$ is odd, then let $g\in M$ be an element of order $(q^n+1)/(q+1)$ acting irreducibly on $V$; as above $\fp(g,G/M)=1$. If $n/r$ is even, we let $g\in M$ be of order $(q^n-1)/(q+1)$, acting irreducibly on each of a pair of complementary maximal totally singular spaces $A$ and $B$. We see that $A\not\cong B$ as $\F_{q^2}\gen{g^r}$-modules. In particular, setting $L=\GL_n(q^2)$ we have that $C_L(g)=C_L(g^r)=H\times K$ where $H$ and $K$ are Singer cycles on $A$ and $B$ (and $H$ and $K$ have the same eigenvalues). Then, by \Cref{l:basic_extension} the number of $r$-extension field subgroups of $L$ containing $g$ is $r$, so $\fp(g,G/M)\le r$ and the result follows.

		Assume next $G=\Sp_n(q)$ and $M$ of type $\Sp_{n/r}(q^r)$. We choose $g\in M$ of order $q^{n/2}+1$. Since $(n,q)\neq (6,2)$ (see \Cref{l:small_groups}), by Zsigmondy's theorem we have that $|g|$ is divisible by a ppd of $q^n-1$, so $\fp(g,G/M)=1$ by  \Cref{l:basic_extension}.
		
		Assume $G=\Omega^+_n(q)$ and $M$ of type $\Or^+_{n/r}(q^r)$, so $n/r\ge 4$ since $M\in \ca A$. Then we let $g\in \Omega^+_{n/r}(q^r)$ be of order $(q^{n/2}-1)/(2,q-1)$, acting irreducibly on each of a pair of complementary maximal totally singular subspaces $A$ and $B$. We deduce $\fp(g,G/M)\le r$ as in the unitary case.

		Now let $G=\Omega^-_n(q)$. If $M$ is of type $\Or^-_{n/r}(q^r)$ with $n/r\ge 4$, we let $g\in \Omega^-_{n/r}(q^r)$ be of order $(q^{n/2}+1)/(2,q-1)$, and then $\fp(g,G/M)=1$. If $M$ is of type $\Or_{n/r}(q^r)$ with $n/r\ge 3$ odd, then we choose $g\in \Omega_{n/r}(q^r)$ of order $(q^{(n-r)/2}+1)/(2,q-1)$, acting irreducibly on a nondengenerate $(n-r)$-space $A$ and trivially on $B=A^\perp$. Setting $L=\GL_n(q)$, we have that $C_L(g)=C_L(g^r)=H\times \GL_r(q)$ where $H$ is a Singer cycle on $A$. Then, by \Cref{l:basic_extension} the number of $r$-extension field subgroups of $L$ containing $g$ is equal to the size of the conjugacy class of a Singer cycle in $\GL_r(q)$, which is $s:=|\GL_r(q)|/(q^r-1)$, so $\fp(g,G/M)\le s$. 
		
		Finally, if $M$ is of type $\GU_{n/2}(q)$ and $G$ is symplectic or orthogonal, we choose $g$ of order $(q^{n/2}+1)/(2,q-1)$ if $n$ is odd, and $(q^{(n-1)/2}-1)/(2,q-1)$ if $n$ is even; we have $\fp(g,G/M)\le r$. 
	\end{proof}
	
	\subsection{Class $\ca C_4$} Subgroups $M$ in class $\ca C_4$ are stabilizers of tensor product decompositions $V=V_1\otimes V_2$ where $\dim(V_i)=n_i$ and $n=n_1n_2$. If $G=\SL_n(q)$ then $n_1\neq n_2$. If $V_1$ and $V_2$ are equipped with a nondegenerate form then they are not similar; so in all cases $M\le \GL(V_1)\otimes \GL(V_2)$. 
	
	\begin{lemma}
		\Cref{t:main_simple} holds if $M$ is in class $\ca C_4$.
	\end{lemma}
	
	\begin{proof}
		Assume first $G=\SL_n(q)$ and $M$ of type $\GL_{n_1}(q)\otimes \GL_{n_2}(q)$ where $1<n_1< n_2<n$ and $n_1n_2=n$. In particular, $n\ge 6$. Let $g=g_1\otimes g_2\in M$ where $g_1\in \SL_{n_1}(q)$ is regular semisimple (e.g., irreducible), and $g_2\in \SL_{n_2}(q)$ is regular unipotent. Then $g$ is regular, 
		and so $|C_G(g)|\le q^n-1$,
		from which    
		\[
		|M|  |C_G(g)|^3< q^{n_1^2+n_2^2+3n}<|G|
		\]
		and the conclusion follows from \Cref{l:fpr_enough}. 
		
		Assume now $G=\Sp_n(q)$, so $M$ is of type $\Sp_{n_1}(q)\otimes \Or^\varepsilon_{n_2}(q)$ with $n_1n_2=n$, $q$ odd, 
		and $\varepsilon \in \{+,-,\circ\}$.
		Since $M\in \ca A$, we have $n_2\ge 3$ (see \cite[Proposition 4.4.4]{kleidman1990liebeck}). Now we let $g=g_1\otimes g_2\in M$, where $g_1\in \Sp_{n_1}(q)$ is unipotent with a single Jordan block
		and $g_2\in \Or^\varepsilon_{n_2}(q)$ is  irreducible if $\varepsilon=-$, irreducible on complementary totally singular subspaces if $\varepsilon =+$, irreducible on a nondegenerate hyperplane if $\varepsilon =\circ$. 
		Then $g$ is regular, and moreover $|C_G(g)|\le 2(q+1)^{n/2}$ (since $g$ does not have both $1$ and $-1$ as eigenvalues; see the proof of \Cref{l:order_centralizers}(iv)). We then see that  $|M||C_G(g)|^3<|G|$.
		
		The other cases are $G=\SU_n(q)$ and $M$ of type $\GU_{n_1}(q)\otimes \GU_{n_2}(q)$; $G=\Omega^+_n(q)$ and $M$ of type $\Sp_{n_1}(q)\otimes \Sp_{n_2}(q)$ or $\Or^{\xi_1}_{n_1}(q)\otimes \Or^{\xi_2}_{n_2}(q)$ (in the latter case, $q$ odd and $(\xi_1, \xi_2)\neq (\circ, -), (-, \circ)$); $G=\Omega_n(q)$ with $nq$ odd and $M$ of type $\Or_{n_1}(q)\otimes \Or_{n_2}(q)$; $G=\Omega^-_n(q)$ with $q$ odd and $M$ of type $\Or_{n_1}(q)\otimes \Or^-_{n_2}(q)$ with $n_1$ odd. Except when $M$ is of type $\Or^{\xi_1}_{n_1}(q)\otimes \Or^{\xi_2}_{n_2}(q)$ with $n_1$ and $n_2$ even, we can choose $g=g_1\otimes g_2$ where one element $g_i$ is unipotent with one Jordan block, and the other element $g_j$ is semisimple with distinct eigenvalues, so $g$ is regular. In the remaining case, we can choose $g_1$ unipotent with two Jordan blocks of size $n_1/2$, and $g_2$ semisimple with distinct eigenvalues. Then the dimension of the centralizer of $g$ in $\Or_n(K)$ is $n$, and the conclusion follows.
	\end{proof}
	
	\subsection{Class $\ca C_5$} Subgroups in class $\ca C_5$ are classical subgroups defined over a proper subfield of $\F_{q^u}$. The possibilities are listed in \cite[\S 4.5]{kleidman1990liebeck}.
	
	\begin{lemma}
		\Cref{t:main_simple} holds if $M$ is in class $\ca C_5$.
	\end{lemma}
	
	\begin{proof}
		Assume first $G=\SL_n(q)$, so $M$ is of type $\GL_n(q_0)$ where $q=q_0^b$ and $b$ is prime. Let $g\in M$ be of order $(q_0^n-1)/(q_0-1)$. Setting $M_0=\SL_n(q_0)$, we claim that $g^G\cap M\subseteq M_0$. Indeed, let $x\in \GL_n(q_0)$ and $z\in Z(\GL_n(q))$, and assume $(xz)^a = g$ with $a\in \GL_n(q)$. Then $x^a$ and $g$ belong to the same maximal torus $T$ of $\GL_n(q)$, and moreover they are defined over $\F_{q_0}$. In particular, either $x^a$ acts irreducibly over $\F_{q_0}$, or it acts diagonally over $\F_{q_0}$, the latter case occurring only when $n=b$.
		In the first case, $x^a\in \gen g$ and so $z=x^{-a}g\in \gen g\le \SL_n(q_0)$, and so $xz\in \GL_n(q_0)$, as desired. In the second case, we have that if $\alpha$ is an eigenvalue of $g$ then $\alpha^{q_0-1}\in \F_{q_0}$,
		which is easily seen to be impossible unless $(q_0,n)=(3,2)$. 
		Since $b=n$, in this exceptional case we have $q=9$, which has been handled in \Cref{l:small_groups}. 
		Therefore $g^G\cap M\subseteq M_0$, as claimed. What is more, $g^G\cap M= g^{M_0}$, since all semisimple elements of $M_0$ with the same characteristic polynomial are conjugate in $M_0$, and so by \Cref{l:fixed_points_precise_normal_subgroup} we get
		\[
		\fp(g,G/M)= \frac{|C_G(g)|}{|C_{M_0}(g)||M:M_0|} \le \frac{(q^n-1)(q_0-1)}{(q-1)(q_0^n-1)|M:M_0|} < \frac{(q/q_0)^{n-1}}{|M:M_0|}.
		\]
		We also have 
		\[
		|G:M|= \frac{|G:M_0|}{|M:M_0|} >\frac{(q/q_0)^{n^2-1}}{|M:M_0|}.
		\]
		In particular, using $|M:M_0|\ge 1$ we have $\fp(g,G/M)^3 < |G:M|$.

		Assume next $G=\SU_n(q)$ and $M$ is of type $\GU_n(q_0)$ where $q=q_0^b$ and $b$ is an odd prime. We have $(n,q) \neq (3,8)$ in view of \Cref{l:small_groups}. Assume first also $(n,q)\neq (4,8)$. If $n$ is odd, let $g\in M$ be of order $(q_0^n+1)/(q_0+1)$, and if $n$ is even, let $g\in M$ be of order $q_0^{n-1}+1$. (The exclusion of the two above cases ensures that $|g|$ is divisible by a ppd of $q_0^{2n}-1$ or $q_0^{2n-2}-1$.) 
		Setting $M_0=\SU_n(q_0)$, by the same argument as in the $\SL_n(q)$ case we have $g^G\cap M=g^{M_0}$, and so
		\[
		\fp(g,G/M)\le \frac{(q^n+1)(q_0+1)}{(q+1)(q_0^n+1)|M:M_0|}.
		\]
		Moreover $|G:M|=|G:M_0|/|M:M_0|$, and using $n\ge 3$ we see that $\fp(g,G/M)^3< |G:M|$. Recall now that we excluded the case $(n,q)=(4,8)$. In this case we have $M=\SU_4(2)$ and we choose $g\in M$ of order $5$.
		Then $g^G\cap M=g^M$ and 
		\[
		\fp(g,G/M)= \frac{|C_G(g)|}{|C_M(g)|} = \frac{455}{5} < |G:M|^{1/3},
		\]
		which concludes the proof in this case. The other cases are similar. If $G=\SU_n(q)$ and $M$ of type $\Or^\varepsilon_n(q)$ we choose $g$ of order $q^{n/2}+1$ or $q^{(n-1)/2}+1$ or $q^{n/2}-1$ (depending on $\varepsilon$). If $G=\SU_n(q)$ and $M$ of type $\Sp_n(q)$, we choose $g$ of order $q^{n/2}+1$. If $G=\Sp_n(q)$ and $M$ of type $\Sp_n(q_0)$, where $q=q_0^b$, we choose $g$ of order $q_0^{n/2}+1$. If $G=\Omega^\varepsilon_n(q)$ and $M$ of type $\Or^\xi_n(q)$ where $q=q_0^b$ and $\varepsilon = \xi^b$, we choose $g$ of order $(q_0^{n/2}+1)/(2,q-1)$ or $(q_0^{(n-1)/2}+1)/(2,q-1)$ or $(q_0^{n/2}-1)/(2,q-1)$ (depending on $\xi$).
		In all cases, a similar argument as above applies.
	\end{proof}
	
	\subsection{Class $\ca C_6$} 
	Subgroups $M$ in class $\ca C_6$ are normalizers of symplectic type subgroups.
	
	Let $r$ be a prime, and $E$ an $r$-group of symplectic type such that $|E| = r^{2m+1}$, $E$ is of exponent $r\cdot (2,r)$ and $E$ is as in Table \ref{c6tab} (see \cite[\S 4.6]{kleidman1990liebeck}). We have $n=r^m$ and $V = V_n(q^u)$ is a faithful absolutely irreducible $\F_{q^u}E$-module, where 
	$q^u \equiv 1 \hbox{ mod }r$. Moreover, $M \le N_{\GL(V)}(E)$, where $N_{\GL(V)}(E)/EZ$ is as in the table (here $Z=Z(\GL(V))$).
	
	\begin{table}[h!] \caption{} \label{c6tab}
		\[
		\begin{array}{|l|l|l|}
			\hline
			E & n & N_{\GL(V)}(E)/EZ  \\
			\hline
			r^{1+2m},\,r \hbox{ odd} & r^m & \Sp_{2m}(r) \\
			4 \circ 2^{1+2m} & 2^m & \Sp_{2m}(2)  \\
			2_{\pm}^{1+2m} & 2^m & \Or^{\pm}_{2m}(2) \\
			\hline
		\end{array}
		\]
	\end{table}
	We begin with a known fact, cf. 
	\cite[p. 706]{shult1965groups}. This reference addresses only the case of elements of prime power order, but the proof works in higher generality, as we now record.

	\begin{lemma}
		\label{l:semiregular_hall_higman}
		With notation as above, let $g\in N_{\GL(V)}(E)$ act with all orbits of size $|g|$ on the nonzero vectors of $E/Z(E)\cong \F_r^{2m}$. Then the centralizer of $g$ in $\GL_n(K)$ has dimension $1+(r^{2m}-1)/|g|$. In particular, if $|g|=r^m+1$ then $g$ is regular in $\GL_n(K)$.
	\end{lemma}
	
	\begin{proof}
		Let  $V_K:=K^n$, let $K[E]$ be the group algebra of $E$  and let $\phi\colon K[E] \to \End_K(V_K)$ be the morphism corresponding to the $K[E]$-module $V_K$; this is also a morphism of $K\gen g$-modules. Since $E$ acts irreducibly, $\phi$ is surjective. Moreover, we have $J:=\mathrm{Ker}(\phi) = (z-\zeta 1)K[E]$, where $Z:=Z(E) = \gen z$ and $z$ acts by the scalar $\zeta$ on $V_K$. Therefore $\End_K(V_K)\cong K[E]/J$ as $K\gen g$-modules. Let $\{a_1Z, \ldots, a_tZ\}$  be an orbit of $g$ on $E/Z(E) \sm \{0\}$, so $t=|g|$.  Then there exist $z_1, \ldots, z_t\in Z$ such that $\{a_1z_1, \ldots, a_tz_t\}$ is an orbit of $g$ on $E$. Now, any set of representatives for the cosets of $Z$ in $E$ maps (via the natural projection) to a $K$-basis of $K[E]/J$.
		In particular, there exists a $K$-basis $v_1, \ldots, v_{n^2}$ of $K[E]/J$ such that $v_1g=v_1$ and such that $g$ permutes the other $v_i$ in orbits of size $|g|$. The lemma follows.
	\end{proof}

	\begin{lemma}
		\Cref{t:main_simple} holds if $M$ is in class $\ca C_6$.
	\end{lemma}
	
	\begin{proof}
		Assume $G=\SL_n(q)$, so $n=r^m$ with $r$ prime and $r\neq p$.
		If $r$ is odd then $q\equiv 1\pmod r$; and if $r=2$ and $n>2$, then $M$ is of type $E.\Sp_{2m}(2)$, where $E=C_4\circ 2^{1+2m}$, and $q \equiv 1 \pmod 4$. For $n=2$ it is also possible that $q \equiv 3 \pmod 4$ and $M \le Q_8.S_3$.
		
		Assume first $n=2$. Then, letting $E$ be the quaternion group of order $8$, we have $M=E.S_3=E.\Or^-_2(2)$ or $M=E.3$. If $M=E.S_3$ (resp. $M=E.3$), let $g\in M$ be an element of order $3$ (resp. $4$). Then $g$ is regular in $G$ and so $|C_G(g)|\le q+1$. Moreover, letting $\delta$ be the number of elements of $M$ of order $3$ (resp. $4$) in $M$, we have $\delta=8$ (resp. $\delta=6$), so $|g^G\cap M|\le \delta$ and 
		\[
		\fp(g,G/M)=\frac{|g^G\cap M||G:M|}{|g^G|}\le \frac{\delta(q+1)}{|M|} < |G:M|^{1/3}. 
		\]
		Assume next $n=3$, so $E.Q_8\le M\le E:\Sp_2(3)$, where $E$ is an extraspecial group of order $27$ and exponent $3$. We may assume $q\ge 7$ in view of \Cref{l:small_groups}.
		Let $g\in M$ be of order $4$, so $g$ is regular semisimple in $G$ by \Cref{l:semiregular_hall_higman}, and $|C_G(g)|\le q^2-1$.
		The number of elements of $M$ of order $4$ is $54$. (Indeed, letting $Z$ be the center of $M$, we have $M/Z \ge \F_3^2\rtimes Q_8$, which has $9\cdot 6$ elements of order $4$,
		and each of these admits a unique lift in $G$ of order $4$.)
		Therefore
		\[
		\fp(g,G/M) =\frac{|g^G\cap M||G:M|}{|g^G|} \le \frac{54(q^2-1)}{|M|} < |G:M|^{1/3},
		\]
		since $|M|\ge 27\cdot 8$. Now assume $n=4$, so $(4\circ E).A_6\le M\le (4\circ E).\Sp_4(2)$, where $E$ is any extraspecial group of order $2^5$. Then by \Cref{l:semiregular_hall_higman} an element of $M$ of order $5$ is regular in $G$, and a similar calculation as above suffices. Finally, assume $n\ge 5$, so $M=ZE.\Sp_{2m}(r)$ where $Z=Z(\SL_n(q))$ (see \cite[Propositions 4.6.5 and 4.6.6]{kleidman1990liebeck}). Let $g\in M$ be of order $r^m+1$, so $g$ is regular in $\GL_n(K)$ by \Cref{l:semiregular_hall_higman}. Moreover, we have $C_{\GL_n(q)}(g)\SL_n(q)=\GL_n(q)$, and so $|C_{G}(g)|\le (q^n-1)/(q-1)$. 
		Therefore
		\[
		|M||C_{G}(g)|^3< (q-1) r^{2m^2+m}r^{2m+2}\frac{(q^n-1)^3}{(q-1)^3} < |G|.
		\]
        
		Assume now $G=\Sp_n(q)$, so $E.\Omega^-_{2m}(2)\le M \le E.\Or^-_{2m}(2)$, where $n=2^m$ and $E=D_8\circ \cdots \circ D_8\circ Q_8$ is an extraspecial group 
		of order $2^{2m+1}$. Let $g\in M$ be of order $2^m+1$, so $g$ is regular in $\GL_n(K)$ by \Cref{l:semiregular_hall_higman}, and  the dimension of the centralizer of $g$ in $\Sp_n(K)$ is $n/2$, and moreover $|C_G(g)|<4(q+1)^{n/2}$ by \Cref{l:order_centralizers}(iv). For $n\ge 8$ we see that  $|M||C_G(g)|^3 < |G|$ and the result follows. Assume then $n=4$. We have $q\ge 5$ in view of \Cref{l:small_groups}. We have $|g|=5$, and if $p\ne 5$ then $|C_G(g)|\le q^2+1$, and if $p=5$ then $|C_G(g)|=2q^2$. The number of elements of $E.\Or^-_4(2)=E.S_5$ of order $5$ is $16\cdot 24=384$,
		so
		\[
		\fp(g,G/M)=\frac{|g^G\cap M||G:M|}{|g^G|} \le \frac{384 \cdot 2q^2}{|M|},
		\]
		which is $<|G:M|^{1/3}$ for $q\ge 5$.
		
		The other cases are similar. Assume $G=\SU_n(q)$ and $M$ of type $E.\Sp_{2m}(r)$ where $E$ is extraspecial of order $r^{2m+1}$ or $E=C_4\circ R$ where $R$ is extraspecial of order $2^{2m+1}$. Exactly as in the linear case, if $n=3$ (resp. $n=4$, resp. $n\ge 5$) then we choose $g$ of order $4$ (resp. $5$, resp. $r^m+1$), so $g$ is regular in $\GL_n(K)$. Finally, assume $G=\Omega^+_n(q)$ and  $E.\Omega^+_{2m}(2)\le M \le E.\Or^+_{2m}(2)$, where $n=2^m$ and $E=D_8\circ \cdots \circ D_8$ is an extraspecial group
		of order $2^{2m+1}$. Let $U:=E/Z(E)$ and let $g\in M$ be of order $2^m-1$, such that as an $\F_2\gen g$-module $U=A/Z(E)\oplus B/Z(E)$, where $A/Z(E)$ and $B/Z(E)$ are totally singular and irreducible.
		By \Cref{l:semiregular_hall_higman}, the dimension of the centralizer of $g$ in $\GL_n(K)$ is $n+2$. By \Cref{l:compare_dim_centralizer_GL_Sp}(ii), it follows that the dimension of the centralizer of $g$ in $\Or_n(K)$ is at most $(n+2)/2$, and we conclude with a similar calculation as above. 
	\end{proof}
	
	\subsection{Class $\ca C_7$}  Subgroups $M$ class $\ca C_7$ are stabilizers of tensor product decompositions $V=V_1\otimes \cdots \otimes V_t$ where $m=\dim(V_i)$ and $n=m^t$. If the $V_i$ are equipped with a nondegenerate form then they are similar; in all cases we have $M/Z\le \PGL(V_1)\wr S_t$.
	
	\begin{lemma}
		\Cref{t:main_simple} holds if $M$ is in class $\ca C_7$.
	\end{lemma}
	
	\begin{proof}
		First, we have $n\ge 8$ since $M\in \ca A$ (see \cite{bray2013holt_colva}).  Let $G=\SL_n(q)$ and let $M$ be of type $\GL_m(q)\wr S_t$, with $n=m^t$ and $m>1$, $t>1$. Since $M\in \ca A$, we have $m\ge 3$ (see \cite[Section 4.7]{kleidman1990liebeck}). Let $g=x\otimes \cdots \otimes x\otimes y\in \GL_m(q)^{\otimes t}$ where $x$ has order $(q^m-1)/(q-1)$ and $y$ is regular unipotent. Then each generalized Jordan block of $g$ over $K$ has size $m$, and each such block occurs with multiplicity at most $m^{t-2}$. (A generalized Jordan block refers to a Jordan block corresponding to any eigenvalue in $K$.) Therefore, the dimension of the centralizer of $g$ in $\GL_n(K)$ is at most $m^{2t-2}$ (cf. proof of \Cref{l:optimization}). Hence $|C_G(g)|<q^{m^{2t-2}}$ and
		\[
		|M| |C_G(g)|^3< t! q^{m^2t +3m^{2t-2}}<|G|.
		\]
		Assume now $G=\Sp_n(q)$ and $M$ of type $\Sp_m(q)\wr S_t$, with $qt$ odd.
		Then we may choose an element as above but where $x$ has order $q^{m/2}+1$.
		Then, the dimension of the centralizer of $g$ in $\Sp_n(K)$ is at most $m^{2t-2}/2+n/2$ (see \cite[Theorem 3.1]{liebeck_seitz_2012unipotent}). Moreover, $g$ has $n/m=m^{t-1}$ generalized Jordan blocks over $K$, so by \Cref{l:order_centralizers}(iii) we have $|C_G(g)|<2^{m^{t-1}}q^{m^{2t-2}/2 + n/2} \le q^{m^{t-1}+ m^{2t-2}/2 + n/2}$, and so
		\[
		|M| |C_G(g)|^3< t!q^{m^2t/2 + mt/2 + 3m^{t-1}+ 3m^{2t-2}/2 + 3n/2}
		\]
		which is less than $|G|$ if $n>8$ (i.e., $n\ge 32$). If $n=8$, since $M\in A$ we have $q\ge 5$ (see \cite{bray2013holt_colva}), and $M=\Sp_2(q)^{\otimes 3}.2^2.S_3$ and $g$ has two $1$-Jordan blocks of dimension $2$, and one rational $f$-block of dimension $4$, where $f$ is the minimum polynomial of an element of order $(q+1)/2$ over $\F_q$. (A rational $f$-block refers to an indecomposable $\F_q \gen{g}$-submodule whose minimum polynomial is $f$.) In particular, the dimension of the centralizer of $g$ in $\Sp_n(K)$ is $2+4 =6$ (see \cite[Theorem 3.1]{liebeck_seitz_2012unipotent}), and $|C_G(g)|=q^4|\GU_1(q)||\Or^\pm_2(q)|\le 2q^4(q+1)^2$
		and we see that $|M||C_G(g)|^3<|G|$. 
		
		The other cases are similar. These are $G=\SU_n(q)$ and $M$ of type $\GU_m(q)\wr S_t$; $G=\Omega^+_n(q)$ and $M$ of type $\Or_m^\pm(q)\wr S_t$ ($q$ odd) or $\Sp_m(q)\wr S_t$ ($qt$ even); $G=\Omega_n(q)$ with $qn$ odd and $M$ of type $\Or_m(q)\wr S_t$. Except for the case where $M$ is of type $\Or_m^\pm(q)\wr S_t$, we can choose $g=x\otimes \cdots \otimes x\otimes y$ where $x$ is semisimple with distinct eigenvalues and $y$ unipotent with one Jordan block, as above. In the exceptional case, we have $m\ge 4$ and we choose $y$ unipotent with two Jordan blocks of size $m/2$. We have that all generalized Jordan blocks of $g$ over $K$ have size $m/2$, and each such block occurs with multiplicity at most $2m^{t-2}$, so the dimension of the centralizer of $g$ in $\GL_n(K)$ (resp. $\Or_n(K)$) is at most $2m^{2t-2}$ (resp. at most $m^{2t-2}$, see \cite[Theorem 3.1]{liebeck_seitz_2012unipotent}), and the conclusion follows.
	\end{proof}

	\subsection{Class $\ca C_8$} 
	
	Subgroups in class $\ca C_8$ are classical subgroups with the same natural module as $G$.

	\begin{lemma}
		\Cref{t:main_simple} holds if $M$ is in class $\ca C_8$.
	\end{lemma}
	
	\begin{proof}
		Assume first $G=\SL_n(q)$, so $M$ is of type $\GU_n(q^{1/2})$, $\Sp_n(q)$, or $\Or^\varepsilon_n(q)$. Assume first $M$ is of type $\GU_n(q^{1/2})$, so $n\ge 3$ and $M=\gen{M_0,y}$ where $M_0=\SU_n(q^{1/2})$ and $y\in \GU_n(q^{1/2})Z$ with $Z=Z(\GL_n(q))$. Then choose $g\in M_0$ of order $(q^{n/2}+1)/(q^{1/2}+1)$ if $n$ is odd, and order $q^{(n-1)/2}+1$ if $n$ is even. First note that $g^G\cap M= g^{M_0}$. Indeed, assume $g^a = xz$ where $a\in G$, $x\in \GU_n(q^{1/2})$ and  $z\in Z$. Then $x$ and $g^a$ belong to the same maximal torus $T$ of $\GL_n(q)$. If $n$ is odd then $T$ is cyclic and irreducible and, since $x$ and $g^a$ both belong to unitary subgroups, we deduce that they must belong to the same cyclic subgroup $C$ of $T$ order $q^{n/2}+1$. In particular $g^a\in C\le M$, which proves that $g^G\cap M \subseteq M\cap \SL_n(q) =M_0$. Moreover, $g^G\cap M=g^{M_0}$ since all irreducible elements of $M_0$ with the same characteristic polynomial are conjugate in $M_0$. The case where $n$ is even is analogous and so we have $g^G\cap M= g^{M_0}$ in all cases. Therefore by \Cref{l:fixed_points_precise_normal_subgroup} we get
		\[
		\fp(g,G/M)=\frac{|C_G(g)|} {|M:M_0||C_{M_0}(g)|} \le \frac{(q^{n-\delta}-1)}{A(q^{(n-\delta)/2}+1)}=\frac{q^{(n-\delta)/2}-1}{A}
		\]
		where $A=\delta=1$ if $n$ is even, and $A=q^{1/2}-1$ and $\delta=0$ if $n$ is odd.
		Next, we have $|M|\le (q^{1/2}-1)|M_0|$
		and so by \Cref{l:order_classical_groups} we deduce $|G:M|\ge |G:M_0|/(q^{1/2}-1)> 9q^{n^2/2-1/2}/(16(q^{1/2}-1))$, and we have $\fp(g,G/M)^3<|G:M|$ as soon as $n\ge 4$. Assume then $n=3$, so $q>5$ in view of \Cref{l:small_groups}. Then we have $M=\SU_3(q^{1/2})\times (3,q^{1/2}-1)$ (see \cite{bray2013holt_colva}), and we still have $\fp(g,G/M)^3<|G:M|$, by using the exact  value of $|G:M|$, rather than the above approximation.

		Assume next $G=\SL_n(q)$ and $M$ of type $\Or^\varepsilon_n(q)$ with $\varepsilon \in \{+,-,\circ\}$. Suppose $\varepsilon = +$, so $n\ge 4$ is even and $q$ is odd, and $M=\CO^+_n(q)\cap G$.
		Setting $M_0=\SO^+_n(q)$, we choose $g\in M_0$ of order $q^{n/2}-1$. Then $g$ is regular semisimple in $G$ (note that  $(q,n)\neq (2,4)$).
		No element of $\CO^+_n(q)\sm \Or^+_n(q)$ can have the same eigenvalues as $g$, so as above we have $g^G\cap M=g^{M_0}$
		and 
		\[
		\fp(g,G/M)\le \frac{|C_G(g)|} {|C_{M_0}(g)|} \le \frac{(q^{n/2}-1)^2}{(q-1)(q^{n/2}-1)} < |G:M|^{1/3}.
		\]
		Assume now $n\ge 3$ is odd and $M$ is of type $\Or_n(q)$. Setting $M_0=\SO_n(q)$, we choose $g\in M_0$ of order $q^{(n-1)/2}+1$ and similarly to above we have
		\[
		\fp(g,G/M)\le \frac{|C_G(g)|} {|C_{M_0}(g)|} \le \frac{q^{n-1}-1}{q^{(n-1)/2}+1} = q^{(n-1)/2}-1< |G:M|^{1/3}.
		\]
		The cases $G=\SL_n(q)$ and $M$ of type $\Or^-_n(q)$ or $\Sp_n(q)$ ar similar; in both cases we choose $g$ of order $q^{n/2}+1$.  Finally, for $G=\Sp_n(q)$ and $M=\Or^\varepsilon_n(q)$ with $\varepsilon \in \{+,-\}$ and $q$ even, letting $g\in M$ be an element without eigenvalue $1$, we have $\fp(g,G/M)=1$.
	\end{proof}
	
	\subsection{Class $\ca N$} Recall that class $\ca N$ was defined at the beginning of \Cref{sec:geometric_classes}.
	
	\begin{lemma}
		\label{l:omega_8_triality}
		\Cref{t:main_simple} holds if $G=\POm^+_8(q)$, or $G=\Sp_4(q)$ with $q$ even, and $M$ is in class $\ca N$.
	\end{lemma}
	
	\begin{proof}
		Assume first $G=\POm^+_8(q)$; we go through the possibilities given in \cite[Table 8.50]{bray2013holt_colva}, which is taken from \cite{kleidman1987maximal_omega8}. For convenience, we work in $G=\Omega^+_8(q)$, and put $d=(2,q-1)$. Using the notation as in the table, we are only concerned with novelties $\mathrm{N_1}$ and $\mathrm{N_4}$. If $M$ is parabolic, we use \Cref{l:parabolic}.  
		Assume now $M\cong d\times G_2(q)$. We choose $g\in \SL_3(q)<M$ of order a ppd of  $q^3-1$.
		Then $g$ preserves  a nondegenerate $6$-space $W$ of plus type (acting irreducibly on complementary maximal totally singular subspaces),
		so $|C_G(g)|\le (q^3-1)(q-1)$.
		Moreover $|C_M(g)|=d(q^3-1)/(q-1)$ and $g^G\cap M=g^M$ (note all cyclic subgroups of $G_2(q)$ of order $|g|$ are conjugate).
		Therefore
		\[
		\fp(g,G/M)=\frac{|C_G(g)|}{|C_M(g)|} \le \frac{(q-1)^2}{d} < |G:M|^{1/3}.
		\]
		Assume next $M=(\Omega^+_2(q)\times \frac 1d \GL_3(q)).[2d]$  with $q\ge 3$; then  $M$ is the stabilizer of $\{A,B\}$ where $A$ and $B$ are totally singular $3$-spaces with trivial intersection and $A\oplus B$ nondegenerate; see \cite[proof of Proposition 3.2.3]{kleidman1987maximal_omega8}. The subgroup $H=\Omega^+_2(q)\times  \frac 1d \GL_3(q)$ preserves $A$ and $B$. We let $g\in \SL_3(q)<H$ be of order a ppd of $q^3-1$, acting trivially on $(A\oplus B)^\perp$. Note that if $g$ lies in a conjugate of $M$, then it lies in the corresponding conjugate of $H$; but $g$ fixes precisely two totally singular $3$-spaces and so $\fp(g,G/M)=1$.
		
		Next let $M=(\Omega^-_2(q)\times \frac 1d \GU_3(q)).[2d]$  with $q\ge 3$. The subgroup $H= \Omega^-_2(q)\times \frac 1d \GU_3(q)$  is the stabilizer in $\GU_4(q^2)\cap G$ of a nondegenerate $1$-space for the unitary geometry (see \cite[proof of Proposition 3.2.2]{kleidman1987maximal_omega8}). We let $g\in H$ be of order a ppd of $q^6-1$ if $q\neq 2$, and of order $9$ if $q=2$.  Similarly to the previous case we have $\fp(g,G/M)=1$.
		
		Assume finally $M=(D_{2(q^2+1)/d})^2[2d].S_2$. Then $M$ is the normalizer of a Sylow $r$-subgroup of $G$, where $r$ is an odd prime divisor of $q^2+1$; we have that $M$ stabilizes a decomposition $V=V_1\perp V_2$ into $4$-spaces of minus type (see \cite[Proof of Proposition 3.3.1]{kleidman1987maximal_omega8}). Let $g\in M$ be of order $(q^2+1)/d$, acting irreducibly on $V_1$ and centralizing $V_2$. Then $\fp(g,G/M)$ is equal to the number of Sylow $r$-subgroups of $\Omega(V_2)\cong \PSL_2(q^2)$, which is $q^2(q^2-1)/2$, and
		\[
		\fp(g,G/M)^3=\frac{q^6(q^2-1)^3}{8} < |G:M|.
		\]
		The proof for $\Omega^+_8(q)$ is complete.
		
		Assume now $G= \Sp_4(q)$ with $q$ even; the classes in $\ca N$ can be found in \cite[Section 14]{aschbacher1984maximal_subgroups} (see also \cite{bray2013holt_colva}). If $M$ is parabolic, then we use \Cref{l:parabolic}. If $M=(C_{q-1})^2\colon D_8$, then $M$ is the normalizer of a nondegenerate $2$-space of plus type in $\SO^+_4(q)<G$. We let $g\in C_{q-1}^2< M$ be of order $q-1$ with eigenvalues $\lambda^{\pm 1}, \lambda^{\pm 2}$ on the natural module $V$;
		then $\fp(g,G/M)=1$, since $g$ fixes only two nondegenerate $2$-spaces.
		Assume next $M=(C_{q+1})^2\colon D_8$, so $M$ is the normalizer of a nondegenerate $2$-space of minus type in $\SO^+_4(q)<G$. We let $g\in M$ be of order $q+1$, with distinct eigenvalues on $V$, and similarly to the previous case we have $\fp(g,G/M)=1$. Assume finally $M=C_{q^2+1}\colon 4$, so $M$ is the normalizer of a subgroup $C_{q^2+1}$. Then we let $g\in M$ be of order $q^2+1$, and we have $\fp(g,G/M)=1$ as in the proof of \Cref{l:classical_C3}. 
		The proof is now complete. 
	\end{proof}

	\section{Classical groups: class $\ca S$}

	\label{sec:class_S}

	In this section we complete the proof of \Cref{t:main_simple}. The only remaining case is when $G$ is classical and $M$ is in the class $\ca S$ of subgroups. Throughout this section, $V=\F_q^n$ denotes the natural module for $G$, where $q$ is a power of the prime $p$.  
	(This is a change of notation from Section \ref{sec:geometric_classes}, where the natural module is $\F_{q^u}^n$ -- so in this section, the unitary case is $G = \PSU_n(q^{1/2})$.)
	Class $\ca S$ consist of the members $M$ of $\ca A(G)$ that do not belong to $\ca C_1, \ldots, \ca C_8, \ca N$. It was proved by Aschbacher \cite{aschbacher1984maximal_subgroups} (see also \cite{liebeck1998subgroup}) that these subgroups $M$ are almost simple. Moreover, if $L$ denotes the quasisimple cover of $S=\text{Soc}(M)$ acting faithfully on $V$, then $V$ is an absolutely irreducible $\F_q L$-module, which cannot be realized over a proper subfield of $\F_q$, and finally, if $G=\PSL_n(q)$ then $L$ fixes no nondegenerate unitary or bilinear form on $V$.

	\subsection{Strategy of proof} 
	\label{subsec:strategy}
	We outline now the strategy of proof of \Cref{t:main_simple} for $M$ in class $\ca S$.
	Letting $L$ be as in the previous paragraph, we seek an element $g\in L$ such that $C_G(g)$ is small. This is much less straightforward than in the previous section, since the embedding $L\hookrightarrow \SL_n(q)$ is not in a known list. As we shall briefly explain now, we will be able to construct such an element merely from algebraic properties of $L$. Our element $g$ will almost always be among the elements in \Cref{table:new_unique} and \Cref{table:sporadics}.
	
	In order to bound $|C_G(g)|$, a first key step is to bound $\dim(C_V(g))$, and for this we use generation properties of $L$ (borrowing ideas and results from \cite{guralnick2012malle} and earlier work). If $L$ is generated by two conjugates of $g$, then by the irreducibility of $L$ we clearly have $\dim(C_V(g))\le n/2$, which already is a useful bound. If $L$ has the stronger property of being generated by three conjugates of $g$ with product equal to 1, then in fact $\dim(C_V(g))\le n/3$ by Scott's lemma (\Cref{l:scott}). When $L=S$ is simple, \cite{guralnick2012malle} produces such an element. However, when $L\neq S$, this is not so straightforward, and we use invariable generation (\Cref{l:invariable generation}) together with a result of Gow (\cite{guralnick2015tiep_lifting}) in order to get to the same conclusion. In fact there are some cases for which neither of these  approaches works, and for these we have to content ourselves with the bound $\dim(C_V(g))\le n/2$. 
	
	For the case where $g$ is semisimple in $G$, as well as the bound for $\dim(C_V(g))$, we need to bound the dimensions of the nontrivial eigenspaces of $g$ on $V\otimes_{\F_q} K$; we will be satisfied with an upper bound of approximately $n/4$. In order to achieve this, we exploit the action of $N_L(\langle g \rangle)$ on $\gen g$ (see \Cref{l:coprime_schur}(v)) in conjunction with the basic \Cref{l:semiregular_eigenvalues}. When $g$ is not semisimple in $G$, we apply the Green correspondence (\Cref{l:green_correspondence}), using the fact that $\gen g$ contains a Sylow $p$-subgroup of $L$.
	
	The plan described above will be accomplished in \Cref{l:class_S_alternating,l:class_S}. In the cases where $S=\PSL_2(r)$ or $\PSL^\pm_3(r)$, the approach does not work; however, much is known about the representation theory of these groups, and we are able to amend the methods and argue somewhat more directly (\Cref{l:SL2}).

	\subsection{Some preliminary lemmas}

	We begin with some preliminary lemmas.
	
	\begin{lemma}
		\label{l:general_parameters_SL}
		Let  $M \le G:=\SL_n(q)$, let $g\in G$ be semisimple and assume that the dimensions of the nontrivial eigenspaces of $g$ occur with multiplicity at least $B$. Assume $c,C$ satisfy \eqref{eq:condition_first} or \eqref{eq:condition_second}, below, and assume $|M| \le q^C$, $\dim(C_V(g))\le cn$, and $cn+B \le n$. Then $|C_G(g)| < |G:M|^{1/3}$.
		\begin{align}
			&c(B+1)\le 2 \text{ \, and \, } n^2(B- 3)\ge B(C+2)	 \label{eq:condition_first}\\			&c(B+1)\ge 2 \text{ \, and \, } n^2\left(B-3c^2B-3 - 3c^2 +6c\right) \ge B(C+2) \label{eq:condition_second}
		\end{align}
	\end{lemma}

	\begin{proof}
		Clearly, $|C_G(g)| < q^d$ where $d$ is the dimension of the centralizer of $g$ in $\GL_n(K)$. Assume first \eqref{eq:condition_first}, so by \Cref{l:optimization_easier} (with $A=cn$) we have $d \le n^2/B$. Therefore, by \Cref{l:order_classical_groups} we get
		\[
		|M||C_G(g)|^3 < q^{C + 3n^2/B} \le q^{n^2-2} < |G|
		\]
		and the result follows. Case \eqref{eq:condition_second} is identical, using the relevant maximum in \Cref{l:optimization_easier}.
	\end{proof}

	\begin{lemma}
		\label{l:general_parameters_SU}
		Let  $M \le G:=\SU_n(q^{1/2})$, let $g\in G$ be semisimple and assume that the dimensions of the nontrivial eigenspaces of $g$ occur with multiplicity at least $B$. Assume $c,C$ satisfy \eqref{eq:condition_first_SU} or \eqref{eq:condition_second_SU}, below, and assume $|M| \le q^C$, $\dim(C_V(g))\le cn$, and $cn+B \le n$. Then $|C_G(g)| < |G:M|^{1/3}$.
		\begin{align}
			&c(B+1)\le 2 \text{ \, and \, } n^2(B- 3) \ge B(2C+3B+5) \label{eq:condition_first_SU}\\
			&c(B+1)\ge 2 \text{ \, and \, } n^2(B-3c^2B- 3 -3c^2 +6c) \ge B(2C+3B+5) \label{eq:condition_second_SU}
		\end{align}
	\end{lemma}

	\begin{proof}
		Put $q_0=q^{1/2}$. By \Cref{l:order_centralizers}(ii) we have $|C_G(g)| \le q_0^{d+E}$ where $d$ is the dimension of the centralizer of $g$ in $\GL_n(K)$ and $E$ is the number of distinct irreducible factors of the characteristic polynomial. Clearly $E$ is at most the number of  distinct eigenvalues of $g$. Assume first \eqref{eq:condition_first_SU}, so by \Cref{l:optimization} (with $A=cn$) we have $d +E \le B+1+n^2/B$. Therefore, by \Cref{l:order_classical_groups} we get
		\[
		|M||C_G(g)|^3< q_0^{2C + 3B+3+3n^2/B} \le q_0^{n^2-2} < |G|.
		\]
		Case \eqref{eq:condition_second_SU} is identical, using the relevant maximum in \Cref{l:optimization_easier}.
	\end{proof}

	\begin{lemma}
		\label{l:general_parameters_Sp}
		Let  $M \le G:=\Sp_n(q)$, let $g\in G$ be semisimple and assume that the dimensions of the nontrivial eigenspaces of $g$ occur with multiplicity at least $B$. Assume $c,C$ satisfy \eqref{eq:condition_first_Sp} or \eqref{eq:condition_second_Sp}, below, and assume $|M| \le q^C$, $\dim(C_V(g))\le cn$, and $cn+B \le n$. Then $|C_G(g)| < |G:M|^{1/3}$.
		\begin{align}
			&cn(B+1)+ B \le 2n \text{ \, and \, } n^2(B- 3) \ge  B( -n + 2C+ 3B +2)  \label{eq:condition_first_Sp}\\
			&cn(B+1)+B \ge 2n \text{ \, and \, } n^2(B-3c^2B- 3 -3c^2 +6c) \ge  nB\left(3c-1 \right) + B(2C+3B + 2)  \label{eq:condition_second_Sp}
		\end{align}
	\end{lemma}

	\begin{proof}
		Let $a$ be the dimension of the $1$-eigenspace of $g$, let $d$ be the dimension of the centralizer of $g$ in $\GL_n(K)$ and let $E_2$ be the number of distinct irreducible factors of degree at least $2$ of the characteristic polynomial of $g$. By \Cref{l:compare_dim_centralizer_GL_Sp,l:order_centralizers} we have $|C_G(g)| \le q^{d/2+a/2+E_2}$.  Next, note that $E_2$ is at most half the number of nontrivial distinct eigenvalues of $g$. Assume first \eqref{eq:condition_first_Sp}, so by \Cref{l:optimization_orthogonal} (with $A=cn$) we have $d+a+2E_2 \le B+ n^2/B$. Therefore, by \Cref{l:order_classical_groups} we get
		\[
		|M||C_G(g)|^3 < q^{C + 3n^2/(2B) + 3B/2} \le q^{n^2/2+n/2-1} < |G|.
		\]
		Case \eqref{eq:condition_second_Sp} is identical, using the relevant maximum in \Cref{l:optimization_easier}.
	\end{proof}

	\begin{lemma}
		\label{l:general_parameters_SO}
		Let  $M \le G:=\Omega^\varepsilon_n(q)$, let $g\in G$ be semisimple and assume that the dimensions of the nontrivial eigenspaces of $g$ occur with multiplicity  at least $B$. Assume $c,C$ satisfy \eqref{eq:condition_first_SO} or \eqref{eq:condition_second_SO}, below, and assume $|M| \le q^C$, $\dim(C_V(g))\le cn$, and $cn+B \le n$. Then $|C_G(g)| < |G:M|^{1/3}$.
		\begin{align}
			&cn(B+1)-B\le 2n \text{ \, and \, } n^2(B- 3) \ge  B(n + 2C+  3B +16) \label{eq:condition_first_SO}\\
			&cn(B+1)-B\ge 2n \text{ \, and \, } n^2(B-3c^2B- 3 -3c^2 +6c) \ge  nB\left(1 -3c\right)  +B(2C +3B+16) \label{eq:condition_second_SO}
		\end{align}
	\end{lemma}

	\begin{proof}
		Let $a,d,E_2$ be as in the proof of \Cref{l:general_parameters_Sp}; again, $E_2$ is at most half the number of nontrivial distinct eigenvalues of $g$. By \Cref{l:compare_dim_centralizer_GL_Sp,l:order_centralizers} we have $|C_G(g)| \le q^{d/2-a/2 +E_2+2}$. Assume first \eqref{eq:condition_first_SO}, so by \Cref{l:optimization_orthogonal} (with $A=cn$) we have $d -a + 2E_2\le B+n^2/B$. Therefore, by \Cref{l:order_classical_groups} we get
		\[
		|M||C_G(g)|^3< q^{C+ 3n^2/(2B) + 3B/2+6} \le q^{n^2/2-n/2-2} < |G|.
		\]
		Case \eqref{eq:condition_second_SO} is identical, using the relevant maximum in \Cref{l:optimization_easier}.
	\end{proof}
	
	\subsection{The proof} 
	
	We now fix some notation. Recall that $G$ is a simple classical group with  natural module $V=\F_q^n$, where $q$ a power of the prime $p$. We set $K=\overline{\F_p}$. We have that $M\in \ca A$ is in class $\ca S$, so $M$ is almost simple; we denote by $S$ the socle of $M$. We will find it convenient, in the proofs, to replace $G$ by the quasisimple cover acting faithfully on $V$, namely $\SL_n(q)$, $\SU_n(q^{1/2})$ or $\Sp_n(q)$ or $\Omega^\varepsilon_n(q)$. We will still denote by $M$ the preimage in $\SL_n(q)$, while we will denote by $L\le \SL_n(q)$ a quasisimple cover of $S$ with $S=L/Z(L)$.  
	For our choice of the element $g\in L$, we will keep throughout the following notation:
	\begin{itemize}
		\item	$d$ denotes the dimension of the centralizer of $g$ in $\GL_n(K)$.
		\item $a$ denotes the dimension of $C_V(g)$. 
	\end{itemize}	
	From time to time, for the reader's convenience we will recall this notation. In some cases below we will make use of the software GAP, and specifically of the GAP Character Table Library \cite{GAP_character_table}. When the modular character table of a group $L$ is available, then for an element of $g$ of $L$ of order not divisible by $p$ (i.e., an element semisimple in $\SL_n(q)$), we can compute the dimension of all eigenspaces of $g$ on $V\otimes_{\F_q} K$, which gives the value of $d$. This allows us to accurately estimate $|C_G(g)|$
	and deduce, for a suitable choice of $g$, the inequality $\fp(g,G/M)<|G:M|^{1/3}$. In most cases the bound $\fp(g,G/M)\le |C_G(g)|$ will be enough; in some cases, for very small $n$, we will use the equality $\fp(g,G/M)= |C_G(g)||g^G\cap M|/|M|$ (\Cref{l:fpr_enough}).
	
	We start with the case where $L = A_m$ and $V$ is the fully deleted permutation module for $L$, defined as follows. Let $\F_q^m$ be the usual permutation module, with submodules $A = \{x_1,\ldots,x_m) : \sum x_i = 0\}$ and $B = \langle (1,\ldots, 1) \rangle$; then $V = A/A\cap B$, of dimension $n=m-1$ if $p\nmid m$, and $n=m-2$ if $p\mid m$.
	
	\begin{lemma}
		\label{l:fully_deleted}
		\Cref{t:main_simple} holds if $M\in \{A_m,S_m\}$ and $M\le G$ via the fully deleted permutation module.
	\end{lemma}
	
	\begin{proof}
		Note that $M$ preserves a nondegenerate quadratic or bilinear form on $V$. Since $M\in \ca A$, it follows that $G$ is orthogonal or symplectic over $\F_p$; see for example \cite[p. 187]{kleidman1990liebeck} for the precise embeddings. We ignore the cases $G=\Sp_4(2)'\cong A_6$, $G=\Omega^-_4(2)\cong \SL_2(4)$, $G=\Omega_5(2)\cong \PSp_4(2)$, andqw $G=\Omega^\pm_6(2)\cong \PSL^\pm_4(2)$, which  have been handled in \Cref{l:small_groups}. Assume first $n=4$, so by \cite{bray2013holt_colva} we have $G=\Om^-_4(q)\cong \PSL_2(q^2)$ and $M=A_5$ with $q=p\ge 7$. Let $g\in M$ be of order $3$, so $|C_G(g)|=(q^2-1)/2$ and by \Cref{l:fpr_enough} we have $\fp(g,G/M) = (q^2-1)/6 < |G:M|^{1/3}$. Assume now $n=5$, so $G=\Omega_5(q)\cong \PSp_4(q)$ and either $A_6\le M \le S_6$ with $q=p\ge 5$ or $M=A_7$ with $q=p=7$. If $M=A_7$ we let $g\in M$ be of order $7$, so $g$ is regular unipotent and $|C_G(g)|=q^2$ and so $\fp(g,G/M) \le 2q^2/7 <|G:M|^{1/3}$. If $M\le S_6$ then we let $g\in M$ be of order $5$, so $g$ is regular in $\SL_n(q)$ and $|C_G(g)|\le (q+1)^2$ and $\fp(g,G/M)\le 2(q+1)^2/5 < |G:M|^{1/3}$. Assume then $n=6$, so $G=\Om^\pm_6(q)\cong e.\PSL^\pm_4(q)$ and $M=e\times A_7$ with $q=p\ge 11$ (here $e\in \{1,2\}$). Let $g\in M$ be of order $7$, so $|C_G(g)|\le  e(q+1)^3/2$; moreover $g^G$ splits into at most two $M$-classes and so by \Cref{l:fpr_enough} we have $\fp(g,G/M)\le (q+1)^3/7<|G:M|^{1/3}$. The case $n\ge 7$ is analogous. Specifically,  
		let $g\in A_m$ be an $(m-\delta)$-cycle, where $\delta=1$ if $m$ is even and $\delta=0$ if $m$ is odd.  In both cases, $g$ is regular in $\SL_n(q)$, and $g^G$ splits into at most two $M$-classes, so $\fp(g,G/M)\le 2 |C_G(g)|/(m-\delta)$, which is easily seen to be $<|G|^{1/3}/((2,q-1)\cdot m!)^{1/3}\le |G:M|^{1/3}$. For example, for $n=8$ and $G=\Sp_8(2)$ and $A_{10}\le M$, $g$ acts on $\F_2^8$ with  irreducible submodules of dimensions $2$ and $6$, so $|C_G(g)|=(2^3+1)(2+1)=27$ and 
		\[
		\fp(g,G/M)\le \frac{2|C_G(g)|}{m-\delta} = 6 < \left(\frac{|\Sp_8(2)|}{10!}\right)^{1/3}.
		\]
		This concludes the proof.
	\end{proof}

	\begin{lemma}
		\label{l:class_S_alternating}
		\Cref{t:main_simple} holds if $M$ is in class $\ca S$ with alternating socle $S=A_m$.
	\end{lemma}
	
	\begin{proof}
		We assume $V$ is not the fully deleted permutation module in view of \Cref{l:fully_deleted}. If $m\le 10$ then the modular character tables of $L$ are available in \cite{GAP_character_table}, and the result follows as discussed before \Cref{l:fully_deleted}. Assume then $m\ge 11$, and let  $\ell$ and $s$ be distinct primes as in \Cref{l:generation_a_n_second}. Without loss of generality, assume $\ell \neq p$. As in \Cref{l:generation_a_n_second}, let $x$ and $y$ be two elements of order $\ell$ of $L$ such that $L=\gen{x,y}$. Since $L$ is irreducible, there exists $g\in \{x,y\}$ such that $a=\dim(C_V(g))\le n/2$. Note that either $\ell \ge 11$ and $g^M$ is the union of at most two $L$-classes, or $\ell=7$ and $g^M$ is an $L$-class. (In order to see this we may assume $L=A_m$ and $M\le S_n$.) In both cases we have $|g^L \cap \gen g|\ge 5$.
		By \Cref{l:prime_order_eigenvalues}, we deduce that $N_L(\gen g)/C_L(g)$ has all orbits of size $|g^L \cap \gen g| \ge 5$ on the nontrivial eigenspaces of $g$.  We now apply \Cref{l:general_parameters_SL,l:general_parameters_SU,l:general_parameters_Sp,l:general_parameters_SO} with $c=1/2$, $B=5$, $C=2n+5$ to deduce the following:
		
		\begin{itemize}
			\item If $G=\SL_n(q)$ and $n\ge 24$, then \eqref{eq:condition_second} is satisfied.
			\item If $G=\SU_n(q^{1/2})$ and $n\ge 47$, then \eqref{eq:condition_second_SU} is satisfied.
			\item If $G=\Sp_n(q)$ and $n\ge 60$, then \eqref{eq:condition_second_Sp} is satisfied.
			\item  If $G=\Omega^\varepsilon_n(q)$ and $n\ge 45$, then \eqref{eq:condition_second_SO} is satisfied.
		\end{itemize}
		
		If $n$ satisfies the above inequalities, then by \Cref{l:general_parameters_SL,l:general_parameters_SU,l:general_parameters_Sp,l:general_parameters_SO} we have $|C_G(g)|<|G:M|^{1/3}$, as desired. Assume then that $n$ is smaller than the above values. Then, by \cite{hiss2001malle}, we see that $m\le 14$. If $m\le 13$, then the modular character tables of  $L$ are available in \cite{GAP_character_table} and the result follows. If $m=14$, then the modular character tables are available unless $L=2.A_{14}$. By \cite{hiss2001malle}, we see that in this case the only option is $n=32$, $p=7$ and $G=\Sp_{32}(q)$. By \Cref{l:generation_a_n} we have that $L=A_{14}$ is generated by two conjugates of an element $g$ of order $11$. We deduce that either $a\le 12$ and each nontrivial eigenspace has dimension at most $2$, or $a\le 2$ and each nontrivial eigenspace has dimension at most $3$; it follows that  $d\le 184$, which implies $|C_G(g)|<|G:M|^{1/3}$. 
	\end{proof}

	Next we address the case where $S$ is not alternating. It is convenient to isolate the alternating and symmetric square for $S=\PSL_m(q)$.

	\begin{lemma}
		\label{l:alternating_square}
	\Cref{t:main_simple} holds if $S=\PSL_m(q)$ and $V$ is the alternating or symmetric square  of the natural module for $S$.
	\end{lemma}

\begin{proof}
	Let $g\in \SL_m(q)$ be of order $(q^m-1)/(q-1)$. We claim that $g$ has distinct eigenvalues on $V$. In order to check this, it is enough to assume that $V$ is the symmetric square. If $\lambda$ is an eigenvalue of $g$ on the natural module, then the eigenvalues on $V$ are $\lambda^{q^i+q^j}$ for $i=j$ or $\{i,j\}\in \binom{[n]}{2}$. It is straightforward to see that $|g|$ does not divide $q^i + q^j - q^k - q^\ell$ for $\{i,j\}\neq \{k,\ell\}$, which is equivalent to saying that the eigenvalues  are pairwise distinct.  Note that for $m=2$ we have $M\in \ca C_8$ in view of the isomorphism $\PSL_2(q)\cong \Omega_3(q)$. For $m\ge 3$, using $|C_G(g)|\le (q^n-1)/(q-1)$ we see that $|\GL_m(q)||C_G(g)|^3 < |\SL_n(q)|$ and the proof is concluded. 
\end{proof}

Next we address the general case.  If $S$ is of Lie type, we denote by $\F_r$ the field of definition. We first exclude the cases $S= \PSL_2(r), \PSL^\pm_3(r)$ or  $|M|\ge q^{2n+4}$, which will be handled in subsequent lemmas.
	
	\begin{lemma}
		\label{l:class_S}
		\Cref{t:main_simple} holds if $M$ is in class $\ca S$ with $|M|\le q^{2n+4}$, and $S$ is sporadic or of Lie type with $S\neq \PSL_2(r),\PSL^\pm_3(r)$.
	\end{lemma}
	
	\begin{proof}
		We may assume $S\neq \PSp_4(2)'\cong A_6, S\neq \PSL_4(2)\cong A_8$ in view of \Cref{l:class_S_alternating}. 
		As discussed before \Cref{l:fully_deleted}, we replace $G$ by its quasisimple cover acting faithfully on $V$; we have $L\trianglelefteq M\le G$ where $L$ is quasisimple with $L/Z(L) \cong S$. Since we have replaced $M$ by its preimage, the bound in the hypothesis now gives $|M| \le q^{2n+5}$.
		
		Next we partition the socles $S$ into two collections, $\ca H_1$ and  $\ca H_2$. See the paragraph after Collection $\ca H_2$ for comments on this choice.

		\begin{collection_1}
			The socles $S$ in $\ca H_1$ are those satisfying either (1) or (2) below.
			\begin{itemize}
				\item[(1)] $S\neq \POm^+_{2m}(r), \PSL^\pm_4(r), E_6^\pm(r)$, $E_7(r)$, and the Schur multiplier of $S$ is trivial.

				\item[(2)] Either $S= E^\pm_6(r)$, or $S=\POm^+_{2m}(r)$ with $m$ odd,		
				or the Schur multiplier of $S$ is not trivial and $S$ is one of the following:
				\begin{align*}
					&\PSL^\varepsilon_m(r) \text{ with $m\ge 5$ and $(\varepsilon,m)\neq (-,6)$}\\
					&\PSp_{2m}(r) \text{ with $m\ge 3$ and $(m,r)\neq (3,2)$}\\
					&\POm^-_{2m}(r) \\
					&\POm_{2m+1}(r) \text { with $m\ge 4$}
				\end{align*}
				In (2), we furthermore \textit{exclude} the following cases:
				\begin{itemize}
					\item[$\diamond$]
					$S=\PSU_m(r)$ with $m\equiv 2 \pmod 4$ and $p$ a ppd of $(r^2)^{m/2}-1$; $S=\PSU_m(r)$ with $m\equiv 3\pmod 4$ and $p$  a ppd of $(r^2)^{m-2}-1$; $S=\PSp_{2m}(r)$ with $p$ a ppd of $r^{2(m-1)}-1$; $S=\POm_{2m+1}(r)$ with $m$ even and $p$ a ppd of $r^m-1$.
				\end{itemize}
			\end{itemize}
		\end{collection_1}
		
		\begin{collection_2} The socles $S$ in $\ca H_2$ are those not belonging to $\ca H_1$. Specifically:
			\begin{align*}
				&\PSL_4^\pm(r)\\
				&\PSU_6(r) \text{ with $(6,r+1)\neq 1$}\\ 
				&\PSp_4(r) \text{ with $r$ odd}\\ &\POm_7(r)\\
				&\POm^+_{2m}(r) \text{ with $m$ even}\\
				&E_7(r)\\
				&\text{the groups in $\ca H_1$(2)($\diamond$) having nontrivial Schur multiplier} \\
				& G_2(3), G_2(4),\, ^2\!B_2(8),  F_4(2), \Sp_6(2)\\
				& M_{12}, M_{22}, J_2, J_3, HS, Suz, McL, Ru, O'N, Co_1, Fi_{22}, Fi'_{24}, B.
			\end{align*} 
		\end{collection_2}
		
		At this point we have defined $\ca H_1$ and $\ca H_2$. Next, for $S$ in $\ca H_1$ we will find an element $g\in L$ satisfying $a=\dim(C_V(g))\le n/3$, while for $S$ in $\ca H_2$ we will find an element $g\in L$ satisfying $a\le n/2$. In order to orient the reader, before giving the details let us briefly motivate the choice of $\ca H_1$ and $\ca H_2$, recalling also \Cref{subsec:strategy} for related remarks. For the groups in $\ca H_1(1)$ we will apply \cite{guralnick2012malle}. The exclusion of the groups in $\ca H_1(1)$ (e.g., $\PSL^\pm_4(r)$) is due to the fact that  in these cases the element $g$ given in \cite{guralnick2012malle} does not work well when applying \Cref{l:semiregular_eigenvalues}. For the groups in $\ca H_1(2)$ we will apply \Cref{l:invariable generation}. The exclusion of the groups in $\ca H_1(2)(\diamond)$ guarantees that if our chosen element $g$ is not semisimple in $\SL_n(q)$, then it is regular in $L$ (see \Cref{l:coprime_schur}(iii)), and so we can effectively apply \Cref{l:green_correspondence}. For the groups in $\ca H_2$ we will use either \cite{guralnick2012malle} or \Cref{l:generation_by_conjugates}. Now we address this in detail.

		Assume first $S$ is in $\ca H_1$. Suppose $S$ is in $\ca H_1$(1);
		then $L=S$. If $S$ is sporadic, let $g=x_1$ be as in \Cref{table:sporadics}. If $S$ is of Lie type and $S\neq \PSL^\varepsilon_m(r)$, let $g=x_1$ be as in \Cref{table:new_unique}. If $S=\PSL^\varepsilon_m(r)$, set $u=1$ if $\varepsilon =+$ and $u=2$ if $\varepsilon =-$, and let $g$ be an element of order $\Phi^*_{um}(r)$ if $m$ is odd, and of order $\Phi^*_{u(m-1)}(r)$ if $m$ is even.  Then, in all cases, \cite[Theorem 1.1, Props 3.4--3.13, Prop 4.5]{guralnick2012malle} shows that $L$ is generated by three conjugates of $g$ whose product is $1$. It follows from \Cref{l:scott} that $a\le n/3$, as desired.
		
		Assume next $S$ is in $\ca H_1$(2). Without changing notation, we let $x_1, x_2 \in L$ be lifts of the corresponding elements of $S$ in \Cref{table:new_unique}, having the same orders (see \Cref{l:coprime_schur}).  Then, by \Cref{l:invariable generation} we deduce that $L$ is invariably generated by $x_1$ and $x_2$. Moreover, in all cases $x_1$ and $x_2$ are regular semisimple (see \Cref{l:coprime_schur}(iii))
		and so by an extension of a theorem of Gow \cite{gow2000commutators} (see \cite[Lemma 5.1]{guralnick2015tiep_lifting}), we can write $x_2=x_1^yx_1^z$ for $y, z \in L$. Then $L=\gen{x_1^y, x_1^z, x_2^{-1}}$, and by \Cref{l:scott}, there exists $g\in \{x_1, x_2\}$ such that $a\le n/3$, as desired.
		
		Assume now $S$ is in $\ca H_2$. If $S=\PSU_m(r)$ with $m>4$, then let $g\in L$ be an element of order $\Phi^*_{2m}(r)$ if $m$ is odd, and of order $\Phi^*_{2(m-1)}(r)$ if $m$ is even. In the other cases, let $g=x_1\in L$ be a lift of the corresponding element in \Cref{table:new_unique} or \Cref{table:sporadics}, having the same order, unless $S=\PSU_4(q)$, in which case we let $g=x_2\in L$ be a lift of the corresponding element in \Cref{table:new_unique}. If $S\neq Co_1$ is sporadic, or $S\neq E_7(r)$ is exceptional, or $S=\POm_{2m+1}(r)$, or $S=\PSU_m(r)$ with $m>4$, or $S=\PSp_{2m}(r)$, then by \cite[Theorem 1.1, Propositions 3.4--3.13]{guralnick2012malle} we deduce that $L$ is generated by two conjugates of $g$. If $S=Co_1$, the same holds by \cite[Proposition 6.2]{guralnick2000kantor}.   
		If $S$ is $\PSL_4^\pm(r)$, or $\POm^+_{2m}(r)$ with $m$ even and $(m,r)\neq (4,2)$, or $E_7(r)$, the same holds from \Cref{l:generation_by_conjugates}. In conclusion, if $S\neq \POm^+_8(2)$ then $L$ is generated by two conjugates of $g$, and since $L$ acts irreducibly on $V$ we deduce  $a\le n/2$, as claimed. 
		
		In the case $S=\POm^+_8(2)$, the modular character tables are available on GAP Character Table Library \cite{GAP_character_table}, so we exclude this case from now on.
		
		At this point, for $S$ in $\ca H_1$ we have $a\le n/3$, and for $S$ in $\ca H_2$ we have $a\le n/2$. Next we divide into three cases.
		
		\textbf{Case 1}: $S$ is in $\ca H_1$ and $g$ is semisimple in $\SL_n(q)$. Then by \Cref{l:coprime_schur}(v) we have that $|N_L(\gen g)/C_L(g)|$ acts semiregularly on $\gen g \sm\{1\}$. In particular, by \Cref{l:semiregular_eigenvalues} we deduce that $|N_L(\gen g)/C_L(g)|$ acts semiregularly on the nontrivial eigenspaces of $g$. Note that if  $S=\PSU_6(r)$ then $S$ is in $\ca H_1(1)$ and $|g|=\Phi^*_{10}(r)$ (since $\PSU_6(r)$ does not belong to $\ca H_1(2)$). Then, by inspection of \Cref{table:new_unique,table:sporadics} we see that $|N_L(\gen g)/C_L(g)|\ge 4$. We then apply \Cref{l:general_parameters_SL,l:general_parameters_SU,l:general_parameters_Sp,l:general_parameters_SO} with $c=1/3$, $B=4$ and $C=2n+5$ to deduce the following:
		
		\begin{itemize}
			\item If $G=\SL_n(q)$ and $n\ge 11$, then \eqref{eq:condition_first} is satisfied.
			\item If $G=\SU_n(q^{1/2})$ and $n\ge 22$, then \eqref{eq:condition_first_SU} is satisfied.
			\item If $G=\Sp_n(q)$ and $n\ge 18$, then \eqref{eq:condition_first_Sp} is satisfied.
			\item  If $G=\Omega^\varepsilon_n(q)$ and $n\ge 25$, then \eqref{eq:condition_first_SO} is satisfied.
		\end{itemize}
		If $n$ satisfies the above inequalities, then by \Cref{l:general_parameters_SL,l:general_parameters_SU,l:general_parameters_Sp,l:general_parameters_SO} we have $|C_G(g)|<|G:M|^{1/3}$, as desired. Assume then that $n$ is smaller than the above values.
		If $S$ is in $\text{Lie}(p')$  or $S$ is sporadic, then by \cite{hiss2001malle} we see that all the possibilities for $S$ are included in \cite{GAP_character_table} and the result follows.
		
		Assume then $S$ is in $\text{Lie}(p)$. Suppose $V$ is $p$-restricted for $L$. Then the possibilities are listed in \cite{lubeck2001small}. Noting that for $n\ge 22$ 
		we have $G=\Omega^\varepsilon_n(q)$, so $V \downarrow L$ is self-dual, 
		we have the following possibilities:
		\begin{align*}
			&S=\PSL^\pm_5(r) \text{ with } n=10,15,23,24\\
			&S=\PSL^\pm_6(r) \text{ with } n=15,20,21\\
			&S=\PSL^\pm_7(r) \text{ with } n=21\\
			&S=\PSp_4(r) \text{ with } n=16 \text{ and $r$ even}\\
			&S=\PSp_6(r) \text{ with } n=8,13,14,21\\
			&S=\PSp_8(r) \text{ with } n=16 \text{ and $r$ even}\\
			&S=\POm_9(r) \text{ with } n=16\\
			&S=\,^2B_2(r) \text{ with } n=4,16\\
			&S=G_2(r) \text{ with } n=6,7,14\\
			& S=\,^2\!G_2(r) \text{ with } n=7 \\
			&S=\,^3D_4(r) \text{ with } n=8\\
			&S=\,\POm^\pm_{10}(r) \text{ with } n=16 \\
		\end{align*}
		We consider each case in turn. Assume first $S=\PSL^\varepsilon_5(r)$ with $\varepsilon \in \{+,-\}$, so $q\ge r$. For $n=10$ or $15$, up to quasi-equivalence $V$ is the alternating square or the symmetric square of the natural 5-dimensional module, so by \Cref{l:alternating_square} we may assume $\varepsilon=-$. For $n=10$, we choose $g$ of order $(r^5+1)/(r+1)$, and we calculate directly that $g$ has distinct eigenvalues on $V$. Assume now $n=15$, so by what just proved $g$ is regular on a $10$-dimensional subspace. Since $|N_L(\gen g)/C_L(g)|=5$, it follows that there are at most $5$ eigenspaces of dimension $2$, so 
		$d\le 25$, which implies $|C_G(g)|<|G:M|^{1/3}$. (Here we used \Cref{l:coprime_schur}(v) and \Cref{l:semiregular_eigenvalues}. In the course of the proof,  we will often apply these results  with no explicit mention.) Assume now $n=23,24$, so we have $|M|<q^{n+3}$. Choosing the original element $g$ and replacing $C=2n+5$ by $C=n+3$ in the above calculation we see that the result follows. (In many cases, below, we will abuse notation and write that $V$ is a certain module, when we really mean that $V$ is quasi-equivalent to that module. This will happen for example in the next paragraph.) 
		
		Assume next $S=\PSL^\varepsilon_6(r)$, so $q\ge r$. In the cases $n=15,21$, by \Cref{l:alternating_square} we may assume $\varepsilon=-$. We choose $g$ of order $(r^5+1)/(r+1)$. If $n=15$, it follows readily by what proved in the previous paragraph for $\PSU_5(r)$ (by restricting to a $10$-dimensional submodule) that $g$ has distinct eigenvalues. If $n=21$, then it similarly follows that $g$ is regular on an $11$-dimensional subspace, so $g$ has at most $5$ eigenspaces of dimension $2$ and $d\le 31$, which easily gives $|C_G(g)| <|G:M|^{1/3}$. For $n=20$, $V$ is the third alternating power. We have $a\le \binom{6}{2}/3 = 5$; since $|N_L(\gen g)/C_L(g)|\ge 5$, each nontrivial eigenspace has dimension at most $20/5=4$ and this is sufficient.
		
		Assume now $S=\PSL^\varepsilon_7(r)$ with $n=21$. Since $|N_L(\gen g)/C_L(g)|=7$, each eigenspace for $g$ has dimension at most $3$ and $d\le 63$ and this gives the result.

		Assume now $S=\PSp_4(r)$ with $n=16$. Then $a\le 4$ and so each eigenspace has dimension at most $4$, so $d=\dim(C_{\GL_n(K)}(g))\le 64$ and this is sufficient.

		In the case $S=\POm_9(r)$ or $\PSp_8(r)$ with $n=16$, or $S=\PSp_6(r)$ with $n=8$, then $V$ is the spin module, which is handled in \Cref{l:liebeck}, below. 
		
		Assume now $S=\PSp_6(r)$ with $n>8$. Note that $|N_L(\gen g)/C_L(g)|=6$. In particular, since $a=\dim(C_V(g))\le n/3$, for $n=13$ we deduce by \Cref{l:semiregular_eigenvalues} that $a=1$ and the nontrivial eigenspaces have dimension at most $2$, which gives $|C_G(g)|<|G:M|^{1/3}$. By the same argument, for $n=14$ (resp. $n=21$) we see that all 
		eigenspaces of $g$ have dimension at most $2$ (resp. at most $3$), and the result follows.
		
		Assume now $S=\,^2B_2(r)$. For $n=4$ we have $r=q>2$ and $G=\Sp_4(q)$ and $S=M$. An element $g$ of $M$ of order $q+\sqrt{2q}+1$ acts irreducibly on $V$, and moreover $g^G\cap M=g^M$, so $\fp(g,G/M) = |C_G(g)|/|C_M(g)| = q-\sqrt{2q}+1 < |G:M|^{1/3}$. For $n=16$, we have that $V$, as an $\F_q L$-module, is equivalent to a module that is not $p$-restricted (see \cite[Theorem, p. 207]{humphreys2006modular}), and so does not give rise to a maximal subgroup, that is, $M\not\in \ca A$.
		
		Next let $S=\,^2G_2(r)$. Then $n=7$, but in this case $S<G_2(r)<G$ and $M \not \in \ca A(G)$. 
		
		Assume now $S=G_2(r)$. If $n=6$ then $q=r$ is even, $G=\Sp_6(q)$ and $S=M$. An element $g \in M$ of order $(q^3+1)/(q+1)$ acts irreducibly on $V$, and $g^G\cap M=g^M$, so $\fp(g,G/M)=|C_G(g)|/|C_M(g)| = q+1 < |G:M|^{1/3}$. The case $n=7$ is entirely analogous; in this case $G=\Omega_7(q)$ and $g$ acts irreducibly on a hyperplane. Assume finally $n=14$; since $|N_G(\gen g)|/|C_G(g)|=6$ and $a\le n/3$, we deduce by \Cref{l:semiregular_eigenvalues} that all eigenspaces of $g$ have dimension at most $2$, which gives the result.
		
		Assume now $S=\,^3D_4(r)$ with $n=8$, so $q=r^3$ and $G=\Omega_8^+(q)$. Then $g$ is regular, with $|C_G(g)| \le (q^2+1)^2$ and  $|C_G(g)|<|G:M|^{1/3}$.

		Assume finally $S=\POm^\varepsilon_{10}(r)$ with $n=16$, so $V$ is the spin module. The case $\varepsilon =+$ is covered in 
	\Cref{l:liebeck}, below, and the case $\varepsilon =-$ can be handled in the same way, by taking $g$ of order $q^4+1$.
		This completes the argument when $V \downarrow L$ is $p$-restricted.

		If $V$ is not $p$-restricted, then by Steinberg's twisted tensor product theorem we have $n\ge h^2$ where $h$ is the smallest degree of a faithful representation of $L$. Since $h\ge 4$ and $n\le 24$, it must be $h=4$ and $n=16$. By \cite{schaffer1999twisted}, there is no case with $M\in \ca A(G)$ (note that $S$ is in $\ca H_1$, so if $S=\PSp_4(r)$ then $r$ is even, and also  $S\neq \PSL_4^\pm(r)$).

		\textbf{Case 2}: $S$ is in $\ca H_2$ and $g$ is semisimple in $\SL_n(q)$. By our choice of the element $g$, we have $|N_L(\gen g)/C_L(g)|\ge 4$ (recall that we chose $g=x_2$ for $S=\PSU_4(r)$). In particular, the overall argument of Case 1 goes through to give the following, except that we apply \Cref{l:general_parameters_SL,l:general_parameters_SU,l:general_parameters_Sp,l:general_parameters_SO} with $c=1/2$ (instead of $c=1/3$), $B=4$ and $C=2n+5$.
		\begin{itemize}
			\item If $G=\SL_n(q)$ and $n\ge 36$, then \eqref{eq:condition_second} is satisfied.
			\item If $G=\SU_n(q^{1/2})$ and $n\ge 71$, then \eqref{eq:condition_second_SU} is satisfied.
			\item If $G=\Sp_n(q)$ and $n\ge 77$, then \eqref{eq:condition_second_Sp} is satisfied.
			\item  If $G=\Omega^\varepsilon_n(q)$ and $n\ge 66$, then \eqref{eq:condition_second_SO} is satisfied.
		\end{itemize}
		If $n$ satisfies the above inequalities, then by \Cref{l:general_parameters_SL,l:general_parameters_SU,l:general_parameters_Sp,l:general_parameters_SO} we have $|C_G(g)|<|G:M|^{1/3}$, as desired. Assume then $n$ is smaller than above values. Suppose $S$ is in $\text{Lie}(p')$ or $S$ is sporadic. By \cite{hiss2001malle}, either the modular character tables of $S$ are available in \cite{GAP_character_table}, in which case the result follows, or we are in one of the following cases:
		\begin{align*}
			&S=\PSU_4(4) \text{ with } n=51,52\\
			&S=\PSp_4(7) \text{ with } n=24,25\\
			&S=F_4(2) \text{ with } n=52\\
			&S=Co_1 \text{ with } n=24.
		\end{align*}
		
		Assume first $S=\PSU_4(4)$. Then $|M|<3^{25}< q^{n/2}$, so we may replace $C=2n+5$ by $C=n/2$ in the above calculation and this gives the conclusion.
		
		Assume now $S=\PSp_4(7)$ with $n=24,25$; by \cite{hiss2001malle} we have $p\neq 7$ and $G=\SL^\pm_{24}(q)$. Let $g\in S$ be of order $7$ and with centralizer of order $4116$; there is only one such $S$-class and so $|g^S\cap \gen g| = 6$. We check with GAP that $S$ is generated by two conjugates of $g$. If $n=24$, then $a=\dim(C_V(x))\le n/2 =12$ and by \Cref{l:semiregular_eigenvalues,l:optimization_easier} we have that $d=\dim(C_{\GL_n(K)}(g))\le 168$, from which we see that $|C_G(x)|<|G:M|^{1/3}$. If $n=25$, then in fact we must have $a\le 7$, so $d\le 103$ and this is sufficient.
		
		Assume now $S=F_4(2)$ with $n=52$; we have $G=\Omega^\varepsilon_n(q)$. We have $|N_L(\gen g)|/|C_L(g)|=12$ and $a\le n/2=26$, so in fact $a\le 16$ and by \Cref{l:optimization_easier} we see that $d\le 364$, which gives $|C_G(g)|^3<|G:M|$.

		Assume finally $S=Co_1$ with $n=24$. Since $|N_L(\gen g)/C_L(g)|=11$ and $a\le n/2=12$, it follows that all eigenspaces of $g$ have dimension at most $2$, and this gives $|C_G(g)|^3<|G:M|$.

		Suppose now $S$ is in $\text{Lie}(p)$. The cases $S=G_2(3), G_2(4), ^2\!B_2(8), \Sp_6(2)$ appear in \cite{GAP_character_table} and the conclusion follows. Assume we are not in these cases. Suppose $V$ is $p$-restricted as $S$-module. Then the possibilities are listed in \cite{lubeck2001small}; these are as follows:
		\begin{align*}
			&S=\PSL^\pm_4(r) \text{ with } n\le 76\\ 
			&S=\PSU_6(r) \text{ with } n=15,20,21,34,35,50,56,70\\
			&S=\PSp_4(r)  \text{ with } n\le 76\\
			&S=\POm_7(r)  \text{ with } n=8,21,26,27,35,40,48,63,64\\
			&S=\POm^+_8(r)  \text{ with } n=26,28,35,48,56\\
			&S=\POm^+_{12}(r)  \text{ with } n=32,64,66,76\\
			&S=F_4(2)  \text{ with } n=26. \\
		\end{align*}
		(Note that the case $S=E_7(r)$ with $n=56$ is excluded in view of the assumption $|M|\le q^{2n+4}$.)
		Assume first $S=\PSU_6(r)$, so $q\ge r$; we have $|N_L(\gen g)|/|C_L(g)|=5$. If $n\ge 34$ then $|M|<q^{n+5}$; in the above calculation we may replace $C=2n+5$ by $n+5$ and $B=4$ by $5$ and we see that $|C_G(g)|<|G:M|^{1/3}$. The cases $n=15,20,21$ were addressed already in Case 1.

		Assume now $S=\POm_7(r)$, so $q=r$ and $|N_L(\gen g)|/|C_L(g)|=6$. If $n\ge 26$ then $|M|<q^{n-2}$; the usual calculation with $c=1/2$, $B=6$ and $C=n-2$ gives the result. If $n=8$ then $G=\Omega_8^+(q)$ and $S$ is conjugate in $\Aut(G/Z(G))$ to a reducible subgroup, so we can exclude this case. If $n=21$ then $V$ is the alternating square; we have $a=\dim(C_V(g))=3$, and so every eigenspace has dimension at most $3$, which implies $|C_G(g)|<|G:M|^{1/3}$.
		
		Assume now $S=\POm^+_{12}(r)$, so $q=r$. We  have $|N_L(\gen g)|/|C_L(g)|=6$; if $n\ge 64$ we have $|M| <q^{n+4}$ and the result follows with the usual calculation. If $n=32$ then $V$ is the spin module. The restriction of $V$ to $\text{Spin}^+_8(q)$ is the sum of four half-spins (two of each type). Letting $g$ be an element of order $q^4-1$, it follows that each eigenspace of $g$ on $V$ has dimension at most $4$, which is sufficient.
		
		Assume now $S=\POm^+_8(r)$, so $q=r$. We have $|N_L(\gen g)|/|C_L(g)|=6$. If $n\ge 48$ then $|M|<q^n$; the usual calculation with $B=6$ and $C=n$ gives the result. The remaining cases are $n=26,28,35$. If $n=26$ then $p=2$ and $V$ is a composition factor of the alternating square. We have $a=2$, so each nontrivial eigenspace has dimension at most $4$ and this is sufficient. If $n=28$, then $V$ is the alternating square, so $a=4$ and each nontrivial eigenspace has dimension at most $4$, and the result follows. If $n=35$ then $V$ is a composition factor of the symmetric square. We have $a= 5$, so all eigenspaces have dimension at most $5$, which gives the result. 
		
		Assume now $S=E_7(r)$ and $n=56$. In this case $|N_L(\gen g)|/|C_L(g)|=14$. Replacing $B=4$ by $B=14$ in the usual calculation gives the result.
		
		If $S=F_4(2)$, then $|N_L(\gen g)|/|C_L(g)|=12$, so in fact all eigenspaces have dimension at most $2$ and the result follows.
		
		Assume then $S=\PSL^\pm_4(r)$, so $q\ge r$. If $n\ge 32$, note first that by \cite{lubeck2001small}, either $n\ge 44$ or $G=\SL_n(q)$ or $\SU_n(q^{1/2})$. In any case we have $|M|<q^{n/2+2}$,  
		and the conclusion follows by the usual calculation with $B=4$ and $C=n/2+2$. The remaining cases are $n=6,10,14,15,16,19,20$. We address the case $S=\PSL_4^-(r)=\PSU_4(r)$; the case $S=\PSL_4(r)$ is nearly identical. If $n=6$ then $M$ is in class $\ca C_8$ in view of $\PSU_4(r)\cong \POm^-_6(r)$. If $n=10$ then $V$ is the symmetric square and $G=\SU_{10}(q^{1/2})$. Then $a=\dim(C_V(g))=2$, so every nontrivial eigenspace has dimension at most $2$, so $d=\dim (C_{\GL_n(K)}(g))\le 20$ and this implies $|C_G(g)|<|G:M|^{1/3}$. If $n=14$ then $p=2$ and $V$ is a composition factor of the adjoint module. We have $a=2$, so every nontrivial eigenspace has dimension at most $12/4=3$ and this gives the result. If $n=15$ then $p\neq 2$ and $V$ is the adjoint module; we have $a=3$ and every nontrivial eigenspace has dimension at most $3$ and we are done. If $n=16$ then $p=3$ and $V\cong S^3(W)/W$ where $W$ is the natural module. One calculates that $a\le 4$, so $d\le 64$ (obtained when  $a=0$) and the result follows. If $n=19$ then $a\le 9$; it follows that in fact $a\le 7$ and all other eigenspaces have dimension at most $3$, so $d\le 85$ and this is sufficient. If $n=20$ then it must be $a\le 8$ and $d\le 100$, so $|C_G(g)|<|G:M|^{1/3}$ also in this case.

		Assume now $S=\PSp_4(r)$, so $q=r$. If $n\ge 35$ then $|M| < q^{n/3+1}$; we conclude with the usual calculation with $B=4$ and $C=n/3+2$. The remaining cases are $n=10,12,13,14,16,20,24,25,30$. For $i=4,5$, let $W_i$ be the natural $i$ dimensional module for $\PSp_4(r)$ (recall $\PSp_4(r)\cong \Omega_5(r)$).	If $n=10$ then $V=S^2(W_4)$. We have $a=\dim(C_V(g))=2$ and so each eigenspace has dimension at most $2$, which gives $|C_G(g)|<|G:M|^{1/3}$. If $n=12$ then $p=5$ and the highest dominant weight is $(1,1)$; it follows that $V$ is a composition factor of $U:=W_4\otimes W_5$. Let $g\in L$ be an element of order $13$; then $g$ has eigenvalues $x,x^5,x^{-1},x^{-5}$ on $W_4$ and eigenvalues $x^3,x^2,x^{-3},x^{-2}$ on $W_5$.
		It follows from an easy calculation that $C_U(g)=0$ and each nontrivial eigenspace on $U$ has dimension at most $2$. 
		This implies that $d\le 24$, which gives $|C_G(g)|<|G:M|^{1/3}$. If $n=13$ (resp. $n=14$) then $p=5$ (resp. $p\neq 5$) and $V$ is a composition factor of $S^2(W_5)$. We have $a=1$ (resp. $a=2$) and each nontrivial eigenspace has dimension at most $12/4=3$, which is sufficient. If $n=16$ then the highest dominant weight is $(1,1)$; as in the case $n=12$ we have that $V$ is a composition factor of $U:=W_4\otimes W_5$. We easily calculate that $C_U(g)=0$, and in particular each nontrivial eigenspace on $U$ has dimension at most $4$; it follows that $d\le 64$, which gives the result. If $n=20$ then as in the $\PSL^\pm_4(r)$-case we see that $d\le 100$ and the result follows. If $n=25$ then
		$a\le 12$, so in fact $a\le 9$ and so $d\le 145$ and we conclude. If $n=24$ then $p=7$ and the highest dominant weight is $(1,2)$ (where the second weight corresponds to the short root); in particular $W$ is a composition factor of $U:=W_4\otimes A$ where $A$ is a $14$-dimensional composition factor of $S^2(W_5)$. Let $g$ be an element of order $25$. We calculate that $\dim(C_A(g))=2$ and each nontrivial eigenspace has dimension $1$; it follows immediately that each eigenspace on $U$ has dimension at most $4$, from which $d\le 96$ and  $|C_G(g)|<|G:M|^{1/3}$.
		
		Assume finally that $V$ is not $p$-restricted as $S$-module.  By \cite{schaffer1999twisted}, the possibilites for the embedding $S<G/Z(G)$ are as follows:
		\[
		\begin{array}{l}
			\PSL_m(q^2) < \PSL_{m^2}(q),\,m=4,5 \\
			\PSL_m(q) < \PSU_{m^2}(q^{1/2}),\,4\le m\le 8 \\
			\PSU_4(q^{3/2}) < \PSU_{64}(q^{1/2}) \\
			\PSp_4(q^3) < \PSp_{64}(q),\,q \hbox{ odd} \\
			\PSp_{2m}(q^2) < \POm^\varepsilon_{4m^2}(q),\,m=2,3,4,\, \varepsilon = (-1)^m  \\
			\POm^\varepsilon_{2m}(q^2) < \POm^\delta_{4m^2}(q),\,m=3,4 \\
		\end{array}
		\]
		If $S\neq \PSL^\pm_4(r), \PSp_4(r)$ then $n\ge 25$ and $V=W\otimes W^{(q)}$ where $W$ is the natural module for $S$. Since $g$ has distinct eigenvalues on $W$, it follows that each eigenspace of $g$ on $V$ has dimension at most $\dim(W) = n^{1/2}$, which is sufficient. Assume now $S=\PSL^\pm_4(r)$.
		If $n=16$, we see that $C_V(g)=0$ and we see that each eigenspace has dimension at most $3$ (in order to prove this, it is enough to check that at least one eigenspace has dimension at most $3$, since the eigenspaces are permuted in orbits of size $4$). This gives the result. If $n=64$ then we see similarly that each eigenspace of $g$ on $W\otimes W^{(q)}$ has dimension at most $3$, from which each eigenspace on $V=W\otimes W^{(q)}\otimes W^{(q^2)}$ has dimension at most $12$ and this is enough. The case $S=\PSp_4(r)$ can be handled by the same argument.
		
		\textbf{Case 3}: $g$ is not semisimple in $\SL_n(q)$. We want to apply  \Cref{l:green_correspondence}. Note that in all cases if $S$ is of Lie type then $g$ is regular in $S$. (The exclusion of the groups in $\ca H_1$(2)($\diamond$) is crucial for this. These groups are included in $\ca H_2$, in which case the element $g=x_1$ from \Cref{table:new_unique} is chosen, and this element is regular by \Cref{l:coprime_schur}(iii).)
		 Then, by \Cref{l:coprime_schur}(vii) we have that $\gen{g}$ contains a Sylow $p$-subgroup $P$ of $L$ with $|P|\ge 5$. Moreover, for every $1\neq P_0\le P$ we have $N_L(P_0)=N_L(P)=N_L(\gen g)$ by \Cref{l:coprime_schur}(v). Let now $K=C_L(g)$. \Cref{l:coprime_schur}(iv) implies that $K$ is abelian and $K\trianglelefteq N_L(P)$; so assumption ($\star$) in \Cref{l:green_correspondence} is satisfied. Therefore, by \Cref{l:green_correspondence} we have  $V\downarrow P = V_0\oplus U$ where all Jordan blocks on $U$ have size $|P|$, and $V_0$ is the sum of at most $|N_L(P):K|$ Jordan blocks of the same size, say $t$. Note that $|N_L(P):K|=|N_L(\gen g):C_L(g)|$ appears in \Cref{table:new_unique} or \Cref{table:sporadics} under $n_i$. Now, for convenience, we replace $\gen g$ by a power of it generating $P$, so now $g$ is unipotent. (This makes it more convenient to compute centralizers.) As always, we have $d=\dim(C_{\GL_n(K)}(g))$ and $a=\dim(C_V(g))$. 
		
		Assume first that $t\ge 4$, so all Jordan blocks of $g$ have size at least $\min\{|P|,t\}\ge 4$. Then by \Cref{l:unipotent_centralizer} with $C=0$ and $B=4$ we deduce $d \le n^2/4$. Moreover, denoting by $R$ the number of Jordan blocks of $g$, we have $R \le n/4$. Next, we can upper bound $|C_G(g)|$ as in \Cref{l:compare_dim_centralizer_GL_Sp,l:order_centralizers}.
		\begin{itemize}
			\item If $G=\SL_n(q)$ and $n\ge 11$, we have $|C_G(g)|<q^d$ and so
			\[
			|M| |C_G(g)|^3 < q^{2n+5 + 3n^2/4} \le q^{n^2-2} <|G|.
			\]
			\item If $G=\SU_n(q^{1/2})$ and $n\ge 22$, then setting $q_0=q^{1/2}$, we have $|C_G(g)|< q_0^{n/4 + n^2/4}$ and so
			\[
			|M| |C_G(g)|^3 < q_0^{4n+10 + 3n/4 + 3n^2/4} \le q_0^{n^2-2} <|G|.
			\]
			\item If $G=\Sp_n(q)$ and $n\ge 24$, then letting $d'$ be the dimension of the centralizer of $g$ in $\Sp_n(K)$, we have $d'\le d/2+R/2$ and $R\le n/4$, so  $|C_G(g)|< q^{d'+R}\le q^{n^2/8 + 3n/8}$ and 
			\[
			|M| |C_G(g)|^3 < q^{2n+5 + 9n/8 + 3n^2/8} \le q^{n^2/2+n/2-1} <|G|.
			\]
			\item If $G=\Omega^\varepsilon_n(q)$ and $n\ge 28$, then letting $d'$ be the dimension of the centralizer of $g$ in $\SO_n(K)$, we have $d'\le d/2$ and $R \le n/4$, so $|C_G(g)|< q^{d'+R}\le q^{n^2/8 + n/4}$ and 
			\[
			|M| |C_G(g)|^3 < q^{2n+5 + 3n/4 + 3n^2/8} \le q^{n^2/2-n/2-2} <|G|.
			\]
		\end{itemize}
		By \cite{hiss2001malle}, we see that for $n\le 27$, either the modular character tables of $S$ and its covers are available in \cite{GAP_character_table}, or $S=\PSp_4(7)$ (with $n=24,25$) or $S=Co_1$ (with $n=24$). If $S=\PSp_4(7)$ then $G=\SL^\pm_n(q)$, in which case we already have the result for $n\ge 22$. If $S=Co_1$, we have $p=q=23$ and $M=2.Co_1$ and $G=\Omega^\varepsilon_n(q)$. In this case it is convenient to change our choice of the element $g$. The restriction of $V$ to $H=Co_2$ has composition factors of degree $23$ and $1$. Now the modular character table of $Co_2$ is available in \cite{GAP_character_table}; we see that an element $g$ of order $11$ has $10$ nontrivial eigenspaces of dimension $2$ and fixed space of dimension $4$. Therefore $d= 56$ and $d'=d/2-4/2 = 26$, where $d'$ is the dimension of the centralizer of $g$ in $\SO_n(K)$, which implies $|M||C_G(g)|^3 < |G|$.

		Assume then $t\le 3$, so $\dim(V_0)  \le t|N_L(A):S| \le 3n_i=:3f$ (here as usual $n_i=|N_L(\gen g)/C_L(g)|$). 
		Applying \Cref{l:unipotent_centralizer} with $C=3f$ and $B=5$, we deduce
		\[
		d \le \frac{n^2}{5} + \frac{36f^2}{5}.
		\]
		Assume first $f\le n/10$, so $d\le 0.272 n^2$. Moreover, letting $R$ denote the number of Jordan blocks of $g$, we have $R\le 3f+ (n-3f)/5\le 11n/25$.  Since $f\ge 4$, we deduce that $n\ge 10f \ge 40$. Using $|M|\le q^{2n+5}$ and \Cref{l:compare_dim_centralizer_GL_Sp,l:order_centralizers} we calculate  similarly to above that $|M||C_G(g)|^3<|G|$, unless $G=\Omega^\varepsilon_n(q)$ with $40\le n\le 43$.
		In these exceptional cases, since $f\le n/10$ we have $f=4$. We have $\dim(V_0)\le 12$ and all Jordan blocks on $U$ have odd size $|g|\ge 5$. If the Jordan block on $V_0$ have size at least $2$, it is easy to calculate that $d\le n^2/5 + 96$. Using $R\le 11n/25$, we see that $|M||C_G(g)|^3<|G|$ for $n\ge 39$. The remaining case is where $g$ is trivial on $V_0$. Then by \cite[Theorem 3.1]{liebeck_seitz_2012unipotent} we have $d'=d/2 - R/2$, where $d'$ is the dimension of the centralizer of $g$ in $\SO_n(K)$, so $d'+R = d/2 +R/2$. Using as above $R\le 22n/50$ and $d\le 0.272n^2$, we see that $|M||C_G(g)|^3 < |G|$ for $n\ge 36$, and so this case is done.
		
		Assume now $f>n/10$. Then by \Cref{l:landazuri_seitz}, $S$ is one of the following  (recall we are excluding $\SL_4(2)\cong A_8$ and $\Sp_4(2)'\cong A_6$ and we already handled $\POm^+_8(2)$):
		\begin{align*}
			&\PSL_m(r) \text{ with }(m,r)=(4,3), (5,2) \\
			&\PSU_m(r) \text{ with }(m,r)=(4,2),(4,3),(5,2),(6,2),(7,2)  \\
			&\PSp_{2m}(r) \text{ with }(m,r)=(2,3),(2,4),(2,5),(2,7),(3,2),(3,3),(4,2),(4,3) \\
			&\POm^-_8(2), \POm_7(3) \\
			& ^2B_2(8), G_2(3), \,^3D_4(2), F_4(2) \\
			&\text{a sporadic group not } ON, He,Th,Fi_{23}, Fi'_{24}, B, M.
		\end{align*}

		In all these cases we have $f\le 22$, and therefore $n<10f \le 220$, and moreover $p$ divides the order of $g$. 
		For all the groups except for $\PSU_7(2),\PSp_4(7), \PSp_8(3), F_4(2), Co_1, Ly, J_4$, the modular character table is available in \cite{GAP_character_table}, in which case the result follows.
		
		Assume now $S$ is one of $\PSU_7(2),\PSp_4(7), \PSp_8(3), F_4(2), Co_1, Ly, J_4$. The group $S=Co_1$ was handled above in the case $t\ge 4$ (and we did not use this assumption). In order to handle the other groups, we keep the notation $t$, $V_0$ and $U$ as above. It is sufficient to prove that $t\ge 4$, as this reduces to a case already considered; also we may assume $f>n/10$ and in particular $n<220$. If $S=\PSU_7(2)$ then by \cite{hiss2001malle} we have $n=42,43$. We have $|g|=\Phi^*_{14}(2)=43$, and so $t\ge 4$ (since otherwise $\dim(V_0)\le 3f\le 21$; but all Jordan blocks on a complement have size at least $43$). If $S=\PSp_4(7)$ then by \cite{hiss2001malle} we have $n=24,25$ or $n\ge 126$. Since $f=4$ we may assume $n=24,25$. But $|g|=\Phi^*_4(7) =25$, so as in the previous case we have $t\ge 4$ and we are done.  Assume now $S=\PSp_8(3)$. Then by \cite{hiss2001malle} we have $n=40,41$. Since $|g|=\Phi^*_8(3)=41$ we deduce $t\ge 4$ exactly as above. Assume $S=F_4(2)$, so $n=52$ and $|g|=13$. We have $G=\Omega^\varepsilon_n(q)$. If $t\le 3$ then $\dim(V_0)\le 3f =36$. But all Jordan blocks on a complement have size $13$, so in fact there are at least two such blocks, and so $\dim(V_0)\le 26$. But also $a=\dim (C_V(g))\le n/2=26$ and all Jordan blocks on $V_0$ have the same size, so $d\le 2\cdot 15^2 + 11\cdot 2^2=494$ (attained if $g$ has two Jordan blocks of size $13$ and the others of size $2$), which implies $|C_G(g)|<|G:M|^{1/3}$.
		If $S=Ly$ or $J_4$, then since $p$ divides $|g|$, by \cite{hiss2001malle} we see that there are no cases with $n<220$. The proof is finally concluded. 
	\end{proof}
	
	The remaining cases are $S=\PSL_2(r), \PSL^\pm_3(r)$ or $|M|\ge q^{2n+4}$. We address now the case $|M|\ge q^{2n+4}$; we use the main result of \cite{liebeck1985orders} asserting that the possibilities for $M$ are rather restricted.
	
	\begin{lemma}
		\label{l:liebeck}
		\Cref{t:main_simple} holds if $M$ is in class $\ca S$ and $|M|\ge q^{2n+4}$.
	\end{lemma}
	
	\begin{proof}
		Let $S$ be the socle of $M$. By \cite[Theorem 4.2]{liebeck1985orders},  we deduce that we are in one of the following cases:
		\begin{itemize}
			\item[(i)] $S=A_m$  and $V$ is the fully deleted permutation module with $n=m-1$ or $m-2$;
			\item[(ii)] $S=\PSL_m(q)$ and $V$ is the alternating square of the natural module;
			\item[(iii)] $S=\POm_7(q)$, $\POm_9(q)$ or $\POm^+_{10}(q)$ and $V$ is a spin module of dimension $8$, $16$ or $16$ (here we allow $q$ even for $\POm_7(q)$ and $\POm_9(q)$);
			\item[(iv)] $S=E_6(q)$ or $E_7(q)$ and $n=27$ or $56$;
			\item[(v)] $S=M_{24}$ or $Co_1$ and $n=11$ or $24$.
		\end{itemize}
		We consider each case in turn, noting  that	(i) has been handled in \Cref{l:fully_deleted} and (ii) has been handled in \Cref{l:alternating_square}. As in the previous proof, we replace $G$ by its cover acting faithfully on $V$.

		(iii) For $n=8$ we have $G=\Om^+_8(q)$, and $S$ is conjugate in $\text{Aut}(G/Z(G))$ to the stabilizer of a nonsingular vector, which has been handled already. For $n=16$ and $S=\POm_9(q)$ or $\POm^+_{10}(q)$, the restriction of $V$ to $\text{Spin}^+_8(q)$ is the sum of the two half-spin modules of dimension $8$, so by letting $g$ be an element of order $q^4-1$, we have that each eigenspace of $g$ has dimension at most $2$, and the result follows.
		
		(iv) For $S=E_6(q)$ or $E_7(q)$, we choose $g$ as in the proof of \Cref{l:class_S}. If $S=E_6(q)$ with $n=27$, then $G=\SL_{27}(q)$. We have $|M|<q^{3n}$ and,  arguing exactly as in the proof of \Cref{l:class_S}, we apply \Cref{l:general_parameters_SL} with $c=1/3$, $B=8$, and $C=3n$; we see that \eqref{eq:condition_second} holds and the conclusion follows. If $S=E_7(q)$ with $n=56$, then $|M|<q^{3n}$, $G$ is symplectic or orthogonal and we apply \Cref{l:general_parameters_Sp} or \Cref{l:general_parameters_SO} with $c=1/2$, $B=14$, and $C=3n$; we see that \eqref{eq:condition_second_Sp} and \eqref{eq:condition_second_SO} are satisfied and the conclusion follows.
		
		(v) If $S=M_{24}$ we can use \cite{GAP_character_table}. If $S=Co_1$ then $q=2$ and $G$ is orthogonal. We choose $g$ as in \Cref{table:sporadics}; we argue as in the proof of \Cref{l:class_S} and we
		apply \Cref{l:general_parameters_SO} with $c=1/2$, $B=11$, and $C=3n$ to see that  \eqref{eq:condition_second_SO} holds and the proof is concluded.
	\end{proof}
	
	The remaining cases are $S=\PSL_2(r)$ or $\PSL^\pm_3(r)$.

	\begin{lemma}
		\label{l:SL2}
		
		\Cref{t:main_simple} holds if $M$ is in class $\ca S$ with $S=\PSL_2(r)$ or $\PSL^\pm_3(r)$.
	\end{lemma}
	
	\begin{proof}
		As usual, we replace $G$ by its cover acting faithfully on $V=\F_q^n$, and $L\trianglelefteq M\le G$ where $L$ is quasisimple with $L/Z(L) \cong S$.
		
		Assume first $S=\PSL_2(r)$. We may assume $r\ge 7$ and $r\neq 9$ in view of $\PSL_2(4)\cong \PSL_2(5)\cong A_5$ and $\PSL_2(9)\cong A_6$. 
		Suppose $p\mid r$, that is, $S\in \text{Lie}(p)$. If $V$ is $p$-restricted as $L$-module, then $r=q$, $V$ is a symmetric power of the natural 2-dimensional module, and a regular unipotent element $g$ of $L$ is regular in $\SL_n(q)$. Moreover, $g^G$ splits into at most $(2,q-1)$-classes, and by \cite{bray2013holt_colva} we have $n\ge 4$; it follows that $\fp(g,G/M)\le (2,q-1)|C_G(g)|/q <|G:M|^{1/3}$. For example, for $n=4$ we have $G=\Sp_4(q)$, $p\ge 5$, and $|C_G(g)|=(2,q-1)q^2$, from which $\fp(g,G/M) \le (2,q-1)^2 q < |G:M|^{1/3}$. If $V$ is not $p$-restricted, then  by \cite{schaffer1999twisted} the possibilities are $\PSL_2(q^f)\cong \PSp_2(q^f) < \PSp_{2^f}(q)$ with $qf$ odd, and $\PSL_2(q^f)\cong \Omega_3(q^f) < \Om_{3^f}(q)$ with $qf$ odd. We have $V=W\otimes W^{(q)} \otimes \cdots \otimes W^{(q^{f-1})}$ where $\dim(W)=2$ or $\dim(W)=3$ in the respective cases. We let $g\in L$ be of order $\Phi^*_2(q^f)$. 
		Assume first $n=8$. We see that $g$ has an eigenvalue generating $\F_{q^6}$, so $g$ has six distinct nontrivial eigenvalues and the conclusion follows. In all the other cases, we see that $g$ has distinct eigenvalues on $W\otimes W^{(q)}$. In particular, each eigenspace on $V$ has dimension at most $\dim(W)^{f-2}$, which is sufficient.

		Assume now $p\nmid r$, that is, $S\in \text{Lie}(p')$, and denote by $\ell$ the prime divisor of $r$. If $r=7,11,13$ we can use \cite{GAP_character_table}; assume then $r\ge 16$. Denoting $e=(2,r-1)$, the possibilities for $n$ are $(r\pm 1)/e$, $r\pm 1$ and $r$. The case $n=r+1$ is excluded since this representation is imprimitive. Let $g$ be a generator of a split torus of $L$, and let $P$ be a Sylow $\ell$-subgroup of $L$. By Brauer's permutation lemma, $\gen g$ permutes the nontrivial irreducible characters of $P$ in orbits of size $(r-1)/e$. It follows that if $n=(r-1)/e$ then $g$ is regular; if $n=(r+1)/e$ then $g$ is regular on a codimension $2/e$-subspace and $\dim(C_V(g))=4-e$; if $n=r-1$ with $r$ odd then $V$ is the sum of two cyclic $K\gen g$-modules of dimension $(r-1)/2$; if $n=r$ then $V=U\oplus W$ where $W$ is a one-dimensional trivial $K\gen g$-module, and $U$ is a cyclic $K\gen g$-module (for $r$ even), or the sum of two cyclic $K\gen g$-modules of dimension $(r-1)/2$  (for $r$ odd).

		Assume now $S=\PSL^\pm_3(r)$ with $r\ge 3$. We choose $|g|=\Phi^*_3(r)$ for $\PSL_3(r)$ and  $|g|=\Phi^*_6(r)$ for $\PSU_3(r)$, so $N:=N_L(\gen g)/C_L(g)$ acts with all orbits of size $3$ on the set of nontrivial eigenspaces of $g$ on $V$. Note also that by \cite[Propositions 3.11 and 3.13]{guralnick2012malle} $S$ is generated by two conjugates of $g$, so $a=\dim(C_V(g))\le n/2$.

		Assume first $p\mid r$. Let $\alpha_1$ and $\alpha_2$ be simple roots for the algebraic group $A_2$ corresponding to $S$, with corresponding fundamental dominant weights $\lambda_1$ and $\lambda_2$. For integers $x$ and $y$, we will denote the character  $x\lambda_1 + y\lambda_2$ by $(x,y)$. Let $\lambda=(x,y)$ be the highest dominant weight of $V$, where $x\ge y$.
		
		Suppose $x=y$, that is, $V$ is self-dual. Then, for each $N$-orbit $\ca O$ of nontrivial eigenspaces, there is another orbit whose eigenspaces have the same dimension (and whose corresponding eigenvalues are inverses of those in $\ca O$). Since $a=\dim(C_V(g))\le n/2$, we may apply \Cref{l:general_parameters_Sp,l:general_parameters_SO} with $c=1/2$ and $B=6$ to deduce:
		\begin{itemize}
			\item If $G=\Sp_n(q)$ and $n\ge 16$, then \eqref{eq:condition_second_Sp} is satisfied.
			\item  If $G=\Omega^\varepsilon_n(q)$ and $n\ge 12$, then \eqref{eq:condition_second_SO} is satisfied.
		\end{itemize}
		
		If $n$ is at least these values, then by \Cref{l:general_parameters_Sp,l:general_parameters_SO} we have $|C_G(g)|<|G:M|^{1/3}$. Assume then $n$ is smaller than these values; by \cite{lubeck2001small}, the only remaining possibility is the adjoint module for $S$, with $n=7$ or $8$ according to whether $p=3$ or $p\neq 3$. For $n=7$ we have $M\not\in A$ by \cite{bray2013holt_colva}. For $n=8$ we have $a=2$, and as above every nontrivial eigenspace of $g$ has dimension $1$. Moreover $G=\Omega^\pm_8(q)$; by \Cref{l:compare_dim_centralizer_GL_Sp}, the dimension of the centralizer of $g$ in $\SO_8(K)$ is $4$ and we see that $|C_G(g)|<|G:M|^{1/3}$. 
		
		Suppose now $x>y$, so $V$ is not self-dual. We claim that $g$ has at least six distinct nontrivial eigenvalues on $V$.
		
		Assume $V$ is $p$-restricted as $L$-module. It follows from Premet's theorem \cite[Theorem 1]{premet1988weights} that every dominant weight $\gamma=(c,d)$ with $\gamma\in X:=\lambda - \mathbb Z \alpha_1 - \mathbb Z \alpha_2$ is a weight of $V$. Suppose $x\ge y+2$. Since $(x-y,0) =\lambda - (y,y) =\lambda - y\alpha_1-y\alpha_2\in X$, again by Premet's theorem, $V$ contains all the weights corresponding to the representation with highest weight $(x-y,0)$, which is the $(x-y)$-th symmetric power. Since $x-y\ge 2$, and since $g$ has at least six distinct nontrivial eigenvalues on $(2,0)$, we deduce that $g$ has at least six distinct nontrivial eigenvalues on $V$, as claimed.
		Now suppose $x=y+1$. Then $(2,1) = \lambda - (y-1,y-1)\in X$, so $V$ contains all the weights corresponding to the representation with highest weight $(2,1)$. This has dimension $15$ and is a constituent of $S^2(W)\otimes W^*$ where $W$ is the $3$-dimensional natural module. We can then calculate that $g$ has at least six distinct nontrivial eigenvalues on $V$.
		
		Assume now $V$ is not $p$-restricted. The possible embeddings are of type $\PSL_3(q)<\PSU_9(q^{1/2})$, or  $\PSL_3(q^f) < \PSL_{3^f}(q)$, or $\PSU_3(q^{f/2}) < \PSU_{3^f}(q^{1/2})$ with $f$ odd. In all cases, it is easy to check that $g$ has at least six distinct nontrivial eigenvalues. For example, in the first case we see that $g$ has two eigenvalues both of which generate $\F_{q^3}$ and that are not $\F_q$-conjugate (and so we get six distinct nontrivial eigenvalues by taking $\F_q$-conjugates). In the latter two cases, we see that $g$ has an eigenvalue generating $\F_{q^{3f}}$. 
		Then, in all cases $g$ has at least six distinct nontrivial eigenvalues, as claimed.
		
		Denote by $E$ the number of distinct eigenvalues of $g$.  Since $a\le n/2$ and since nontrivial eigenspaces occur in orbits of size $3$, it can be seen that $d + E$ is at most
		\[
		7 + \left(\frac n2\right)^2 + 3\left((\frac{n}{2}-3)/3\right)^2 + 3 = \frac{n^2}{3} - n + 13
		\]
		(attained when $g$ has trivial eigenspace of dimension $n/2$, three eigenspaces of dimension $1$ and three eigenspaces of dimension $(n/2-3)/3$), from which it also follows that $d\le n^2/3 - n +6$. 
		If $G=\SL_n(q)$ then we deduce  $|C_G(g)|^3< q^{n^2 - 3n +18}$, which is $< |G:M|$ if $n\ge 15$. If $G=\SU_n(q^{1/2})$ then, setting $q_0=q^{1/2}$, by \Cref{l:order_centralizers} we deduce  $|C_G(g)|^3< q_0^{3(d+E)}\le q_0^{n^2-3n+39}$, which is $<|G:M|$ if $n\ge 24$.
		
		Let us address the remaining cases (i.e. the case $n\le 14$ for $G = \SL_n(q)$ and the case $n\le 23$ for $G = \SU_n(q^{1/2})$). Assume first $V$ is $p$-restricted; the options are $n=6,10$ for $G=\SL_n(q)$; and $n=6,10,15,18,21$ for $G=\SU_n(q^{1/2})$.  If $n=6$ then $V=S^2(W)$ (where $W$ is the natural 3-dimensional module), $g$ is regular and we have $|C_G(g)|<|G:M|^{1/3}$. If $n=10$ then $V=S^3(W)$; we see that $a=\dim(C_V(g))=1$ and so $|C_G(g)|<|G:M|^{1/3}$. If $n=15$ then it must be $a\le 6$, which implies $d\le 51$ and the result follows. If $n=18$ then $p=5$ and $V$ is a constituent of $S^3(W)\otimes W^*$. Since by the above the largest eigenspace on $S^3(W)$ has dimension at most $2$, it follows that $a \le 6$, which gives the conclusion. Finally, if $n=21$ then $a\le 9$, so $d\le 111 = 3\cdot 6^2 + 3$ and this is sufficient. Assume now $V$ is not $p$-restricted with $n\le 23$; by \cite{schaffer1999twisted}, the possible embeddings are of type $\SL_3(q^2)<\SL_9(q)$ and $\SL_3(q)<\SU_9(q^{1/2})$. In the first case, $a=3$ and every nontrivial eigenspace has dimension $1$, so $d=15$ and the result follows. In the second case, there are three nontrivial eigenspaces of dimension $1$, and three nontrivial eigenspaces of dimension $2$, so again $d=15$ and $|C_G(g)|<|G:M|^{1/3}$ also in this case.

		Assume finally $p\nmid r$. By \cite{hiss2001malle}, if $n\le 40$ then we have $r\le 5$ and we can use \cite{GAP_character_table}. From now on, assume $n\ge 41$. We first claim that $|M|< q^n$. This is the case if $r\le 4$, so assume $r\ge 5$. By \cite{landazuri1974minimal}, if $S=\PSL_3(r)$ (resp. $\PSU_3(r)$) then $n\ge r^2-1$ (resp. $n\ge r(r^{2}-1)/(r+1)$), and we deduce that $|M| \le q|\Aut(S)| < q^n$, as claimed. So from now on we have $n\ge 41$ and $|M|<q^n$.

		Suppose $g$ is semisimple. Note that every prime divisor of $|g|$ is at least $7$ and also we have $|N_L(\gen g)/C_L(g)|=3$. By \cite[Theorem 1.2]{tiep2008hall}, we deduce that $g$ has at least $5$ distinct eigenvalues. Since these are permuted in orbits of size $3$, we deduce that $g$ has at least $6$ distinct eigenvalues. In particular, by the same calculation as in the case $p\mid r$ we see that $d\le n^2/3 - n +6$ and $d+E \le n^2/3-n + 13$, where $E$ denotes the number of distinct eigenvalues of $g$. Denote also by $E_2$ the number of distinct irreducible factors of degree at least two of the characteristic polynomial of $g$. Next we upper bound $|C_G(g)|$ using \Cref{l:compare_dim_centralizer_GL_Sp,l:order_centralizers}.
		
		\begin{itemize}
			\item If $G=\SL_n(q)$, we have $|C_G(g)|<q^d$ and so
			\[
			|M| |C_G(g)|^3 < q^{n + n^2 - 3n + 18} \le q^{n^2-2} <|G|.
			\]
			\item If $G=\SU_n(q^{1/2})$, then setting $q_0=q^{1/2}$  we have $|C_G(g)|< q_0^{d+E}$ and so
			\[
			|M| |C_G(g)|^3 < q_0^{2n + n^2 - 3n + 39} \le q_0^{n^2-2} <|G|.
			\]
			\item If $G=\Sp_n(q)$, then $d/2+E_2+a/2 \le n^2/6 - 2n +26 +n/4$. By \Cref{l:compare_dim_centralizer_GL_Sp,l:order_centralizers} we have $|C_G(g)|< q^{d/2+E_2+a/2}$ and so
			\[
			|M| |C_G(g)|^3 < q^{n + n^2/2 - 6n + 78 +3n/4} \le q^{n^2/2+n/2-1} <|G|.
			\]
			\item If $G=\Omega^\varepsilon_n(q)$, then  $d/2+E_2 \le n^2/6 - 2n +26$. By \Cref{l:compare_dim_centralizer_GL_Sp,l:order_centralizers} we have $|C_G(g)|< q^{d/2+E_2+2}$ and so
			\[
			|M| |C_G(g)|^3 < q^{n + n^2/2 - 6n + 78 +6} \le q^{n^2/2-n/2-2} <|G|.
			\]
		\end{itemize}

		Suppose finally $g$ is not semisimple; we apply \Cref{l:green_correspondence}. Letting $P$ be a Sylow $p$-subgroup of $\gen g$, we have that $V\downarrow P = V_0\oplus U$ where all Jordan blocks on $V_0$ have size $|P|\ge 7$, and $U$ is the sum of at most $3$ indecomposable module of dimension $t$. If $t\ge 4$ then applying \Cref{l:unipotent_centralizer} with $C=0$ and $B=4$ we deduce $d\le n^2/4$. In particular, the exact same computation as in \Cref{l:class_S} reduces to the case $n\le 29$, for which we are done (as we are assuming $n\ge 41$).
		
		Assume then $t\le 3$; so $\dim(V_0)\le 9$. By \Cref{l:unipotent_centralizer} with $C=9$ and $B=7$ we see that $d\le n^2/7 + 70$. The number $R$ of Jordan blocks of $g$ is at most $9+(n-9)/7 = n/7+54/7$. Then we can calculate similarly to above that, since $n\ge 27$, the conclusion holds. The proof is now complete.
	\end{proof}

	%%%%%%%%%%%%%%%

    \section{Proof of \Cref{cor:polynomials}}
    \label{sec:polynomials}
    
    For completeness, we give a proof of the deduction of \Cref{cor:polynomials} from \Cref{t:main}, which is an immediate application of Frobenius density theorem.

    \begin{proof}[Proof of \Cref{cor:polynomials}]
    Let $L$ be the Galois closure of $\Q(\alpha)$ in $\overline{\Q}$ and let $G=\Gal(L/\Q)$ be the Galois group. Then $G$ acts faithfully and transitively on the roots of $f$ and $M:=\Gal(L/\Q(\alpha))$ is a point stabilizer. The assumption that $\Q(\alpha)/\Q$ is minimal is equivalent to the fact that the action is primitive, and the assumption that $\Q(\alpha)/\Q$ is not Galois is equivalent to the fact that the action is not regular.  In particular, by \Cref{t:main} there exists $g\in G$ with $1\le \fp(g)\le n^{1/3}$. By Frobenius density theorem (see for example \cite[Theorem p. 11]{lenstra}), the (natural or analytic) density of primes $p$ such that $f$ has at least one root and at most $n^{1/3}$ roots in $\F_p$ exists, and is equal to the proportion of elements of $G$ with at least one fixed point and at most $n^{1/3}$ fixed points. This proportion is at least $1/|G|>0$ and the statement follows.        
    \end{proof}

	%	\bibliography{references}

\begin{thebibliography}{99}
		
		\bibitem{AJL} H.H. Andersen, J. Jørgensen and P. Landrock, The projective indecomposable modules of $SL(2,p^n)$, {\it Proc. London Math. Soc.} {\bf 46} (1983), 38--52. 
		
		\bibitem{aschbacher1984maximal_subgroups} M. Aschbacher, On the maximal subgroups of the finite classical groups, {\it Invent. Math.} {\bf 76} (1984), 469--514.
		
		\bibitem{AS} M. Aschbacher and L. Scott, Maximal subgroups of finite groups, {\it J. Algebra} {\bf 92} (1985), 44--80.

\bibitem{Babai} L. Babai, On the order of uniprimitive permutation groups, {\it Ann. of Math.} {\bf 113} (1981), 553--568.
        
		\bibitem{Baddeley}  R.W. Baddeley, Primitive permutation groups with a regular nonabelian normal subgroup, {\it Proc. London Math. Soc.} {\bf 67} (1993), 547--595.
		
		\bibitem{benson1998representations} D.J. Benson, {\it Representations and cohomology, I. Basic representation theory of finite groups and associative algebras}, Second edition, Cambridge Studies in Advanced Mathematics, 30. Cambridge University Press, Cambridge, 1998. 
		
		\bibitem{bray2013holt_colva} J. N. Bray, D. F. Holt, and C. M. Roney-Dougal, {\it The maximal subgroups of the low-dimensional finite classical groups}, London Math. Soc. Lecture Note Series, vol. 407, Cambridge University Press, Cambridge, 2013. 
		
		\bibitem{GAP_character_table} T. Breuer, The \textsf{GAP} Character Table Library, 
		Version 1.3.11, http://www.math.rwth-aachen.de/\~{}Thomas.Breuer/ctbllib, May 2025.
		
		\bibitem{breuer2008probabilistic} T. Breuer, R.M. Guralnick and W.M. Kantor, Probabilistic generation of finite simple groups, II, {\it J. Algebra} {\bf 320} (2008), 443--494.
		
		%\bibitem{bursurv} T.C. Burness, Simple groups, fixed point ratios and applications, in: {\it Local Representation Theory and Simple Groups}, EMS Ser. Lect. Math., European Mathematical Society, Z\"urich (2018), 267--322.
		
		\bibitem{burness2022guralnick} T.C. Burness and R.M. Guralnick, Fixed point ratios for finite primitive groups and applications, {\it Adv. Math.} {\bf 411} (2022), Paper No. 108778, 90 pp.
		
		\bibitem{CG} A.M. Cohen and R.L. Griess, On finite simple subgroups of the complex Lie group of type $E_8$, The Arcata Conference on Representations of Finite Groups (Arcata, Calif., 1986), 367--405, {\it Proc. Sympos. Pure Math.} {\bf 47}, Part 2, Amer. Math. Soc., Providence, RI, 1987.
		
		\bibitem{atlas} J.H. Conway, R.T. Curtis, S.P. Norton, R.A. Parker, R.A. Wilson, {\it Atlas of Finite Groups}, Oxford Univ. Press, 1985. 
		
		\bibitem{Cralt} D.A. Craven, Alternating subgroups of exceptional groups of Lie type, {\it Proc. Lond. Math. Soc.} {\bf 115} (2017), 449--501.
		
		\bibitem{craven2023maximal} D. A. Craven, The maximal subgroups of the exceptional groups $F_4(q)$, $E_6(q)$ and $^2\!E_6(q)$ and related almost simple groups, {\it Invent. Math.} {\bf 234} (2023), 637--719.
		
		\bibitem{Crmem1} D. A. Craven, Maximal $PSL_2$ subgroups of exceptional groups of Lie type, {\it Mem. Amer. Math. Soc.} {\bf 276} (2022), no. 1355, v+155pp.
		
		\bibitem{Crmem2} D.A. Craven, On medium-rank Lie primitive and maximal subgroups of exceptional groups of Lie type, {\it Mem. Amer. Math. Soc.} {\bf 288} (2023), no. 1434, v+213 pp.
		
		\bibitem{craven2022E7} D. A. Craven, On the maximal subgroups of $E_7(q)$ and related almost simple groups, preprint (arXiv:2201.07081), 2022.
		
		\bibitem{CST} D. A. Craven, D. I. Stewart, and A. R. Thomas, A new maximal subgroup of $E_8$ in characteristic 3, {\it Proc. Amer. Math. Soc.} {\bf 150} (2022), 1435--1448.
		
		\bibitem{Deriz} D.I. Deriziotis, The centralizers of semisimple elements of the Chevalley groups $E_7$ and $E_8$,
		{\it Tokyo J. Math.} {\bf 6} (1983), 191--216.
		
		\bibitem{DL} D.I. Deriziotis and M.W. Liebeck, Centralizers of semisimple elements in finite twisted groups of Lie type, 
		{\it J. London Math. Soc.} {\bf 31} (1985), 48--54.
		
		\bibitem{DLP} H. Dietrich, M. Lee and T. Popiel, The maximal subgroups of the Monster, {\it Adv. Math.} {\bf 469} (2025), Paper No. 110214, 33 pp.

        \bibitem{DM} J.D. Dixon and B. Mortimer, {\it Permutation groups}, Graduate Texts in Mathematics, 163, Springer-Verlag, New York, 1996.
		
		\bibitem{el_bach} M. El Bachraoui, Primes in the interval $[2n,3n]$, {\it Int. J. Contemp. Math. Sci.} {\bf 1} (2006), 617--621.
		
		%\bibitem{FKS} B. Fein, W.M. Kantor and M. Schacher, Relative Brauer groups, II, {\it J. Reine Angew. Math.} {\bf 328} (1981), 39--57.
		
		%\bibitem{gorenstein2007book} 
		
		\bibitem{gow2000commutators} R. Gow, Commutators in finite simple groups of Lie type, {\it Bull. Lond. Math. Soc.} {\bf 32} (2000), 311--315.
		
		\bibitem{guralnick2000kantor}  R.M. Guralnick and W.M. Kantor, Probabilistic generation of finite simple groups, {\it J. Algebra} {\bf 234} (2000), 743--792.

        \bibitem{GM} R.M. Guralnick and K. Magaard, On the minimal degree of a primitive permutation group, {\it J. Algebra} {\bf 207} (1998), 127--145.
		
		\bibitem{guralnick2012malle} R.M. Guralnick and G. Malle, Products of conjugacy classes and fixed point spaces, {\it J. Amer. Math. Soc.} {\bf 25} (2012), 77--121.
		
		\bibitem{guralnick2012beauville} R.M. Guralnick and G. Malle, Simple groups admit Beauville structures, {\it J. Lond. Math. Soc.} {\bf 85} (2012), 694--721.
		
		\bibitem{guralnick2015tiep_lifting} R. M. Guralnick and P. H. Tiep, Lifting in Frattini covers and a characterization of finite solvable groups, {\it J. Reine Angew. Math.} {\bf 708} (2015), 49--72.
		
		\bibitem{guralnick2020larsen_tiep_character} R.M. Guralnick, M. Larsen and P.H. Tiep, Character levels and character bounds, {\it Forum Math. Pi} {\bf 8} (2020), e2, 81 pp.
		
		\bibitem{guralnick1999praeger} R. M. Guralnick, T. Penttila, C. E. Praeger and J. Saxl, Linear groups with orders having certain large prime divisors, {\it Proc. London Math. Soc.} {\bf 78} (1999), 167--214.
		
		\bibitem{hiss2001malle} G. Hiss and G. Malle, Low-dimensional representations of quasisimple groups, {\it LMS J. Comput. Math.} {\bf 4} (2001), 22--63.
		
		\bibitem{humphreys2006modular}  J.E. Humphreys, {\it Modular representations of finite groups of Lie type}, London Math. Soc. Lecture Note Series, Vol. 326, Cambridge University Press, 2006. 
		
		\bibitem{IKMM} I.M. Isaacs, T.M. Keller, U. Meierfrankenfeld and A. Moreto, Fixed point spaces, primitive character degrees and conjugacy class sizes, {\it Proc. Amer. Math. Soc.} {\bf 134} (2006), 3123--3130.
		
		\bibitem{atlas_breuer_characters} C. Jansen, K. Lux, R. Parker, and R. Wilson, {\it An atlas of Brauer characters}, Oxford Univ.  Press, 1995.
		
		\bibitem{kleidman1987maximal_omega8}  P.B. Kleidman, The maximal subgroups of the finite 
		8-dimensional orthogonal groups $P\Omega_8^+(q)$ and of their automorphism groups, {\it J. Algebra} {\bf 110} (1987), 173--242.
		
		\bibitem{kleidman1988maximal_3D4} P. B. Kleidman, The maximal subgroups of the Steinberg triality groups $^3\!D_4(q)$ and of their automorphism groups, {\it J. Algebra} {\bf 115} (1988), 182--199.
		
		\bibitem{kleidman1988maximalG_2} P. B. Kleidman, The maximal subgroups of the Chevalley groups $G_2(q)$ with $q$ odd, of the Ree groups $^2\!G_2(q)$, and of their automorphism groups, {\it J. Algebra} {\bf 117} (1988), 30--71.
		
		\bibitem{kleidman1990liebeck} P.B. Kleidman and M.W. Liebeck, {\it The Subgroup Structure of the Finite Classical Groups}, London Math. Soc. Lecture Note Series, Vol. 129, Cambridge Univ. Press, 1990.
		
		\bibitem{KW22} P.B. Kleidman and R.A. Wilson, The maximal subgroups of $Fi_{22}$,
		{\it Math. Proc. Cambridge Philos. Soc.} {\bf 102} (1987), 17--23.
		
		\bibitem{KWJ4} P.B. Kleidman and R.A. Wilson, The maximal subgroups of $J_4$, {\it Proc. London Math. Soc.} {\bf 56} (1988), 
		484--510.
		
		\bibitem{KW23} P.B. Kleidman, R.A. Parker and R.A. Wilson, The maximal subgroups of the Fischer group $Fi_{23}$, 
		{\it J. London Math. Soc.} {\bf 39} (1989), 89--101.
		
		\bibitem{KostTiep}  A.I. Kostrikin and P.H. Tiep, {\it Orthogonal decompositions and integral lattices}, De Gruyter Expositions in Mathematics, 15. Walter de Gruyter and Co., Berlin, 1994.
		
		\bibitem{landazuri1974minimal} V. Landazuri and G. M. Seitz, On the minimal degrees of projective representations of the finite Chevalley groups, {\it J. Algebra} {\bf 32} (1974), 418--443.
		
		\bibitem{Lawunip} R. Lawther, Unipotent classes in maximal subgroups of exceptional algebraic groups, {\it J. Algebra} {\bf 322}  (2009), 270--293.
		
		\bibitem{Law} R. Lawther, Sublattices generated by root differences, {\it J. Algebra} {\bf 412} (2014), 255--263.
		
		\bibitem{liebeck1985orders} M.W. Liebeck, On the orders of maximal subgroups of the finite classical groups, {\it Proc. Lond. Math. Soc.} {\bf 50} (1985), 426--446.
		
		\bibitem{LPS} M.W. Liebeck, C.E. Praeger and J. Saxl, On the O'Nan-Scott theorem for finite primitive permutation groups, {\it J. Aust. Math. Soc. A} {\bf 44} (1988), 389--396.
		
		\bibitem{liebeck1998subgroup} M.W. Liebeck and G.M. Seitz, On the subgroup structure of classical groups, {\it Invent. Math.} {\bf 134} (1998), 427--453. 
		
		\bibitem{LS1} M. W. Liebeck and G. M. Seitz, On the subgroup structure of exceptional groups of Lie type, {\it Trans. Amer. Math. Soc.} {\bf 350} (1998), 3409--3482.
		
		\bibitem{LS2} M. W. Liebeck and G. M. Seitz, On finite subgroups of exceptional algebraic groups, {\it J. Reine Angew. Math.} {\bf 515} (1999), 25--72.
		
		\bibitem{LS3} M.W. Liebeck and G.M. Seitz, A survey of maximal subgroups of exceptional groups of Lie type, {\it Groups, combinatorics and geometry} (Durham, 2001), 139--146, World Sci. Publ., River Edge, NJ, 2003.
		
		\bibitem{LSroot} M. W. Liebeck and G. M. Seitz, Subgroups generated by root elements in groups of Lie type, {\it Ann. of Math.}  {\bf 139} (1994), 293--361.
		
		\bibitem{LSGeomDed} M. W. Liebeck and G. M. Seitz, Maximal subgroups of exceptional groups of Lie type, finite and algebraic, {\it Geom. Dedicata} {\bf 35} (1990), 353--387.
		
		\bibitem{LSmem96} M. W. Liebeck and G. M. Seitz, Reductive subgroups of exceptional algebraic groups, {\it Mem. Amer. Math. Soc.} {\bf 121} (1996), no. 580, vi+111pp.
		
		\bibitem{liebeck_seitz_2012unipotent} M.W. Liebeck and G.M. Seitz, {\it Unipotent and nilpotent classes in simple algebraic groups and Lie algebras}, Mathematical Surveys and Monographs, vol. 180, Amer. Math. Soc., Providence, RI, 2012.
		
		\bibitem{LSS} M.W. Liebeck, J. Saxl, and G.M. Seitz, Subgroups of maximal rank in finite exceptional groups of Lie type, {\it Proc. London Math. Soc.} {\bf 65} (1992), 297--325.
		
		\bibitem{Linton} S.A. Linton, The maximal subgroups of the Thompson group, {\it J. London Math. Soc.} {\bf 39} (1989), 
		79--88.
		
		\bibitem{LW24} S.A. Linton and R.A. Wilson, The maximal subgroups of the Fischer groups $Fi_{24}$ and $Fi_{24}'$,
		{\it Proc. London Math. Soc.} {\bf 63} (1991), 113--164.
		
		\bibitem{Litterick} A.J. Litterick, On non-generic finite subgroups of exceptional algebraic groups, {\it Mem. Amer. Math. Soc.} {\bf  253} (2018), no. 1207, v+156pp. 
		
		\bibitem{lubeck2001small} F. L\"ubeck, Small degree representations of finite Chevalley groups in defining characteristic, {\it LMS J. Comput. Math.} {\bf 4} (2001), 135--169.
		
		\bibitem{malle1991maximal2F4} G. Malle, The maximal subgroups of $^2\!F_4(q^2)$, {\it J. Algebra} {\bf 139} (1991), 52--69.
		
		\bibitem{malle_testerman_2011linear}  G. Malle and D. Testerman, {\it Linear algebraic groups and finite groups of Lie type}, Cambridge Studies in Advanced Mathematics, 133. Cambridge University Press, Cambridge, 2011.
		
		\bibitem{maroti2002orders} A. Mar\'oti, On the orders of primitive groups, {\it J. Algebra} {\bf 258} (2002), 631--640.
		
		\bibitem{mizuno} K. Mizuno, The conjugate classes of Chevalley groups of type $E_6$, {\it J. Fac. Sci. Univ. Tokyo} {\bf 24} (1977), 525--63.
		
		\bibitem{nagura1952interval} J. Nagura, On the interval containing at least one prime number,
		{\it Proc. Japan Academy} {\bf 28} (1952), 177--181.
		
		\bibitem{neumann1966thesis} P.M. Neumann, A study of some finite permutation groups, DPhil Thesis, University of Oxford, 1966.
		
		
		
		\bibitem{premet1988weights} A.A. Premet, Weights of infinitesimally irreducible representations of Chevalley groups over a field of prime characteristic, {\it Math. USSR Sb.}  {\bf 61} (1988), no. 1, 167--183.
		
		
		\bibitem{robbins1955remark} H. Robbins, A remark on Stirling's formula, {\it Amer. Math. Monthly}  {\bf 62} (1955), no. 1, 26--29.
		
		\bibitem{schaffer1999twisted} M. Schaffer, Twisted tensor product subgroups of finite classical groups, {\it Comm. Algebra} {\bf 27} (1999), 5097--5166.
		
		\bibitem{scott1977matrices} L.L. Scott, Matrices and cohomology, {\it Ann. of Math.} {\bf 105} (1977), 473--492.
		
		\bibitem{SS} D. Segal and A. Shalev, On groups with bounded conjugacy classes, {\it Quart. J. Math.} {\bf 50} (1999), 505--516.
		
		\bibitem{shinoda1974conjugacy} K. Shinoda, The conjugacy classes of Chevalley groups of type ($F_4$) over finite fields of characteristic 2, {\it J. Fac. Sci. Univ. Tokyo} {\bf 21} (1974), 133--159.
		
		\bibitem{shinoda1974conjugacy_ree} K. Shinoda, The conjugacy classes of the finite Ree groups of type ($F_4$), {\it J. Fac. Sci. Univ. Tokyo} {\bf 22} (1975), 1--15. 
		
		\bibitem{shoji1974conjugacy} T. Shoji, The conjugacy classes of Chevalley groups of type ($F_4$) over finite fields of characteristic $p\ne 2$, {\it J. Fac. Sci. Univ. Tokyo} {\bf 21} (1974), 1--17.
		
		\bibitem{shult1965groups} E.E. Shult, On groups admitting fixed point free abelian operator groups, {\it Illinois J. Math.} {\bf 9} (1965), 701--720.
		
		\bibitem{Spalt} N. Spaltenstein, Caracteres unipotents de $^3\!D_4(q)$, {\it Comment. Math. Helv.} {\bf 57} (1982), 676--
		691.

        \bibitem{lenstra} P. Stevenhagen and H.W. Lenstra, Chebotar{\"e}v and his density theorem, {\it The Mathematical Intelligencer} {\bf 18} (1996), 26--37.

		\bibitem{suzuki1962class} M. Suzuki, On a class of doubly transitive groups, {\it Ann. of Math.} {\bf 75} (1962), 105--145.
		
		\bibitem{tiep2008hall} P.H. Tiep and A.E. Zalesski, Hall-Higman type theorems for semisimple elements of finite classical groups, {\it Proc. Lond. Math. Soc.} {\bf 97} (2008) 623--668.

        \bibitem{WilCo1} R.A. Wilson, On the 3-local subgroups of Conway's group $Co_1$, {\it J. Algebra} {\bf 113} (1988), 
        261--262.
		
		\bibitem{WilBM} R.A. Wilson, The maximal subgroups of the Baby Monster, I, {\it J. Algebra} {\bf 211} (1999), 1--14.
		
		\bibitem{Zs} K. Zsigmondy, Zur Theorie der Potenzreste, {\it Monatshefte Math. Phys.} {\bf 3} (1892), 265--284. 
		
		
		
		
		
		
		
		
		
		
		
	\end{thebibliography}
	%	\bibliographystyle{alpha}

\end{document}